\documentclass[11pt, a4paper]{amsart}
\usepackage{pifont}
\usepackage{multirow}
\usepackage[T1]{fontenc}
\usepackage[english]{babel}
\usepackage[style=numeric, sorting=nyt]{biblatex} 
\usepackage{csquotes}
\usepackage[margin=2.5cm]{geometry}
\usepackage{enumerate}
\usepackage{float}
\usepackage[final]{graphicx}
\usepackage[normalem]{ulem}
\usepackage{cancel}

\usepackage{caption}
\usepackage{subcaption}

\DeclareCaptionSubType[alph]{figure}
\usepackage{hyperref}
\usepackage{xurl}
\hypersetup{
    breaklinks=true,
    colorlinks=true,
    linkcolor=blue,
    filecolor=magenta,      
    urlcolor=cyan,
    }
\usepackage{enumitem}
\usepackage{comment}
\usepackage{soul}
\usepackage{mathrsfs}

\usepackage{amsmath}
\usepackage{amsfonts}
\usepackage{amssymb}
\usepackage{amsthm}
\usepackage{tikz-cd}

\usepackage{tikz}
\usetikzlibrary{decorations.markings}
\usetikzlibrary{decorations.pathreplacing}
\usetikzlibrary{arrows,shapes,positioning,patterns}
\usetikzlibrary{knots}
\tikzstyle{none}=[inner sep=0pt]
\pgfdeclarelayer{edgelayer}
\pgfdeclarelayer{nodelayer}
\pgfsetlayers{edgelayer,nodelayer,main}
\usepackage{xcolor}

\newcommand{\R}{\mathbb R}

\newcommand{\Z}{\mathbb Z} 
\newcommand{\F}{\mathscr F}
\newcommand{\K}{\mathscr K}
\newcommand{\U}{\mathscr U}

\renewcommand{\L}{\mathscr L}

\newcommand{\A}{\mathscr A}

\newcommand{\ep}{\varepsilon}
\newcommand{\si}{\sigma}

\newcommand{\del}{\partial}

\DeclareMathOperator{\tb}{tb}

\DeclareMathOperator{\rot}{rot}

\newcommand{\co}{\thinspace\colon}

\newtheorem{thm}{Theorem}[section]
\newtheorem{prop}[thm]{Proposition}
\newtheorem{lem}[thm]{Lemma}
\newtheorem{cor}[thm]{Corollary}

\newtheorem{defn}[thm]{Definition}
\newtheorem{con}[thm]{Conjecture} 

\theoremstyle{definition}
\newtheorem{rem}[thm]{Remark}

\definecolor{nicegreen}{rgb}{0.0, 0.5, 0.0}
\definecolor{niceyellow}{rgb}{1.0, 0.65, 0.0}
\definecolor{niceblue}{rgb}{0.0, 0.28, 0.67}

\newcommand{\Addresses}{{
		\bigskip
		\footnotesize
		
		\textsc{Dipartimento di Matematica, Largo Bruno Pontecorvo 5, 56127 Pisa, Italy}\par\nopagebreak
}}

\title{Strongly Invertible Legendrian Links}
\author{Carlo Collari and Paolo Lisca}
\date{\today}

\begin{document}

\begin{abstract}
We introduce and study strongly invertible Legendrian links in the standard contact three-dimensional space. 
We establish the equivariant analogs of basic results separately well-known for strongly invertible and Legendrian links, i.e.~the existence of transvergent front diagrams, an equivariant Legendrian Reidemeister theorem, and an equivariant stabilization theorem à la Fuch-Tabachnikov. 
We also introduce a maximal equivariant Thurston-Bennequin number for strongly invertible links and we exhibit infinitely many such links for which the invariant coincides with the usual maximal Thurston-Bennequin number. 
We conjecture that such a coincidence does not hold general and that there exist strongly invertible knots having Legendrian representatives isotopic to their reversed Legendrian mirrors but not isotopic to any strongly invertible Legendrian knot.
\end{abstract}

\maketitle
\section{Introduction}\label{s:intro} 

Strongly invertible knots and links have been studied for a long time~\cite{Sa86} and they have seen a recent burst of interest~\cite{AB21, BI22, DMS23, DP23-1, DP23-2, DPF23, HHS22, La22, LS22, LW21, MP23, Sn18, Wa17}.
On the other hand, the literature on Legendrian knots and links is quite vast and continuously expanding -- we refer the reader to the surveys~\cite{Et05, EN20}.  
The purpose of this paper is to initiate the study of the {\em strongly invertible Legendrian links} in $\R^3$. 
We start by recalling the definitions of a strongly invertible link and of a Legendrian link.  
A link $L\subset\R^3$ is a finite collection of smoothly embedded circles. 
Let $\tau\thinspace\colon\R^3\to\R^3$ be the involution given by 
\[
\tau(x,y,z) = (x,-y,-z). 
\]
We say that $L$ is {\em strongly invertible} if $\tau(L_i)=L_i$ for each connected component $L_i\subset L$, and $\tau|_{L_i}$ has exactly two fixed points. Two strongly invertible links are {\em equivalent} if there is a family of strongly invertible links that connects them. 
The \emph{standard contact structure} on $\R^3$ is the tangent plane distribution $\xi_{\rm st}:=\ker(dz - ydx)\subset\R^3$.
Note that $\xi_{\rm st}$ is preserved by $\tau$.
A link is \emph{Legendrian} if it is everywhere tangent to $\xi_{\rm st}$. Two Legendrian links are {\em equivalent} if there is a family of Legendrian links that connects them.  In analogy with the above, we introduce the following definition. 

\begin{defn}\label{d:silk}
A {\em strongly invertible Legendrian link} is a link that is simultaneously strongly invertible with respect to $\tau$ and Legendrian with respect to $\xi_{\rm st}$. 
Two strongly invertible Legendrian links are {\em equivalent} if they are connected by a smooth family of strongly invertible Legendrian~links. 
\end{defn}

Following~\cite{BI22}, we say that a  diagram $D$ in the $xz$-plane is~\emph{transvergent} if the reflection map~$(x,z)\mapsto (x,-z)$ fixes $D$ setwise.

Denote by $\pi$ the projection map to the $xz$-plane. 
The image $\F:=\pi(\L)$ of a Legendrian link $\L$ is called the {\em front} of $\L$. 
It is a well-known and easy-to-show fact that, given a front $\F$, there is a unique Legendrian link $\L$ such that $\pi(\L)=\F$ (see e.g.~\cite{Et05}).  
After a small, $C^\infty$-small Legendrian perturbation, a Legendrian link admits a generic front~\cite{Sw92}. 
The diagram associated with a generic front will be called a {\em front diagram}. 
As before, we introduce the following definition in analogy to the smooth case. 

\begin{defn}
A front $\F$ of a strongly invertible Legendrian link is a~\emph{transvergent front} if 
the reflection $(x,z)\mapsto (x,-z)$ fixes $\F$ setwise.
\end{defn}

Two among the simplest examples of transvergent front diagrams are shown in Figure~\ref{f:unknots}. More complicated examples are shown in Figures~\ref{f:SILtorusK}, \ref{f:SILtwistK}, and~\ref{f:m9_44-pair}.  

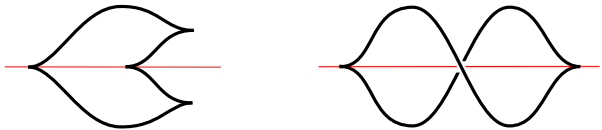
\begin{figure}[ht]
\centering
\begin{minipage}[c]{4cm}
\begin{tikzpicture}[very thick,scale = 1.2]
  \path[draw=red,line cap=butt,line join=miter,line width=0.010cm]
    (7.2231,-4.4067) -- (9.5914,-4.4095);
  \draw[very thick]
    (7.4842,-4.4080) .. controls (7.7872, -4.4198) and (8.0497, -3.7507) ..
    (8.4827, -3.7397) .. controls (8.9157, -3.7286) and (8.9734, -4.0064) ..
    (9.2965, -4.0042);
  \draw[very thick]
    (7.4789,-4.4088) .. controls (7.7820, -4.3970) and (8.0515, -5.0586) ..
    (8.4845, -5.0697) .. controls (8.9175, -5.0807) and (9.0360, -4.8059) ..
    (9.2759, -4.8038);
  \draw[very thick]
    (8.5467,-4.4049) .. controls (8.8794, -4.4044) and (8.8940, -3.9974) ..
    (9.2945, -4.0035);
  \draw[very thick]
    (8.5481,-4.4055) .. controls (8.8808, -4.4060) and (8.8821, -4.8112) ..
    (9.2826, -4.8051);
\end{tikzpicture}
\end{minipage}
\begin{minipage}[c]{4cm}
\begin{tikzpicture}[very thick,scale=1.2]
  \draw[very thick]
    (4.9433,-1.5746) .. controls (5.3414, -1.5720) and (5.2489, -2.2292) ..
    (5.7444, -2.2332) .. controls (6.1539, -2.2364) and (6.3962, -0.9288) ..
    (6.8018, -0.9218) .. controls (7.1993, -0.9150) and (7.2153, -1.5752) ..
    (7.5838, -1.5786);
  \path[draw=white,line cap=butt,line join=miter,line width=0.150cm]
    (6.5362,-1.5812) -- (6.0054,-1.5810);
  \path[draw=red,line cap=butt,line join=miter,line width=0.010cm]
    (4.7148,-1.5792) -- (7.7956,-1.5752);
  \path[draw=white,line cap=butt,line join=miter,line width=0.10cm]
    (6.4113,-1.8412) -- (6.1383,-1.3059);
  \draw[very thick]
    (4.9433,-1.5804) .. controls (5.3414, -1.5829) and (5.2489, -0.9257) ..
    (5.7444, -0.9218) .. controls (6.1539, -0.9185) and (6.3962, -2.2262) ..
    (6.8018, -2.2331) .. controls (7.1993, -2.2400) and (7.2153, -1.5797) ..
    (7.5838, -1.5763);
\end{tikzpicture}
\end{minipage} \caption{Examples of transvergent front diagrams}
\label{f:unknots}
\end{figure}

Recall that every strongly invertible link admits a transvergent diagram (see Section~\ref{s:existence}).  
Therefore, it is natural to expect is that each strongly invertible link is equivalent to the strongly invertible Legendrian link associated with a transvergent front diagram. 
Proposition~\ref{p:legrep} below shows that this indeed holds. 

Front diagrams of equivalent Legendrian links are connected by a sequence of Legendrian Reidemeister moves ($LR$-moves, for short)~\cite{Sw92}. 
Similarly, due to a theorem of Lobb and Watson~\cite{LW21}, transvergent diagrams of strongly invertible links (and, more generally, strongly involutive links) are connected by a sequence of equivariant moves that we call {\em Lobb-Watson moves}~\cite[Theorem~2.3]{LW21}. 
The following is an analog for strongly invertible Legendrian links. 

\begin{thm}\label{t:SILRT}
Let $\F_0$ and $\F_1$ be transvergent front diagrams of strongly invertible Legendrian links $\L_0$ and, respectively, $\L_1$.
Then, $\L_0$ and $\L_1$ are equivalent if and only if $\F_0$ and $\F_1$ are connected by a finite sequence of the following: 
\begin{itemize}
\item 
equivariant planar isotopies; 
\item 
pairs of $LR$-moves applied symmetrically with respect to the $x$-axis;
\item 
moves from the list of Figure~\ref{f:moves};
\item 
moves obtained from those of Figure~\ref{f:moves} by a $\pi$-rotation of the page. 
\end{itemize}
\begin{figure}[ht]
\centering
\begin{tikzpicture}[very thick, scale = 0.8]
\draw (-1,-1).. controls +(2,0) and +(-1,0) .. (1,1);
\draw[white, line width = 4] (-1,1).. controls +(2,0) and +(-1,0) .. (1,-1);
\draw[red, very thin] (-1,0) -- (1,0);
\draw[white, line width = 2] (-1,1).. controls +(2,0) and +(-1,0) .. (1,-1);
\draw (-1,1).. controls +(2,0) and +(-1,0) .. (1,-1);
\draw (-1,.75) .. controls +(.25,0) and +(-.5,0) .. (0,0);
\draw (-1,-.75) .. controls +(.25,0) and +(-.5,0) .. (0,0);

\node at (2,0.5) {CX};
\draw[latex-latex, thick] (1.25,0) -- (2.75,0);

\begin{scope}[shift ={+(4,0)}]
\draw (-1,-1).. controls +(1,0) and +(-2,0) .. (1,1);
\draw[white, line width = 4] (-1,1).. controls +(1,0) and +(-2,0) .. (1,-1);
\draw[white, line width = 4] (-1,.75) .. controls +(.5,0) and +(-.5,0) .. (.5,0);
\draw[white, line width = 4] (-1,-.75) .. controls +(.5,0) and +(-.5,0) .. (.5,0);
\draw[red, very thin] (-1,0) -- (1,0);
\draw (-1,.75) .. controls +(.5,0) and +(-.5,0) .. (.5,0);
\draw (-1,-.75) .. controls +(.5,0) and +(-.5,0) .. (.5,0);
\draw[white, line width = 2] (-1,1).. controls +(1,0) and +(-2,0) .. (1,-1);
\draw (-1,1).. controls +(1,0) and +(-2,0) .. (1,-1);
\end{scope}

\begin{scope}[shift ={+(8,0)}]
\draw (-1,-1).. controls +(2.25,0) and +(-1.25,0) .. (1,1);
\draw[white, line width = 4] (-1,-.75).. controls +(.75,0) and +(-2,0) .. (1,.75);
\draw (-1,-.75).. controls +(.75,0) and +(-2,0) .. (1,.75);
\draw[white, line width = 4] (-1,.75).. controls +(.75,0) and +(-2,0) .. (1,-.75);
\draw[red, very thin] (-1,0) -- (1,0);
\draw[white, line width = 2] (-1,.75).. controls +(.75,0) and +(-2,0) .. (1,-.75);
\draw (-1,.75).. controls +(.75,0) and +(-2,0) .. (1,-.75);
\draw[white, line width = 4] (-1,1).. controls +(2.25,0) and +(-1.25,0) .. (1,-1);
\draw (-1,1).. controls +(2.25,0) and +(-1.25,0) .. (1,-1);

\node at (2,0.5) {XX};
\draw[latex-latex, thick] (1.25,0) -- (2.75,0);

\begin{scope}[shift ={+(4,0)}]
\draw (-1,-1).. controls +(1.25,0) and +(-2.25,0) .. (1,1);
\draw[white, line width = 4] (-1,-.75).. controls +(2,0) and +(-.75,0) .. (1,.75);
\draw (-1,-.75).. controls +(2,0) and +(-.75,0) .. (1,.75);
\draw[white, line width = 4] (-1,.75).. controls +(2,0) and +(-.75,0) .. (1,-.75);
\draw[red, very thin] (-1,0) -- (1,0);
\draw[white, line width = 2] (-1,.75).. controls +(2,0) and +(-.75,0) .. (1,-.75);
\draw (-1,.75).. controls +(2,0) and +(-.75,0) .. (1,-.75);
\draw[white, line width = 4] (-1,1).. controls +(1.25,0) and +(-2.25,0) .. (1,-1);
\draw (-1,1).. controls +(1.25,0) and +(-2.25,0) .. (1,-1);
\end{scope}
\end{scope}

\begin{scope}[shift ={+(0,-3)}, yscale = .8]
\draw (-1,-1.5).. controls +(.75,0) and +(-0.25,-.15) .. (.35,-.35);
\draw (-1,-1).. controls +(.75,0) and +(-0.25,-.15) .. (.35,-.35);
\draw[red, very thin] (-1,0) -- (1,0);
\draw (-1,1.5).. controls +(.75,0) and +(-0.25,.15) .. (.35,.35);
\draw (-1,1).. controls +(.75,0) and +(-0.25,.15) .. (.35,.35);
\node at (2,0.5) {CC};
\draw[latex-latex, thick] (1.25,0) -- (2.75,0);

\begin{scope}[shift ={+(4,0)}, yscale = .8]
\draw (-1,-1.5).. controls +(1,0) and +(-0.35,-.15) .. (.75,.75);
\draw (-1,-1).. controls +(.25,0) and +(-0.35,-.15) .. (.75,.75);
\draw[white, line width = 4] (-1,1.5).. controls +(1,0) and +(-0.35,.15) .. (.75,-.75);
\draw[white, line width = 4] (-1,1).. controls +(.25,0) and +(-0.35,.15) .. (.75,-.75);
\draw[red, very thin] (-1,0) -- (1,0);
\draw[white, line width = 2] (-1,1.5).. controls +(1,0) and +(-0.35,.15) .. (.75,-.75);
\draw[white, line width = 2] (-1,1).. controls +(.25,0) and +(-0.35,.15) .. (.75,-.75);
\draw (-1,1.5).. controls +(1,0) and +(-0.35,.15) .. (.75,-.75);
\draw (-1,1).. controls +(.25,0) and +(-0.35,.15) .. (.75,-.75);
\end{scope}

\begin{scope}[shift ={+(7,0)}, xscale = .75]
\draw[very thick] (-.25,-1) .. controls +(.75,.5) and +(-.3,0) .. (3,1);
\draw[white , line width = 3] (0,0) -- (3,0);
\draw[very thin, red] (0,0) -- (3,0);
\draw[white , line width = 4.5] (-.25,1) .. controls +(.75,-.5) and +(-.3,0) .. (2.5,0);
\draw[very thick] (-.25,-1) .. controls +(.75,.5) and +(-.3,0) .. (2.5,0);
\draw[white , line width = 4.5] (-.25,1) .. controls +(.75,-.5) and +(-.3,0) .. (3,-1);
\draw[very thick] (-.25,1) .. controls +(.75,-.5) and +(-.3,0) .. (3,-1);
\draw[very thick] (-.25,1) .. controls +(.75,-.5) and +(-.3,0) .. (2.5,0);
\node at (4.15,0.5) {CR};
\draw[latex-latex, thick] (3.25,0) -- (5.15,0);
\end{scope}

\begin{scope}[shift ={+(11,0)}, xscale = .75]
\draw[very thin, red] (0,0) -- (3,0);
\draw[very thick] (3,1) .. controls +(-.5,0) and +(.5,0) .. (1,0);
\draw[very thick] (3,-1) .. controls +(-.5,0) and +(.5,0) .. (1,0);
\end{scope}
\end{scope}
\end{tikzpicture}
 \caption{Equivariant Legendrian Reidemeister moves.}
\label{f:moves}
\end{figure}
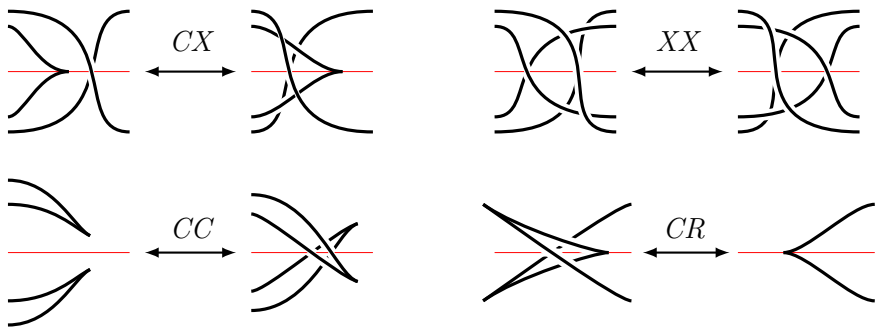
\end{thm}

Recall that Legendrian links can be {\em stabilized} (see~\cite[\S 4.2]{FT97} and~\cite[\S 2.7]{Et05}). 
It is a well-known result of Fuchs and Tabacnhikov that after sufficiently many stabilizations, two (oriented) Legendrian knots in $(\R^3,\xi_{\rm std})$ with the same topological type become Legendrian isotopic~\cite[Theorem~4.4]{FT97}. 
In Section~\ref{s:stab} we show that an {\em unoriented} strongly invertible Legendrian link $\L$ with the choice of a point $p = \tau(p)\in \L$ admits two different types of stabilizations. In particular, a strongly invertible Legendrian front $\K$ with a distinguished fixed point $p\in\K$ admits two types of stabilizations that we call $S(\K,p)$ and $T(\K,p)$. 
For example, Figure~\ref{f:unknots} contains the diagrams of $\U_S := S(\U,c_R)$ (on the left) and $\U_T := T(\U,c_R)$ (on the right), where $\U$ is the strongly invertible Legendrian unknot given by the diagram of Figure~\ref{f:unknot} and $c_R$ is the fixed point corresponding to the right cusp of the diagram.
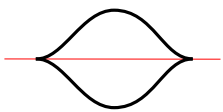
\begin{figure}[ht]
\centering
\begin{tikzpicture}[very thick,scale = 1.2]
\path[draw=red,line cap=butt,line join=miter,line width=0.010cm]
    (7.2281,-4.6268) -- (9.5964,-4.6296);
\draw[very thick]
    (7.5721,-4.6268) .. controls (7.8751, -4.6386) and (8.0768, -4.0868) ..
    (8.4318, -4.0883) .. controls (8.8649, -4.0903) and (8.9688, -4.6290) ..
    (9.2919, -4.6268);
\draw[very thick]
    (7.5721,-4.6268) .. controls (7.8751, -4.6150) and (8.0768, -5.1668) ..
    (8.4318, -5.1653) .. controls (8.8649, -5.1634) and (8.9688, -4.6246) ..
    (9.2919, -4.6268);
\end{tikzpicture} \caption{Transvergent front diagram of $\U$}
\label{f:unknot}
\end{figure}~\\
The Thurston-Bennequin number is an integer-valued invariant of Legenrdian knots which can be computed directly from a front \cite{Et05}. 
Observe that, since the Legendrian knots in Figure~\ref{f:unknots} satisfy $\tb(\U_S) = \tb(\U_T) = -2$, by the classification of the Legendrian unknots~\cite{EF09}, it follows that $\U_S$ and $\U_T$ are Legendrian isotopic as unoriented Legendrian knots. Moreover, when the two diagrams in Figure~\ref{f:unknots} are regarded as smooth diagrams, they represent equivalent strongly invertible knots, cf.~\cite{LW21}. 
On the other hand, we show in Section~\ref{s:stab} -- see Corollary~\ref{c:notequ} -- that the unoriented strongly invertible Legendrian unknots $\U_S$ and $\U_T$ are {\em not} equivalent. 

In Section~\ref{s:stabequiv} we apply Theorem~\ref{t:SILRT} to obtain the following result, which is an analog of the Fuchs and Tabachnikov's result in the present setting. 

\begin{thm}\label{t:stabs}
Let $\L$ and $\L'$ be strongly invertible Legendrian links which are equivalent as strongly invertible links.  
Then, after sufficiently many $S$-stabilizations, $\L$ and $\L'$ become equivalent strongly invertible Legendrian links. 
\end{thm}

\begin{rem}\label{r:unoriented}
To keep things simple in this paper we mainly consider strongly invertible {\em unoriented} Legendrian links.
As in the non-equivariant case, one could introduce \emph{oriented} strongly invertible Legendrian links and the corresponding equivalence between them. 
The natural analogs of Theorems~\ref{t:SILRT} and~\ref{t:stabs} would still hold. 
\end{rem}

It is natural to wonder whether the maximal Thurston-Bennequin number of a strongly invertible knot is always realized by a strongly invertible Legendrian knot. This leads to the following 

\begin{defn}\label{d:maxeqtb}
Let $K\subset\R^3$ be a strongly invertible knot. 
We define the {\em maximal equivariant Thurston-Bennequin number} of $K$, denoted $\overline\tb_e(K)$, as the maximal Thurston-Bennequin number of a strongly invertible Legendrian knot representing $K$. 
\end{defn}

Note that the maximal equivariant Thurston-Bennequin number can be refined to yield an invariant of strongly invertible knots with a fixed involution, by considering the maximal Thurston-Bennequin number achieved within the appropriate equivalence class of strongly invertible knots. 
However, in this paper we do not consider such a refinement.

Denote by~$\overline{\tb}(K)$ the maximal Thurston-Bennequin number among all Legendrian representatives of $K$. 
For any knot $K$ the inequality $\overline{\tb}_e(K) \leq \overline{\tb}(K)$ clearly holds, so it is natural to ask whether it is possibile to have a strict inequality. 

In Section~\ref{s:maxeqtb} we prove the following 

\begin{prop}\label{p:torus-twist}
If $K$ is either a torus knot of type $T(2,2n+1)$, $n\in\Z$, or a twist knot. 
Then, we have $\overline{\tb}_e(K) = \overline{\tb}(K)$.
\end{prop}

On the other hand, in Section~\ref{s:expconj} we present some experimental results which suggest that the equality $\overline{\tb}_e = \overline{\tb}$ should not hold in general. 

\begin{rem}
One may define strong invertibility with respect to the involution $\si\co\R^3\to\R^3$ given by $\si(x,y,z) = (-x,-y,z)$. 
Note that both the involutions $\tau$ and $\si$ preserve $\xi_{\rm st}$, but $\tau$ reverses the orientation of the contact planes, while $\si$ preserves it. 
Moreover, while the fixed-point set $F_\tau$ is tangent to $\xi_{\rm std}$, the fixed-point set $F_\si$ -- which coincides with the $z$-axis -- is {\em transverse} to $\xi_{\rm std}$. 
We do not investigate strong $\si$-invertibility in this paper. 
\end{rem}
The paper is organized as follows. In Section~\ref{s:existence} we establish~Proposition~\ref{p:legrep} and in Section~\ref{s:LeRt} we prove Theorem~\ref{t:SILRT}. 
In Section~\ref{s:stab} we define the stabilizations of strongly invertible Legendrian links and show that the strongly invertible Legendrian unknots of Figure~\ref{f:unknots} are not equivalent. 
In Section~\ref{s:stabequiv} we prove Theorem~\ref{t:stabs}. 
In Section~\ref{s:maxeqtb} we show that the equivariant maximal Thurston-Bennequin number coincides with the maximal Thurston-Bennequin number for certain families of knots. In Section~\ref{s:expconj} we describe some experimental results and state three conjectures.  

\subsection*{Acknowledgements} 
The authors are grateful to Lenhard Ng for helpful e-mail messages. 

\section{Existence of strongly invertible Legendrian representatives}\label{s:existence}

In this section, we prove Proposition~\ref{p:legrep}. The following well-known proposition (cf.~\cite{LW21}) shows that, given a strongly invertible link $L\subset\R^3$, up to equivalence one can always assume that the projection $\pi(L)$ on the $xz$-plane is generic and therefore yields a transvergent diagram of $L$. 
We include a proof for the sake of completeness and because in Section~\ref{s:LeRt}, we use a similar approach. 

\begin{prop}\label{p:intr-diagrs}
Each strongly invertible link $L$ is equivalent to a strongly invertible link $\tilde L$ with a transvergent diagram $D$. 
There exists $D$ such that each crossing involves two strands meeting at a $\pi/2$-angle, $D$ intersects transversely the $x$-axis, and the restriction to $D$ of the coordinate-$x$ function has a finite number of critical points. 
Moreover, each non-crossing point of the intersection of $D$ with the $x$-axis has a vertical tangent line. 
\end{prop} 

\begin{proof}
The link $L$ can be decomposed as 
\[
L = A \cup \tau(A), 
\]
where $A\subset\R^3$ is a finite union of smoothly embedded arcs such that $A\cap\tau(A)$ consists of the fixed points of $\tau|_L$. 
By perturbing slightly $A$ away from its endpoints, we obtain a strongly invertible link 
\[
\tilde L = \tilde A \cup \tau(\tilde A)
\]
which is $\tau$-equivariantly isotopic to $L$, and a corresponding strongly invertible link $\tilde L$. 
Now, $\pi(\tilde L)$ is generic and thus yields a transvergent diagram for $L$. General position ensures the stated properties of $D$. 
\end{proof}

We call ``bad'' a diagram crossing whose over-strand has a higher slope than the under-strand. 

\begin{prop}\label{p:legrep}
Each strongly invertible link $L\subset\R^3$ is equivalent to the strongly invertible Legendrian link associated with a transvergent front diagram $\F$. More precisely, given a transvergent diagram $D$ for $L$, $\F$ is obtained by $\tau$-equivariantly replacing each point of $D$ having a vertical tangent with a cusp and each ``bad'' crossing with either one of the configurations illustrated in Figure~\ref{f:crossings}. 
\begin{figure}[ht]
\begin{tikzpicture}[very thick, scale = 0.8]
\draw (-1,1) -- (1,-1);
\draw[line width = 8, white] (1,1) -- (-1,-1);
\draw[red, thin] (-1.4,0) -- (1.4,0);
\draw[line width = 3, white] (1,1) -- (-1,-1);
\draw (-1,-1) -- (1,1);

\draw[-latex, line width = 1] (2,0) -- (3,0);
\node at (7,0) {or};
\begin{scope}[shift = {(5,0)}]
\draw (-1,1) 
.. controls +(.5,-.5) and +(-.5,0) .. (.4,.75) 
.. controls +(-.5,0) and +(.5,0) .. (-.6,-.35) 
.. controls +(1,0) and +(-.5,.5) .. (1,-1);
\draw[line width = 5, white] (-1,-1) 
.. controls +(.5,.5) and +(-.5,0) .. (.4,-.75) 
.. controls +(-.5,0) and +(.5,0) .. (-.6,.35) 
.. controls +(1,0) and +(-.5,-.5) .. (1,1);
\draw[red, thin] (-1.4,0) -- (1.4,0);
\draw[line width = 3, white] (-1,-1) 
.. controls +(.5,.5) and +(-.5,0) .. (.4,-.75) 
.. controls +(-.5,0) and +(.5,0) .. (-.6,.35) 
.. controls +(1,0) and +(-.5,-.5) .. (1,1);
\draw (-1,-1) 
.. controls +(.5,.5) and +(-.5,0) .. (.4,-.75) 
.. controls +(-.5,0) and +(.5,0) .. (-.6,.35) 
.. controls +(1,0) and +(-.5,-.5) .. (1,1);
\end{scope}
\begin{scope}[shift = {(9,0)}, scale = -1]
\draw (-1,1) 
.. controls +(.5,-.5) and +(-.5,0) .. (.4,.75) 
.. controls +(-.5,0) and +(.5,0) .. (-.6,-.35) 
.. controls +(1,0) and +(-.5,.5) .. (1,-1);
\draw[line width = 5, white] (-1,-1) 
.. controls +(.5,.5) and +(-.5,0) .. (.4,-.75) 
.. controls +(-.5,0) and +(.5,0) .. (-.6,.35) 
.. controls +(1,0) and +(-.5,-.5) .. (1,1);
\draw[red, thin] (-1.4,0) -- (1.4,0);
\draw[line width = 3, white] (-1,-1) 
.. controls +(.5,.5) and +(-.5,0) .. (.4,-.75) 
.. controls +(-.5,0) and +(.5,0) .. (-.6,.35) 
.. controls +(1,0) and +(-.5,-.5) .. (1,1);
\draw (-1,-1) 
.. controls +(.5,.5) and +(-.5,0) .. (.4,-.75) 
.. controls +(-.5,0) and +(.5,0) .. (-.6,.35) 
.. controls +(1,0) and +(-.5,-.5) .. (1,1);
\end{scope}
\end{tikzpicture} 
~\\
\vspace{0.1cm}
~\\
\begin{tikzpicture}[very thick, scale = 0.8]
\draw (-1,1) -- (1,-1);
\draw[line width = 8, white] (1,1) -- (-1,-1);
\draw[line width = 3, white] (1,1) -- (-1,-1);
\draw (-1,-1) -- (1,1);

\draw[-latex, line width = 1] (2,0) -- (3,0);
\node at (7,0) {or};
\begin{scope}[shift = {(5,0)}]
\draw (-1,1) -- (1,-1);
\draw[line width = 5, white] (-1,-1) 
.. controls +(.5,.5) and +(-.5,0) .. (.5,-.75) 
.. controls +(-.5,0) and +(.5,0) .. (-.5,.75) 
.. controls +(1,0) and +(-.5,-.5) .. (1,1);
\draw (-1,-1) 
.. controls +(.5,.5) and +(-.5,0) .. (.5,-.75) 
.. controls +(-.5,0) and +(.5,0) .. (-.5,.75) 
.. controls +(1,0) and +(-.5,-.5) .. (1,1);
\end{scope}
\begin{scope}[shift = {(9,0)}, yscale =-1]
\draw (-1,-1) 
.. controls +(.5,.5) and +(-.5,0) .. (.5,-.75) 
.. controls +(-.5,0) and +(.5,0) .. (-.5,.75) 
.. controls +(1,0) and +(-.5,-.5) .. (1,1);
\draw[line width = 5, white] (-1,1) -- (1,-1);
\draw (-1,1) -- (1,-1);
\end{scope}
\end{tikzpicture} \caption{The two possible modifications near a ``bad" crossing to turn a transvergent diagram into a transvergent front diagram.}
\label{f:crossings}
\end{figure}
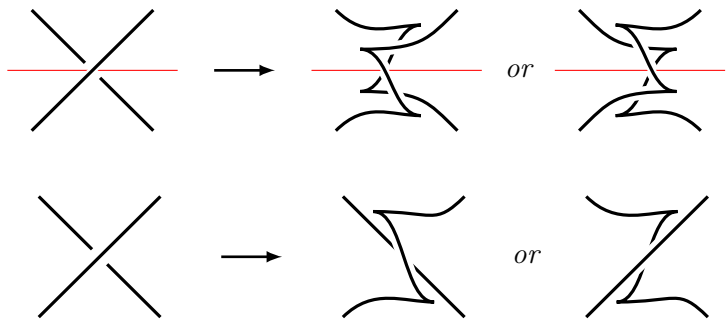
\end{prop}

\begin{proof}
Let $L\subset\R^3$ be a strongly invertible link and $D$ a transvergent diagram of $L$ as in Proposition~\ref{p:intr-diagrs}. 
Arguing as in the non-equivariant case (cf.~\cite[Section~2.3]{Et05}) we modify $D$ into the transvergent front diagram of a strongly invertible Legendrian link equivalent to $L$ as follows. 
We replace a neighborhood of each bad crossing on the $x$-axis with either one of the configurations at the top right of Figure~\ref{f:crossings}. 
We replace a neighborhood of each bad crossing off the $x$-axis with either one of the configurations at the bottom right of Figure~\ref{f:crossings} while replacing at the same time the $\tau$-image of the same neighborhood with the other configuration. 
We claim that the result of the above procedure is the transvergent front diagram of a Legendrian link equivalent to $L$ as a strongly invertible link. 
Replacements of points having vertical tangents with cusps and the $\tau$-equivariant local modifications around crossings off the $x$-axis do not modify the equivalence class of $L$ as a strongly invertible link. 
On the other hand, it is easy to check that the local modifications around bad crossings on the $x$-axis amount to applications of equivariant Reidemester moves of type M2 from Figure~9 of~\cite{LW21} (beware that in~\cite{LW21} the axis of symmetry is drawn vertically). 
Therefore, our claim follows from (the easy direction of)~\cite[Theorem~2.3]{LW21} and the proof is complete. 
\end{proof}

We illustrate Proposition~\ref{p:legrep} with the following example. A transvergent diagram $D$ of a strongly invertible unknot $U$ is shown on the left-hand side of Figure~\ref{f:ex1}. 
\begin{figure}[ht]
\centering 
\begin{minipage}[c]{5cm}
\begin{tikzpicture}[very thick, scale = 0.8]
\begin{scope}[scale = 1.2]
\draw[very thick] 
    (7.0115,-6.8921) .. controls (7.0115, -6.2306) and (7.9375,
    -5.8337) .. (8.3393, -6.8892) .. controls (8.7280, -7.9102) and (9.6573,
    -7.5535) .. (9.6573, -6.8921);
\path[draw=white,line cap=butt,line join=miter,line width=0.2cm,miter
    limit=4.00] (8.1756,-6.8921) .. controls (8.2550, -6.8921) and (8.4402,
    -6.8921) .. (8.4931, -6.8921);
\path[draw=red,nonzero rule,line cap=butt,line join=miter,line
    width=0.010cm,miter limit=4.00] (6.7469,-6.8921) -- (9.9219,-6.8921);
\path[draw=white,line cap=butt,line join=miter,line width=0.2cm,miter
    limit=4.00] (8.2782,-7.0252) .. controls (8.3250, -6.9114) and (8.3686,
    -6.7896) .. (8.3967, -6.7328);
\draw[very thick] 
    (7.0115,-6.8921) .. controls (7.0115, -7.6858) and (8.0698,
    -7.6858) .. (8.3344, -6.8921) .. controls (8.6793, -5.8571) and (9.6573,
    -6.2306) .. (9.6573, -6.8921);
\end{scope}    
\end{tikzpicture}
\end{minipage}
\begin{minipage}[c]{5cm}
\centering 
\begin{tikzpicture}[very thick, scale = 0.8]
\begin{scope}[scale = -1]
\draw(-2.5,0) 
.. controls +(.75,0) and +(-.5,0) .. (.4,1) 
.. controls +(-.5,0) and +(.5,0) .. (-.6,-.35)
.. controls +(1,0) and +(-.5,0) .. (.85,-.75) .. controls +(.5,0) and +(-.5,0) .. (2.5,0);
\draw[line width = 5, white] (.4,-1) 
.. controls +(-.5,0) and +(.5,0) .. (-.6,.35) 
.. controls +(1,0) and +(-.5,0) .. (.85,.75);
\draw[red, thin] (-3,0) -- (3,0);
\draw[line width = 3, white] (.4,-1) 
.. controls +(-.5,0) and +(.5,0) .. (-.6,.35) 
.. controls +(1,0) and +(-.5,0) .. (.85,.75);
\draw (-2.5,0) 
.. controls +(.75,0) and +(-.5,0) .. (.4,-1) 
.. controls +(-.5,0) and +(.5,0) .. (-.6,.35) 
.. controls +(1,0) and +(-.5,0) .. (.85,.75) 
.. controls +(.5,0) and +(-.5,0) .. (2.5,0);
\end{scope}
\end{tikzpicture}
\end{minipage} \caption{Construction of a transvergent front diagram}
\label{f:ex1}
\end{figure}
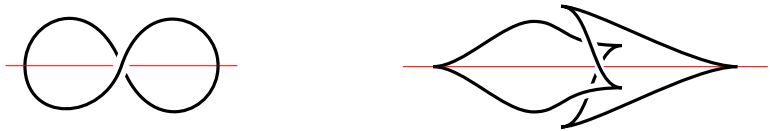
We use the procedure described in the proof of Proposition~\ref{p:legrep},
replacing a neighborhood of the crossing on the $x$-axis with the rightmost picture at the top of Figure~\ref{f:crossings}.
The result is the transvergent front diagram $\F$ shown on the right-hand side of Figure~\ref{f:ex1}.  

\section{The Legendrian equivariant Reidemeister theorem}\label{s:LeRt}

In this section, we prove Theorem~\ref{t:SILRT}, which is an analog of the Legendrian Reidemeister theorem in the equivariant case. 
We also illustrate the equivariant Legendrian Reidemeister moves by show that the diagram on the right of Figure~\ref{f:ex1} and the diagram on the left of Figure~\ref{f:unknots} represent equivalent strongly invertible Legendrian unknots. 

Furthermore, we describe a local modification that can be used to produce Legendrian isotopic strongly invertible Legendrian links which are equivalent as strongly invertible links, but inequivalent as strongly invertible Legendrian links.

Note that the fixed-point set $F_\tau$ of $\tau$ coincides with the $x$-axis, which is tangent to $\xi_{\rm std}$ and naturally oriented. We shall always assume $F_\tau$ to have this natural orientation. 

\begin{proof}[Proof of Theorem~\ref{t:SILRT}]
It is easy to see that if ${\F}_0$ and ${\F}_1$ are related by any of the moves mentioned in the statement then $\L_0$ and $\L_1$ are equivalent. 
Conversely, let $\{\L_t\}_{t\in [0,1]}$ be a family of strongly invertible Legendrian links.
In general, if $t\notin \{0,1\}$ the Legendrian link $\L_t$ may not have a front diagram because the front $\pi(\L_t)$ could be non-generic.
But we claim that by applying a $C^\infty$-small and $\tau$-equivariant Legendrian perturbation to the family, one can make sure that $\L_t$ admits a front diagram $\widetilde{\F}_t$ for every value of $t$ except for $t$ belonging to a finite (possibly empty) set of ``critical'' values. 
The moves in the list of Figure~\ref{f:moves} are obtained by comparing the fronts $\widetilde{\F}_t$ for values of $t$ which are immediately before and after each critical value, as we now explain. 
We argue in a way similar to the proof of Proposition~\ref{p:intr-diagrs}. There is a decomposition  
\[
\L_t = \A_t \cup \tau(\A_t), 
\]
where $\{\A_t \}_{t\in [0,1]}$ is a family of regular Legendrian arcs such that, for each $t\in [0,1]$, $\A_t$ is a finite union of $r$
smoothly embedded arcs such that $\A_t\cap\tau(\A_t)$ consists of the fixed points of $\tau\vert_{\L_t}$. More precisely, 
\[
\del\A_t = \L_t \cap \tau(\L_t) = \L_t \cap \{ (x,0,0)\mid x\in\R \}
\]
consists of $2r$ endpoints. Furthermore, the tangent line to $\A_t$ at each endpoint is the direction of the $y$-axis.

Relative versions of the arguments in~\cite{Sw92} imply that, after a $C^\infty$-small Legendrian perturbation of 
$\{\A_t \}_{t\in [0,1]}$ which coincides with the identity on the $x$-axis, there exists a finite subset $\{t_1,\ldots,t_k\}\subset (0,1)$ such that, for each $t\not\in\{t_1,\ldots,t_k\}$, the front 
$\F_t = \pi(\A_t)$ is front diagram with no crossings nor cusps on the $x$-axis and such that the endpoints of $\F_t$ are simple. 
Let us call such fronts {\em $x$-generic}. By general position, the projection $\pi(\A_{t})$ can stop being an $x$-generic front at, say, $t_i$ only if, as $t$ tends towards $t_i$, a crossing or a cusp moves towards the $x$-axis or an endpoint. 
Given the above discussion, we may assume that, for $\ep>0$ sufficiently small, the whole front stays generic except in a small ball. 
Then, for each $i=1,\ldots, k$ one of the following must hold as $t$ goes through $t_i$: 
\begin{enumerate}
    \item $\F_t$ changes, away from its endpoints, as in a standard Legendrian Reidemeister move; 
    \item one strand of $\F_t$ goes through some endpoint as in Figure~\ref{f:transitions}(a);
    \item one crossing of $\F_t$ crosses the $x$-axis as in Figure~\ref{f:transitions}(b);
    \item one cusp of $\F_t$ crosses the $x$-axis as in Figure~\ref{f:transitions}(c);
    \item one cusp reaches an endpoint and disappears as in Figure~\ref{f:transitions}(d). 
\end{enumerate}
\begin{figure}[ht]
    \centering
    \begin{tikzpicture}[scale=0.5]

\begin{scope}[shift = {+(-5,0)}]
\draw[very thin, red] (0,0) -- (3,0);
\draw[fill, very thick] (1.5,0) circle (.03);
\draw[very thick] (1.5,0) circle (.03) .. controls +(.5,0) and +(-.3,-1) .. (3,2);
\draw[white , line width = 3] (0,1) .. controls +(-1,0) and +(.5,-1) .. (3,-1);
\draw[very thick]  (0,1) .. controls +(.5,-1) and  +(-1,0) .. (3,-1);
\node at (1.5,-1.5) {$t_i - \epsilon$};
\end{scope}

\draw[very thin, red] (0,0) -- (3,0);
\draw[fill, very thick] (1.5,0) circle (.03);
\draw[very thick] (1.5,0) circle (.03) .. controls +(.5,0) and +(-.3,-1) .. (3,2);
\draw[line width = 2, white] (0,1) -- (3,-1);
\draw[very thick] (0,1) -- (3,-1);
\node at (1.5,-1.5) {$t_i$};
\node at (-1,0) {$\rightsquigarrow$};
\node at (4,0) {$\rightsquigarrow$};
\node at (1.5,-2.5) {(a)};

\begin{scope}[shift = {+(5,0)}]
\draw[very thin, red] (0,0) -- (3,0);
\draw[fill, very thick] (1.5,0) circle (.03);
\draw[very thick] (1.5,0) circle (.03) .. controls +(.5,0) and +(-.3,-1) .. (3,2);
\draw[white , line width = 3]  (0,1) .. controls +(1,0) and +(-.5,1) .. (3,-1);
\draw[very thick] (0,1) .. controls +(1,0) and +(-.5,1) .. (3,-1);
\node at (1.5,-1.5) {$t_i+ \epsilon$};
\end{scope}

\begin{scope}[shift = {+(10,0)}]
\draw[very thick] (0,-1) .. controls +(1,0) and +(-.5,-1) .. (3,1);
 \draw[white , line width = 2] (0,0) -- (3,0);
\draw[very thin, red] (0,0) -- (3,0);
\draw[white , line width = 3] (0,1) .. controls +(-1,0) and +(.5,-1) .. (3,-1);
\draw[very thick]  (0,1) .. controls +(.5,-1) and  +(-1,0) .. (3,-1);
\node at (1.5,-1.5) {$t_i - \epsilon$};
\end{scope}

\begin{scope}[shift = {+(15,0)}]
\draw[very thick] (0,-1) -- (3,1);
\draw[white , line width = 3] (0,1) -- (3,-1);
\draw[very thin, red] (0,0) -- (3,0);
\draw[white , line width = 2.5] (0,1) -- (3,-1);
\draw[very thick] (0,1) -- (3,-1);
\node at (-1,0) {$\rightsquigarrow$};
\node at (4,0) {$\rightsquigarrow$};
\node at (1.5,-1.5) {$t_i$};
\node at (1.5,-2.5) {(b)};
\end{scope}

\begin{scope}[shift = {+(20,0)}]
\draw[very thick]  (3,1) .. controls +(-1,0) and +(.5,1) .. (0,-1);
\draw[white , line width = 2] (0,0) -- (3,0);
\draw[very thin, red] (0,0) -- (3,0);
\draw[white , line width = 3]  (0,1) .. controls +(1,0) and +(-.5,1) .. (3,-1);
\draw[very thick] (0,1) .. controls +(1,0) and +(-.5,1) .. (3,-1);
\node at (1.5,-1.5) {$t_i + \epsilon$};
\end{scope}

 \begin{scope}[shift = {+(-5,-5)}]
\draw[very thick] (.75,-.75) .. controls +(.5,.25) and +(-.3,-1) .. (3,2);
\draw[very thick] (.75,-.75) .. controls +(.5,.25) and +(-.3,-.3) .. (3,1);
\draw[white , line width = 2] (0,0) -- (3,0);
\draw[very thin, red] (0,0) -- (3,0);
\node at (1.5,-1.5) {$t_i - \epsilon$};
\end{scope}

\begin{scope}[shift = {+(0,-5)}]
\draw[very thin, red] (0,0) -- (3,0);
\draw[very thick] (1.5,0) .. controls +(.5,.25) and +(-.3,-1) .. (3,2);
\draw[very thick] (1.5,0) .. controls +(.5,.25) and +(-.3,-.3) .. (3,1);
\node at (-1,0) {$\rightsquigarrow$};
\node at (4,0) {$\rightsquigarrow$};
\node at (1.5,-1.5) {$t_i$};
\node at (1.5,-2.5) {(c)};
\end{scope}

\begin{scope}[shift = {+(5,-5)}]
\draw[very thin, red] (0,0) -- (3,0);
\draw[very thick] (2,.5) .. controls +(.5,.25) and +(-.3,-1) .. (3,2);
\draw[very thick] (2,.5) .. controls +(.5,.25) and +(-.3,-.3) .. (3,1);
\node at (1.5,-1.5) {$t_i + \epsilon$};
\end{scope}

\begin{scope}[shift = {+(10,-5)}]
\draw[very thick] (-.25,-1) .. controls +(.5,.25) and +(-.3,-1) .. (3,2);
\draw[white , line width = 2] (0,0) -- (3,0);
\draw[very thin, red] (0,0) -- (3,0);
\draw[very thick] (-.25,-1) .. controls +(.5,.25) and +(-.3,0) .. (2,0) circle (.03);
\node at (1.5,-1.5) {$t_i - \epsilon$};
\end{scope}

\begin{scope}[shift = {+(15,-5)}]
\draw[very thick] (.75,-.5) .. controls +(.5,.25) and +(-.3,-1) .. (3,2);
\draw[white , line width = 2] (0,0) -- (3,0);
\draw[very thin, red] (0,0) -- (3,0);
\draw[very thick] (.75,-.5) .. controls +(.5,.25) and +(-.3,0) .. (2,0) circle (.03);
\node at (1.5,-1.5) {$t_i - \frac{\epsilon}{2}$};
\node at (-1,0) {$\rightsquigarrow$};
\node at (4,0) {$\rightsquigarrow$};
\node at (1.5,-2.5) {(d)};
\end{scope}

\begin{scope}[shift = {+(20,-5)}] 
\draw[very thin, red] (0,0) -- (3,0);
\draw[very thick] (3,2) .. controls +(-.3,-1) and +(.3,0) .. (2,0);
\draw (2,0) circle (.03);
\node at (1.5,-1.5) {$t_i + \epsilon$};
\end{scope}
\end{tikzpicture}     \caption{Examples of transitions through non-$x$-generic fronts.}
    \label{f:transitions}
\end{figure}
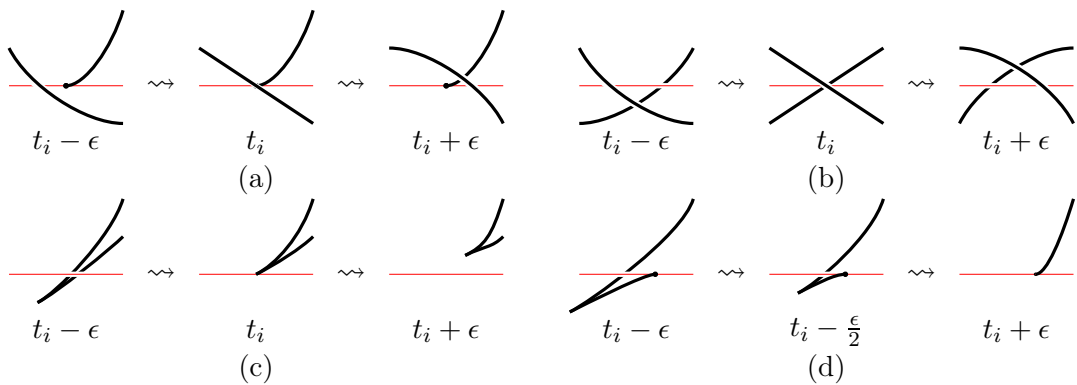
Observe that the above requirements are equivalent to asking that the Legendrian graph given by the union of $\A_{t}$ with the $x$-axis has a generic front for all but finitely many values of $t$, for which one of (1) -- (5) happens. 
Now note that if $c\in\pi(\A_t)\setminus\{\text{$x$-axis}\}$ is a crossing or a cusp, by general position we may suppose that $\tau(c)$ is neither a crossing nor a cusp of $\pi(\A_t)$.
Similarly, if $p,\tau(p)\in\pi(\A_t)\setminus\{\text{$x$-axis}\}$ and $\ell_p$ is the line in the $xz$-plane tangent to $\pi(\A_t)$ at $p$, we may assume that $\tau(\ell_p)$ is distinct from the line tangent to $\pi(\A_t)$ at $\tau(p)$. 
Under these assumptions 
\[
\pi(\L_t) = \pi(\A_t)\cup\pi(\tau\A_t) 
\]
is not a generic front diagram if and only if $t\in\{t_1,\ldots, t_k\}$ and one of the following holds: 
\begin{itemize}
    \item $A_t$ does not have a generic front diagram and either (1) or (2) above holds;
    \item $A_t$ has a generic front diagram and (3), (4), or (5) above holds.
\end{itemize}
It is now straightforward to verify at each $t\in\{t_1,\ldots, t_k\}$ the front $\pi(\L_t) = \widetilde{\F}_t$ changes according to one of the moves of in the statement of the theorem, so the proof is concluded.
\end{proof}

Theorem~\ref{t:SILRT} can be applied to show that the strongly invertible Legendrian unknot on the left-hand side of Figure~\ref{f:unknots} is equivalent to the strongly invertible Legendrian unknot on the right-hand side of Figure~\ref{f:ex1}. Indeed, this fact follows immediately from the following more general statement. 

\begin{cor}\label{c:equ}
Four transvergent front diagrams which differ locally as in Figure~\ref{f:equ} represent pairwise equivalent strongly invertible Legendrian links. 
\begin{figure}[ht]
\centering
\begin{minipage}[c]{3.5cm}
\begin{tikzpicture}[scale = 1.2]
\path[draw=red,line cap=butt,line join=miter,line width=0.010cm] 
(7.2231,-4.4067) -- (9.5914,-4.4095);
\draw[very thick] 
(7.5714,-4.1061) .. controls (7.8745, -4.1179) and (8.0497, -3.7507) .. (8.4827, -3.7397) .. controls (8.9157, -3.7286) and (8.9734, -4.0064) .. (9.2965, -4.0042);
\draw[very thick]
(8.5467,-4.4049) .. controls (8.8794, -4.4044) and (8.8940, -3.9974) .. (9.2945, -4.0035);
\draw[very thick]
(8.5481,-4.4055) .. controls (8.8808, -4.4060) and (8.8821, -4.8112) .. (9.2826, -4.8051);
\draw[very thick]
(7.5642,-4.7031) .. controls (7.8672, -4.6913) and (8.0425, -5.0585) .. (8.4755, -5.0695) .. controls (8.9085, -5.0806) and (8.9662, -4.8029) .. (9.2892, -4.8050);
\end{tikzpicture}
\end{minipage}
\begin{minipage}[c]{3.5cm}
\begin{tikzpicture}[scale = 1.2]
\draw[very thick]
(8.1035,-4.7173) .. controls (8.2977, -4.4211) and (8.8940, -3.9974) .. (9.2945, -4.0035);
\path[draw=black,line cap=butt,line join=miter,line width=0.050cm] 
(7.4409,-4.6398) .. controls (7.7349, -4.6442) and (8.0515, -5.0586) .. (8.4845, -5.0697) .. controls (8.9175, -5.0807) and (9.0360, -4.8059) .. (9.2759, -4.8038);
\path[draw=black,line cap=butt,line join=miter,line width=0.050cm] (8.1035,-4.7173) .. controls (8.1433, -4.6020) and (8.7608, -4.4035) .. (9.1612, -4.4096);
\path[draw=white,line cap=butt,line join=miter,line width=0.10cm] (8.2895,-4.4067) -- (8.5237,-4.4098);
\path[draw=red,line cap=butt,line join=miter,line width=0.010cm] (7.2231,-4.4067) -- (9.5914,-4.4095);
\path[draw=white,line cap=butt,line join=miter,line width=0.040cm] (8.3252,-4.3511) -- (8.4686,-4.4614);
\path[draw=white,line cap=butt,line join=miter,line width=0.10cm] (8.4808,-4.4724) -- (8.5768,-4.5416);
\path[draw=black,line cap=butt,line join=miter,line width=0.050cm] (8.0880,-4.0891) .. controls (8.2821, -4.3853) and (8.8785, -4.8090) .. (9.2789, -4.8030);
\path[draw=white,line cap=butt,line join=miter,line width=0.10cm] (8.4516,-4.2830) -- (8.5868,-4.3237);
\path[draw=black,line cap=butt,line join=miter,line width=0.050cm] 
(8.0911,-4.0916) .. controls (8.1309, -4.2070) and (8.7605, -4.4148) .. (9.1610, -4.4087);
\path[draw=black,line cap=butt,line join=miter,line width=0.050cm] 
(7.4426,-4.1078) .. controls (7.7536, -4.0977) and (8.0497, -3.7507) .. (8.4827, -3.7397) .. controls (8.9157, -3.7286) and (8.9734, -4.0064) .. (9.2965, -4.0042);
\end{tikzpicture}
\end{minipage}
\begin{minipage}[c]{3.5cm}
\begin{tikzpicture}[scale = 1.5]
\draw[very thick] 
(7.5746,-4.1163) .. controls (8.1896, -4.5488) and (8.2545, -4.0258) .. (9.1373, -4.0230);
\path[draw=white,line cap=butt,line join=miter,line width=0.10cm] (8.1760,-4.2073) -- (8.3112,-4.2479);
\draw[very thick] 
(7.9507,-4.6840) .. controls (8.1448, -4.3878) and (8.8504, -4.0177) .. (9.1307, -4.0216);
\draw[very thick] 
(7.9507,-4.6840) .. controls (7.9904, -4.5687) and (8.6079, -4.3702) .. (9.0084, -4.3763);
\path[draw=white,line cap=butt,line join=miter,line width=0.040cm] (8.2084,-4.3743) -- (8.3648,-4.3775);
\path[draw=red,line cap=butt,line join=miter,line width=0.010cm] (7.4392,-4.3754) -- (9.2943,-4.3770);
\path[draw=white,line cap=butt,line join=miter,line width=0.10cm] (8.2146,-4.3296) -- (8.3522,-4.4189);
\path[draw=white,line cap=butt,line join=miter,line width=0.10cm] (8.3700,-4.4274) -- (8.4727,-4.4923);
\path[draw=white,line cap=butt,line join=miter,line width=0.10cm] (8.3589,-4.2664) -- (8.4942,-4.3071);
\path[draw=white,line cap=butt,line join=miter,line width=0.10cm] (8.0828,-4.2212) -- (8.1997,-4.3214);
\draw[very thick]
(7.9400,-4.0641) .. controls (8.1341, -4.3603) and (8.8579, -4.7344) .. (9.1382, -4.7305);
\draw[very thick]
(7.9383,-4.0583) .. controls (7.9780, -4.1736) and (8.6077, -4.3815) .. (9.0081, -4.3754);
\path[draw=white,line cap=butt,line join=miter,line width=0.10cm] (8.0909,-4.4668) -- (8.2755,-4.5388);
\draw[very thick]
(7.5701,-4.6469) .. controls (8.1851, -4.2144) and (8.2500, -4.7374) .. (9.1328, -4.7402);
\end{tikzpicture}
\end{minipage}
\begin{minipage}[c]{3.5cm}
\begin{tikzpicture}[scale = 1.3]
\draw[very thick]
(7.4987,-4.8833) .. controls (8.3018, -5.0440) and (8.1640, -4.7464) .. (8.8327, -4.7990);
\draw[very thick]
(7.8307,-5.2813) .. controls (8.2697, -5.2653) and (8.3734, -4.8033) .. (8.8298, -4.8011);
\draw[very thick]
(7.8376,-5.2806) .. controls (9.2769, -5.5504) and (9.2595, -5.0289) .. (9.6715, -5.0141);
\path[draw=white,line cap=butt,line join=miter,line width=0.040cm] (8.2956,-5.0161) -- (8.4520,-5.0193);
\path[draw=red,line cap=butt,line join=miter,line width=0.010cm] (7.3004,-5.0169) -- (9.9017,-5.0166);
\path[draw=white,line cap=butt,line join=miter,line width=0.10cm] (8.3305,-4.9763) -- (8.4475,-5.0871);
\path[draw=white,line cap=butt,line join=miter,line width=0.10cm] (8.3700,-5.0693) -- (8.4727,-5.1341);
\path[draw=white,line cap=butt,line join=miter,line width=0.10cm] (8.1424,-4.8201) -- (8.2593,-4.9202);
\draw[very thick]
(7.8396,-4.7504) .. controls (8.3313, -4.7391) and (8.3678, -5.2430) .. (8.8385, -5.2360);
\draw[very thick]
(7.8376,-4.7504) .. controls (9.2220, -4.4635) and (9.3601, -5.0287) .. (9.6744, -5.0119);
\path[draw=white,line cap=butt,line join=miter,line width=0.10cm] (8.1247,-5.1160) -- (8.3094,-5.1880);
\draw[very thick]
(7.5037,-5.1497) .. controls (8.3068, -4.9890) and (8.1690, -5.2866) .. (8.8377, -5.2340);
\end{tikzpicture}
\end{minipage} \caption{Equivalent strongly invertible Legendrian links}
\label{f:equ}
\end{figure}
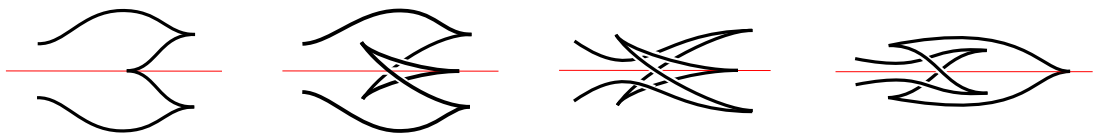
\end{cor}

\begin{proof} 
Denote by $\F_0,\F_1,\F_2$ and $\F_3$ the transvergent front diagrams differing in a disk as in Figure~\ref{f:equ}, numbered from the left to the right. 
Then, $\F_1$ is obtained from $\F_0$ applying a $CR$-move, and by a symmetric $LR$-move of type R2 one obtains $\F_2$ from $\F_1$ and $\F_3$ from $\F_2$.
\end{proof}

\section{Stabilizations of strongly invertible Legendrian links}\label{s:stab}

The purpose of this section is to define stabilizations of strongly invertible Legendrian links. 
This will be done using transvergent front diagrams. 

We will say that such a diagram is {\em connected} if it represents a strongly invertible knot. 
Similarly, given a transvergent front diagram $\F$ representing a Legendrian link $\L$, we refer to the projection of each component of $\L$ simply as to a {\em connected component of $\F$}. 

The following definition introduces stabilizations for transvergent front diagrams. 
They will be used, in conjunction with Theorem~\ref{t:SILRT}, to define stabilizations for equivalence classes of strongly invertible Legendrian links. 

\begin{defn}\label{d:frontstabs}
Let $\F$ be a connected, transvergent front diagram and $c\in\F$ a cusp on the $x$-axis. 
We denote by $S(\F,c)$ -- respectively $T(\F,c)$ -- the transvergent front diagram obtained by the operation shown 
in Figure~\ref{f:S-stabs} -- respectively Figure~\ref{f:T-stabs}.  
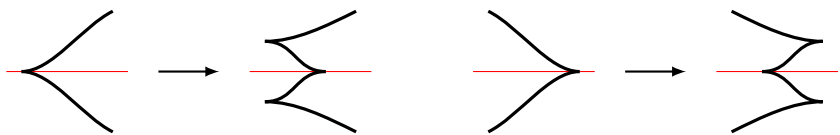
\begin{figure}[ht]
\centering
\begin{minipage}[c]{6cm}
\begin{tikzpicture}[scale = 0.8]
\draw[red] (-3,-5) -- (-1,-5); 
\draw[very thick] (-1.25, -6) .. controls +(-.5,.25) and +(.5,0) .. (-2.75, -5);
\draw[very thick] (-1.25, -4) .. controls +(-.5,-.25) and +(.5,0) ..(-2.75, -5);
\draw[thick, -latex] (-.5,-5) -- (.5,-5);
\begin{scope}[shift = {+(4,0)}]
\draw[red] (-3,-5) -- (-1,-5); 
\draw[very thick] (-1.25, -6) .. controls +(-.5,.25) and +(.5,0) .. (-2.75, -5.5).. controls +(.5,0) and +(-.5,0) .. (-1.75, -5);
\draw[very thick] (-1.25, -4) .. controls +(-.5,-.25) and +(.5,0) ..(-2.75, -4.5).. controls +(.5,0) and +(-.5,0) .. (-1.75,-5);
\end{scope}
\end{tikzpicture}
\end{minipage}
\begin{minipage}[c]{6cm}
\begin{tikzpicture}[scale = 0.8]
\draw[red] (-3,-5) -- (-1,-5); 
\draw[very thick] (-2.75, -6) .. controls +(.5,.25) and +(-.5,0) .. (-1.25, -5);
\draw[very thick] (-2.75, -4) .. controls +(.5,-.25) and +(-.5,0) ..(-1.25, -5);
\draw[thick, -latex] (-.5,-5) -- (.5,-5);

\begin{scope}[shift = {+(4,0)}]
\draw[red] (-3,-5) -- (-1,-5); 
\draw[very thick] (-2.75, -6) .. controls +(.5,.25) and +(-.5,0) .. (-1.25, -5.5).. controls +(-.5,0) and +(.5,0) .. (-2.25, -5);
\draw[very thick] (-2.75, -4) .. controls +(.5,-.25) and +(-.5,0) ..(-1.25, -4.5).. controls +(-.5,0) and +(.5,0) .. (-2.25, -5);
\end{scope}
\end{tikzpicture}
\end{minipage}

\caption{$S$-stabilization of a transvergent front diagram}
\label{f:S-stabs}
\end{figure}
\begin{figure}[ht]
\centering
\begin{minipage}[c]{3cm}
\begin{tikzpicture}[scale = 0.8]
\draw[red] (-3,-5) -- (-1,-5); 
\draw[very thick] (-1.25, -6) .. controls +(-.5,.25) and +(.5,0) .. (-2.75, -5);
\draw[very thick] (-1.25, -4) .. controls +(-.5,-.25) and +(.5,0) ..(-2.75, -5);
\draw[thick, -latex] (-.5,-5) -- (.5,-5);
\end{tikzpicture}
\end{minipage}
\begin{minipage}[c]{3cm}
\begin{tikzpicture}[scale = 1]
\draw[very thick] (7.6660,-2.0875) .. controls (6.8679, -2.0913) and (6.8146, -0.9288) .. (6.4090, -0.9218) .. controls (6.0115, -0.9150) and (5.9955, -1.5752) .. (5.6270, -1.5786);
\draw[very thick] (7.6643,-1.0420) .. controls (6.8926, -1.0449) and (6.8146, -2.2262) .. (6.4090, -2.2331) .. controls (6.0115, -2.2400) and (5.9955, -1.5797) .. (5.6270, -1.5763);
\path[draw=red,line cap=butt,line join=miter,line width=0.010cm] (7.7956,-1.5741) -- (5.4152,-1.5752);
\path[draw=white,line cap=butt,line join=miter,line width=0.1cm] (6.9656,-1.6337) -- (6.8906,-1.5104);
\draw[very thick] (7.6660,-2.0875) .. controls (6.8679, -2.0913) and (6.8146, -0.9288) .. (6.4090, -0.9218) .. controls (6.0115, -0.9150) and (5.9955, -1.5752) .. (5.6270, -1.5786);
\end{tikzpicture}
\end{minipage}
\begin{minipage}[c]{3cm}
\begin{tikzpicture}[scale = 0.8]
\draw[red] (-3,-5) -- (-1,-5); 
\draw[very thick] (-2.75, -6) .. controls +(.5,.25) and +(-.5,0) .. (-1.25, -5);
\draw[very thick] (-2.75, -4) .. controls +(.5,-.25) and +(-.5,0) ..(-1.25, -5);
\draw[thick, -latex] (-.5,-5) -- (.5,-5);
\end{tikzpicture}    
\end{minipage}
\begin{minipage}[c]{3cm}
\begin{tikzpicture}[scale = 1]
\draw[very thick]
(5.5448,-2.0875) .. controls (6.3429, -2.0913) and (6.3962, -0.9288) .. (6.8018, -0.9218) .. controls (7.1993, -0.9150) and (7.2153, -1.5752) .. (7.5838, -1.5786);
\path[draw=red,line cap=butt,line join=miter,line width=0.010cm] (5.4152,-1.5741) -- (7.7956,-1.5752);
\path[draw=white,line cap=butt,line join=miter,line width=0.15cm] (6.4113,-1.8412) -- (6.1383,-1.3059);
\draw[very thick]
(5.5465,-1.0420) .. controls (6.3182, -1.0449) and (6.3962, -2.2262) .. (6.8018, -2.2331) .. controls (7.1993, -2.2400) and (7.2153, -1.5797) .. (7.5838, -1.5763);
\end{tikzpicture}
\end{minipage} \caption{$T$-tabilization of a transvergent front diagram}
\label{f:T-stabs}
\end{figure}
On the left-hand sides of the pictures the case of a left-pointing cusp is shown, on the right-hand sides the case of a right-pointing cusp. 
The fronts $S(\F,c)$ and $T(\F,c)$ are, respectively, the {\em $S$-stabilization} and the {\em $T$-stabilizations} of~$\F$ at $c$. 
\end{defn}

Let $\F$ be a transvergent front diagram and $c\in \F$ a cusp on the $x$-axis. We say that a cusp $c$ is {\em standard} if $c$ is a left-pointing left cusp or a right-pointing right cusp. 
For instance, the cusps of the diagram of Figure~\ref{f:unknot} are both standard cusps.

\begin{defn}\label{d:linkstabs}
Let $(\L,p)$ be a pair consisting of a strongly invertible Legendrian link $\L$ and a distinguished point $p\in\L$ with $\tau(p)=p$. 
Given $\F$, a transvergent front diagram representing~$\L$, denote by $c\in \F$ the cusp on the $x$-axis corresponding to $p$.
The {\em $S$-stabilization} of $\L$ at $p$ is the strongly invertible Legendrian link $S(\L,p)$ determined by~$S(\F,c)$ if $c$ is in standard form, and by~$T(\F,c)$ otherwise. 
Similarly, the {\em $T$-stabilization} of $\L$ at $p$ is the strongly invertible Legendrian link $T(\L,p)$ determined by $T(\F,c)$ if $c$ is in standard form, and by $S(\F,c)$ otherwise. 
\end{defn}

We claim that Definition~\ref{d:linkstabs} is well-posed. 
Let $\F$ and $M\F$ be two transvergent front diagrams such that $M\F$ is obtained from $\F$ applying one of the moves $M$ of Theorem~\ref{t:SILRT}. There is an obvious bijective correspondence between the crossings and the cusps of $\F$ and $M\F$ outside the local picture where the two fronts differ. This correspondence can be extended to include an identification of the cusps on the axis of $\F$ and $M\F$. 
We shall denote by $Mc$ the cusp of $M\F$ corresponding to a cusp $c\in\F$ under the move $M$.

To check that Definition~\ref{d:linkstabs} is well-posed it will suffice to establish the following two lemmas. 

\begin{lem}\label{l:Z-M}
Let $\F$ be a transvergent front diagram and $c\in\F$ a cusp on the $x$-axis. 
Suppose that $Z\in\{S, T\}$ and let $M$ be one of the moves of Theorem~\ref{t:SILRT} with $M\neq CR$. 
Then, the diagrams $Z(M\F,Mc)$ and $Z(\F,c)$ represent equivalent strongly invertible links. 
\end{lem}

\begin{proof}
When $M\in\{XX,CC,SR\}$ the statement is clear because $M$ and $Z$ are supported on disjoint discs. 
The same holds if $M=CX$ but it involves a cusp different from $c$.  
If $M=CX$ involves the cusp $c$ and $Z= S$, then $Z(\F,c)$ is obtained from $Z(M\F,Mc)$ by applying one $CX$-move and two symmetric Reidemeister moves, see~Figure~\ref{f:Z-M1} for an exemplification. 
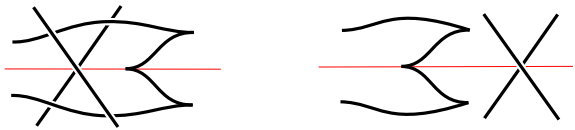
\begin{figure}[ht]
    \centering
    \begin{minipage}[c]{4cm}
\begin{tikzpicture}[scale = 1.2]
\draw[very thick]
    (8.4981,-3.7120) .. controls (8.4812, -3.7442) and (7.5721, -5.0349) ..
    (7.5721, -5.0349);    
\path[draw=red,line cap=butt,line join=miter,line width=0.010cm]
    (7.2231,-4.4070) -- (9.5914,-4.4099);
 \path[draw=white,line cap=butt,line join=miter,line width=0.10cm,miter
    limit=4.00] (8.3333,-3.8848) .. controls (8.3430, -3.8855) and (8.4358,
    -3.8857) .. (8.4358, -3.8857);
\draw[very thick]
    (7.3075,-4.1089) .. controls (7.6106, -4.1207) and (7.8984, -3.8958) ..
    (8.3314, -3.8847) .. controls (8.7643, -3.8737) and (8.9734, -4.0067) ..
    (9.2965, -4.0045);
\draw[very thick]
    (8.5467,-4.4053) .. controls (8.8794, -4.4048) and (8.8940, -3.9977) ..
    (9.2945, -4.0038);
\draw[very thick]
    (8.5481,-4.4058) .. controls (8.8808, -4.4063) and (8.8821, -4.8115) ..
    (9.2826, -4.8055);
\path[draw=white,line cap=butt,line join=miter,line width=0.10cm,miter
    limit=4.00] (7.6932,-4.7826) .. controls (7.6971, -4.7844) and (7.7892,
    -4.8189) .. (7.7892, -4.8189);
\draw[very thick]
    (7.2938,-4.7005) .. controls (7.5968, -4.6887) and (7.8846, -4.9136) ..
    (8.3176, -4.9246) .. controls (8.7506, -4.9357) and (8.9597, -4.8027) ..
    (9.2828, -4.8048);
\path[draw=white,line cap=butt,line join=miter,line width=0.10cm,miter
    limit=4.00] (7.9777,-4.3564) .. controls (7.9797, -4.3600) and (8.0483,
    -4.4528) .. (8.0483, -4.4528);
\path[draw=white,line cap=butt,line join=miter,line width=0.10cm,miter
    limit=4.00] (7.7004,-3.9595) .. controls (7.7023, -3.9632) and (7.7710,
    -4.0559) .. (7.7710, -4.0559);
\path[draw=white,line cap=butt,line join=miter,line width=0.10cm,miter
    limit=4.00] (8.3482,-4.8856) .. controls (8.3501, -4.8892) and (8.4187,
    -4.9820) .. (8.4187, -4.9820);
\draw[very thick]
    (7.5381,-3.7256) .. controls (7.5511, -3.7439) and (8.4641, -5.0485) ..
    (8.4641, -5.0485);
\end{tikzpicture}
\end{minipage}
\begin{minipage}[c]{4cm}
\begin{tikzpicture}[scale = 1.2]
\draw[very thick]
    (10.2141,-3.9705) .. controls (10.1972, -4.0027) and (9.4117, -5.1168) ..
    (9.4117, -5.1168);
\path[draw=red,line cap=butt,line join=miter,line width=0.010cm]
    (7.5939,-4.5517) -- (10.3811,-4.5445);
\draw[very thick]
    (7.8494,-4.1475) .. controls (8.2823, -4.1364) and (8.3509, -3.8521) ..
    (9.2507, -4.1441);
\draw[very thick]
    (8.5009,-4.5449) .. controls (8.8336, -4.5444) and (8.8482, -4.1373) ..
    (9.2487, -4.1434);
\draw[very thick]
    (8.5023,-4.5454) .. controls (8.8350, -4.5459) and (8.8363, -4.9511) ..
    (9.2368, -4.9451);
\path[draw=white,line cap=butt,line join=miter,line width=0.10cm,miter
    limit=4.00] (9.7814,-4.5012) .. controls (9.7834, -4.5048) and (9.8520,
    -4.5976) .. (9.8520, -4.5976);
\draw[very thick]
    (9.4036,-3.9670) .. controls (9.4166, -3.9853) and (10.2277, -5.1323) ..
    (10.2277, -5.1323);
\draw[very thick]
    (7.8335,-4.9357) .. controls (8.2664, -4.9467) and (8.3345, -5.2371) ..
    (9.2343, -4.9450);
\end{tikzpicture}
\end{minipage} \caption{The transvergent front diagrams $Z(M\F,Mc)$ (left) and $Z(\F,c)$ (right) for $M=CX$ and $Z=S$}
    \label{f:Z-M1}
\end{figure}
Similarly, if $Z=T$, then $Z(\F,c)$ is obtained from $Z(M\F,Mc)$ by applying one $CX$-move and one $XX$-move, see~Figure~\ref{f:Z-M2}.  
\begin{figure}[ht]
    \centering
\begin{minipage}[c]{4.5cm}
\begin{tikzpicture}[scale = 1.2]
\draw[very thick]
    (6.9857,-5.3793) .. controls (6.9689, -5.4116) and (6.0597, -6.7022) ..
    (6.0597, -6.7022);

  \path[draw=red,line cap=butt,line join=miter,line width=0.010cm]
    (5.6320,-6.0764) -- (9.2306,-6.0764);

  \path[draw=white,line cap=butt,line join=miter,line width=0.10cm,miter
    limit=4.00] (6.8209,-5.5521) .. controls (6.8306, -5.5529) and (6.9234,
    -5.5531) .. (6.9234, -5.5531);

\draw[very thick]
    (5.7951,-5.7762) .. controls (6.0982, -5.7880) and (6.3860, -5.5631) ..
    (6.8190, -5.5521);

  \path[draw=white,line cap=butt,line join=miter,line width=0.10cm,miter
    limit=4.00] (6.1808,-6.4499) .. controls (6.1847, -6.4518) and (6.2768,
    -6.4863) .. (6.2768, -6.4863);

\draw[very thick]
    (5.7814,-6.3678) .. controls (6.0845, -6.3560) and (6.3723, -6.5809) ..
    (6.8053, -6.5920);

  \path[draw=white,line cap=butt,line join=miter,line width=0.10cm,miter
    limit=4.00] (6.4654,-6.0238) .. controls (6.4673, -6.0274) and (6.5360,
    -6.1202) .. (6.5360, -6.1202);

  \path[draw=white,line cap=butt,line join=miter,line width=0.10cm,miter
    limit=4.00] (6.1880,-5.6269) .. controls (6.1900, -5.6305) and (6.2586,
    -5.7233) .. (6.2586, -5.7233);

\draw[very thick]
    (6.7989,-6.5922) .. controls (7.5970, -6.5961) and (7.7975, -5.4297) ..
    (8.2030, -5.4227) .. controls (8.6005, -5.4158) and (8.6166, -6.0761) ..
    (8.9850, -6.0795);

  \path[draw=white,line cap=butt,line join=miter,line width=0.10cm,miter
    limit=4.00] (6.8358,-6.5529) .. controls (6.8378, -6.5565) and (6.9064,
    -6.6493) .. (6.9064, -6.6493);

  \path[draw=white,line cap=butt,line join=miter,line width=0.10cm]
    (7.6343,-6.1246) -- (7.5698,-6.0324);

\draw[very thick]
    (6.8136,-5.5516) .. controls (7.5838, -5.5406) and (7.7975, -6.7270) ..
    (8.2030, -6.7340) .. controls (8.6005, -6.7408) and (8.6166, -6.0806) ..
    (8.9850, -6.0772);

\draw[very thick]
    (6.0257,-5.3930) .. controls (6.0387, -5.4113) and (6.9517, -6.7159) ..
    (6.9517, -6.7159);
\end{tikzpicture}
\end{minipage}
\hspace{1cm}
\begin{minipage}[c]{4.5cm}
\begin{tikzpicture}[scale=1.2]
 \draw[very thick]
    (5.5577,-2.2236) .. controls (6.3559, -2.2275) and (6.3788, -1.1767) ..
    (6.7843, -1.1697) .. controls (7.1818, -1.1629) and (7.2282, -1.7114) ..
    (7.5967, -1.7148);

  \path[draw=red,line cap=butt,line join=miter,line width=0.010cm]
    (5.4281,-1.7102) -- (8.7596,-1.7160);

  \path[draw=white,line cap=butt,line join=miter,line width=0.10cm]
    (6.3411,-1.7819) -- (6.2761,-1.6702);

 \draw[very thick]
    (8.5101,-1.1428) .. controls (8.4933, -1.1750) and (7.7077, -2.2891) ..
    (7.7077, -2.2891);

  \path[draw=white,line cap=butt,line join=miter,line width=0.10cm,miter
    limit=4.00] (8.0774,-1.6735) .. controls (8.0794, -1.6771) and (8.1480,
    -1.7699) .. (8.1480, -1.7699);

 \draw[very thick]
    (7.6996,-1.1393) .. controls (7.7127, -1.1576) and (8.5237, -2.3046) ..
    (8.5237, -2.3046);

 \draw[very thick]
    (5.5595,-1.1782) .. controls (6.3311, -1.1810) and (6.3687, -2.2553) ..
    (6.7742, -2.2622) .. controls (7.1717, -2.2691) and (7.2282, -1.7159) ..
    (7.5967, -1.7125);
\end{tikzpicture}
\end{minipage} \caption{The transvergent front diagrams $Z(M\F,Mc)$ (left) and $Z(\F,c)$ (right) for $M=CX$ and $Z=T$}
\label{f:Z-M2}
\end{figure}
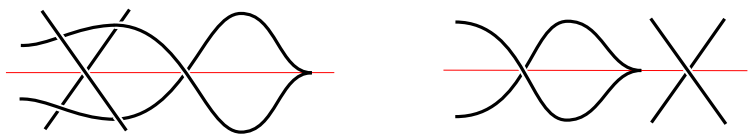
Figures~\ref{f:Z-M1} and~\ref{f:Z-M2} illustrate the cases when $c$ is a right-pointing cusp. 
The cases when $c$ is left-pointing are completely analogous, and the corresponding equivalence can be obtained from those in Figures~\ref{f:Z-M1} and~\ref{f:Z-M2} by a $\pi$ rotation of the plane. 
\end{proof}

Given $Z\in\{S, T\}$, we denote by $\overline{Z}$ the operation obtained from $Z$ by swapping the letters $S$ and $T$, so that $\overline{S} = T$ and $\overline{T} = S$.

\begin{lem}\label{l:Z-CR}
Let $\F$ be a transvergent front diagram, $c\in\F$ a cusp on the $x$-axis, and let $Z\in\{S, T\}$. 
Then, the diagrams $Z(CR\F, CRc)$ and $\overline{Z}(\F,c)$ represent equivalent strongly invertible links. 
\end{lem}

\begin{proof}
When $c$ is right-pointing, Figure~\ref{f:mu2} shows that $S(CR\F, CRc)$ and $T(\F,c)$ are obtained from each other by a $CX$ move, a symmetric $R2$-move and a symmetric $R1$-move. 
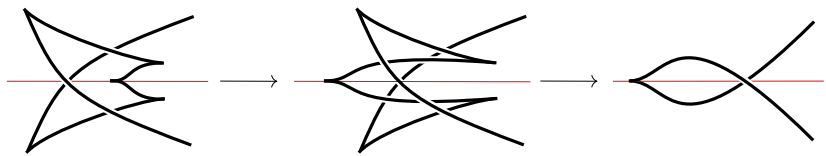
\begin{figure}[ht]
\centering
\begin{tikzpicture}[scale = 1.8]
\draw[very thick] 
(18.5351,-0.5535) .. controls (17.5058, -1.5876) and (17.4880, -0.9886) .. (17.1874, -0.9904);

\draw[very thick] 
(15.2194,-1.5044) .. controls (15.4547, -0.9530) and (15.5407, -0.8475) .. (16.4316, -0.5144);

\draw[very thick] 
(12.7898,-1.5061) .. controls (13.0251, -0.9547) and (13.1112, -0.8492) .. (14.0020, -0.5161);

\path[draw=red,line cap=butt,line join=miter,line width=0.010cm] (12.6369,-0.9930) -- (14.1077,-0.9942);

\draw[very thick]
(13.7844,-0.8592) .. controls (13.5121, -0.8521) and (13.5548, -1.0065) .. (13.3902, -0.9891);

\draw[very thick]
(13.7870,-1.1250) .. controls (13.5148, -1.1320) and (13.5560, -0.9736) .. (13.3915, -0.9910);

\path[draw=white,line cap=butt,line join=miter,line width=0.10cm] (13.3253,-0.7432) -- (13.4555,-0.7892);

\draw[very thick] 
(13.7735,-0.8579) .. controls (13.5818, -0.8504) and (12.8991, -0.6222) .. (12.7609, -0.4593);

\path[draw=white,line cap=butt,line join=miter,line width=0.10cm] (13.0275,-0.9441) -- (13.1108,-1.0363);

\draw[very thick]
(13.7914,-1.1217) .. controls (13.5996, -1.1292) and (12.9169, -1.3574) .. (12.7787, -1.5203);

\path[draw=white,line cap=butt,line join=miter,line width=0.10cm] (13.3252,-1.1788) -- (13.4600,-1.2534);

\draw[very thick] 
(12.7720,-0.4735) .. controls (13.0073, -1.0249) and (13.0933, -1.1304) .. (13.9842, -1.4635);

\path[draw=white,line cap=butt,line join=miter,line width=0.10cm] (15.6192,-0.8334) -- (15.7931,-0.8377);

\draw[very thick] 
(16.2140,-0.8575) .. controls (15.0198, -0.7628) and (15.2537, -0.9928) .. (14.9532, -0.9910);

\path[draw=white,line cap=butt,line join=miter,line width=0.10cm] (15.7549,-0.7415) -- (15.8851,-0.7875);

\path[draw=white,line cap=butt,line join=miter,line width=0.10cm] (15.3702,-0.8211) -- (15.4188,-0.8970);

\path[draw=white,line cap=butt,line join=miter,line width=0.10cm] (15.3465,-1.1147) -- (15.4616,-1.1335);

\draw[very thick]
(16.2031,-0.8562) .. controls (16.0114, -0.8487) and (15.3287, -0.6205) .. (15.1904, -0.4576);

\path[draw=white,line cap=butt,line join=miter,line width=0.10cm] (15.4570,-0.9424) -- (15.5403,-1.0346);

\draw[very thick] 
(16.2209,-1.1200) .. controls (16.0292, -1.1275) and (15.3465, -1.3557) .. (15.2083, -1.5186);

\path[draw=red,line cap=butt,line join=miter,line width=0.010cm] (14.7360,-0.9935) -- (16.4618,-0.9969);

\draw[very thick] 
(16.2159,-1.1195) .. controls (15.0217, -1.2142) and (15.2563, -0.9888) .. (14.9557, -0.9906);

\path[draw=white,line cap=butt,line join=miter,line width=0.10cm] (15.6251,-1.0999) .. controls (15.7166, -1.1568) and (15.8061, -1.2090) .. (15.8905, -1.2491);

\draw[very thick]
(15.2016,-0.4718) .. controls (15.4369, -1.0232) and (15.5229, -1.1287) .. (16.4138, -1.4618);

\path[draw=red,line cap=butt,line join=miter,line width=0.010cm] (17.0654,-0.9933) -- (18.6058,-0.9919);

\path[draw=white,line cap=butt,line join=miter,line width=0.10cm] (17.9716,-0.9442) -- (18.0874,-1.0326);

\draw[very thick]
(18.5285,-1.4279) .. controls (17.4992, -0.3938) and (17.4813, -0.9929) .. (17.1808, -0.9911);

\draw[->,line width=0.010cm] 
(14.1955,-0.9920) -- (14.6122,-0.9934);
\draw[->,line width=0.010cm] 
(16.5347,-0.9916) -- (16.9515,-0.9930);
\end{tikzpicture} \caption{Equivalence of $S_R(CR\F, CRc)$ and $T(\F,c)$.}
\label{f:mu2}
\end{figure}
Similarly, Figure~\ref{f:mu3} shows that $T(CR\F, CRc)$ and $S(\F,c)$ are obtained from each other by a $CC$ move.
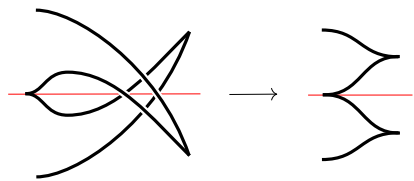
\begin{figure}[ht]
\centering
\begin{tikzpicture}[scale = 1.3]
  \path[draw=red,line cap=butt,line join=miter,line width=0.010cm]
    (15.2013,-3.9734) -- (17.1463,-3.9701);
  \draw[very thick]
    (17.0356,-3.3365) .. controls (16.7181, -4.2625) and (15.8786, -4.8184) ..
    (15.4860, -4.8169);

  \path[draw=white,line cap=butt,line join=miter,line width=0.10cm]
    (16.5398,-4.1027) -- (16.6054,-4.1624);

  \path[draw=white,line cap=butt,line join=miter,line width=0.10cm]
    (16.6821,-3.9353) -- (16.7296,-4.0035);

  \draw[very thick]
    (17.0356,-3.3365) .. controls (16.6652, -3.7069) and (16.2902, -4.1867) ..
    (15.8046, -4.1879) .. controls (15.5681, -4.1885) and (15.5608, -3.9656) ..
    (15.3740, -3.9682);

  \path[draw=white,line cap=butt,line join=miter,line width=0.10cm]
    (16.3354,-3.9386) -- (16.4145,-3.9964);

  \path[draw=white,line cap=butt,line join=miter,line width=0.10cm]
    (16.5386,-3.7566) -- (16.6012,-3.8328);

  \path[draw=red,line cap=butt,line join=miter,line width=0.010cm]
    (18.2354,-3.9826) -- (19.2914,-3.9821);

  \draw[very thick]
    (19.1631,-3.5895) .. controls (18.7739, -3.5782) and (18.7867, -3.9921) ..
    (18.3832, -3.9796);

  \draw[very thick]
    (19.1491,-3.5921) .. controls (18.7551, -3.5924) and (18.7666, -3.3174) ..
    (18.3740, -3.3189);

  \draw[very thick]
    (19.1646,-4.3694) .. controls (18.7754, -4.3807) and (18.7882, -3.9668) ..
    (18.3847, -3.9793);

  \draw[very thick]
    (19.1506,-4.3668) .. controls (18.7566, -4.3665) and (18.7681, -4.6416) ..
    (18.3755, -4.6401);

  \draw[->,line width=0.010cm]
    (17.4389,-3.9764) -- (17.9222,-3.9729);

  \draw[very thick]
    (17.0311,-4.5966) .. controls (16.7136, -3.6705) and (15.8741, -3.1147) ..
    (15.4815, -3.1162);

  \draw[very thick]
    (17.0367,-4.6002) .. controls (16.6663, -4.2298) and (16.2913, -3.7500) ..
    (15.8058, -3.7488) .. controls (15.5692, -3.7482) and (15.5619, -3.9711) ..
    (15.3752, -3.9685);
\end{tikzpicture} \caption{Equivalence between $T(CR\F, CRc)$ and $S(\F,c)$.}
\label{f:mu3}
\end{figure}
When $c$ is left-pointing the argument is similar. 
The corresponding diagrams are obtained from those of Figures~\ref{f:mu2} and~\ref{f:mu3} by a $\pi$-rotation.
\end{proof}

\begin{prop}\label{p:linkstabwellposed}
Definition~\ref{d:linkstabs} is well-posed.
\end{prop}

\begin{proof}
Suppose that $(\F,c)$ and $(\F',c')$ are transvergent front diagrams with distinguished cusps on the $x$-axis. 
Then, $(\F,c)$ and $(\F',c')$ are connected by a sequence $\{M_i\}_{i=0}^k$ of the moves of Theorem~\ref{t:SILRT}. 
In other words, there is a sequence $(\F_i,c_i)$, $i=0,\ldots, k$, of transvergent front diagrams with distinguished cusps on the $x$-axis such that $(\F_0,c_0)=(\F,c)$, $(\F_k,c_k) = (\F',c')$ and $(\F_{i+1}, c_{i+1}) = (M_i\F_i, M_ic_i)$ for each $i=0,\ldots, k$. 
Assume first that either both $c$ and~$c'$ are in standard form or neither of them is. 
We need to check that, for any $Z\in\{S,T\}$, the strongly invertible Legendrian link given by~$Z(\F,c)$ is equivalent to the strongly invertible Legendrian link given by~$Z(\F',c')$.
Since $c$ is in standard form if and only if so is $c'$, the total number of $CR$ moves among the $M_i$'s 
must be even.  
Using this together with the fact that the transformation $Z\to\overline{Z}$ is an involution and applying Lemmas~\ref{l:Z-M} and~\ref{l:Z-CR}, we conclude that $Z(\F,c)$ and $Z(\F',c')$ are equivalent. 
Now suppose that exactly one among $c$ and $c'$ is in standard form.
In this case we have to check that the strongly invertible Legendrian link given by  $Z(\F,c)$ is equivalent to the strongly invertible Legendrian link given by 
$\overline{Z}(\F',c')$. 
Moreover, the total number of $CR$ moves among the $M_i$'s is odd. 
The conclusion follows from Lemmas~\ref{l:Z-M} and~\ref{l:Z-CR} as in the previous case. 
\end{proof}

With Definition~\ref{d:linkstabs} in place, we can now observe that the strongly invertible Legendrian unknots given in Figure~\ref{f:unknots} are nothing but $S(\U,c_R)$ and $T(\U,c_R)$, where $\U$ is given by the diagram of Figure~\ref{f:unknot} and $c_R$ is the fixed point corresponding to the right cusp. 
As we observed in Section~\ref{s:intro}, the Thurston-Bennequin numbers of $\U_S:=S(\U,c_R)$ and $\U_T:=T(\U,c_R)$ are both equal to $-2$. 
Thus, by the classification of Legendrian unknots \cite{EF09}, $\U_S$ and $\U_T$ are equivalent as unoriented Legendrian knots. 
On the other hand, the following corollary of Theorem~\ref{t:SILRT} implies that $\U_S$ and $\U_T$ are not equivalent as strongly invertible Legendrian knots. 
\begin{cor}\label{c:notequ}
Two transvergent front diagrams which differ only in a disk as in Figure~\ref{f:notequ} represent strongly invertible Legendrian links not equivalent to each other.
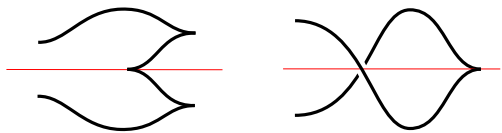
\begin{figure}[ht]
\centering
\begin{minipage}[c]{3.5cm}
\begin{tikzpicture}[scale = 1.2]
\path[draw=red,line cap=butt,line join=miter,line width=0.010cm] (7.2231,-4.4067) -- (9.5914,-4.4095);
\draw[very thick] 
(7.5714,-4.1061) .. controls (7.8745, -4.1179) and (8.0497, -3.7507) .. (8.4827, -3.7397) .. controls (8.9157, -3.7286) and (8.9734, -4.0064) .. (9.2965, -4.0042);
\draw[very thick] 
(8.5467,-4.4049) .. controls (8.8794, -4.4044) and (8.8940, -3.9974) .. (9.2945, -4.0035);
\draw[very thick] 
(8.5481,-4.4055) .. controls (8.8808, -4.4060) and (8.8821, -4.8112) .. (9.2826, -4.8051);
\draw[very thick]
(7.5642,-4.7031) .. controls (7.8672, -4.6913) and (8.0425, -5.0585) .. (8.4755, -5.0695) .. controls (8.9085, -5.0806) and (8.9662, -4.8029) .. (9.2892, -4.8050);
\end{tikzpicture}
\end{minipage}
\begin{minipage}[c]{4cm}
\begin{tikzpicture}[scale=1.2]
\draw[very thick]
(5.5448,-2.0875) .. controls (6.3429, -2.0913) and (6.3962, -0.9288) .. (6.8018, -0.9218) .. controls (7.1993, -0.9150) and (7.2153, -1.5752) .. (7.5838, -1.5786);
\path[draw=white,line cap=butt,line join=miter,line width=0.15cm] (6.4113,-1.8412) -- (6.1383,-1.3059);
\path[draw=white,line cap=butt,line join=miter,line width=0.10cm] (6.5362,-1.5812) -- (6.0054,-1.5810);
\path[draw=red,line cap=butt,line join=miter,line width=0.010cm] (5.4152,-1.5741) -- (7.7956,-1.5752);
\draw[very thick] 
(5.5465,-1.0420) .. controls (6.3182, -1.0449) and (6.3962, -2.2262) .. (6.8018, -2.2331) .. controls (7.1993, -2.2400) and (7.2153, -1.5797) .. (7.5838, -1.5763);
\end{tikzpicture}
\end{minipage} \caption{Not equivalent strongly invertible Legendrian links}
\label{f:notequ}
\end{figure}
In particular, if $(\K,p)$ is a strongly invertible Legendrian knot with a distinguished fixed point $p\in\K$, then $S(\K,p)$ and~$T(\K,p)$ are not equivalent.
\end{cor}

\begin{proof} 
Let $\F$ and $\F'$ the transvergent front diagrams differing only in a disk as in Figure~\ref{f:notequ}. 
Suppose by contradiction that the corresponding strongly invertible links $\L$ and $\L'$ are equivalent. 
Then, the fronts $\F$ and $\F'$ should be connected by a sequence of moves as in Theorem~\ref{t:SILRT}. 
Fix an orientation $\mathcal O$ on $\F$ which makes the relevant cusp point upwards. 
Then, a moment's reflection shows that the moves of Theorem~\ref{t:SILRT} carry $\mathcal O$ to an orientation of $\F'$ which makes the corresponding cusp point upwards. 
But a simple calculation shows that, denoting by $\Vec{\L}$ and $\Vec{\L'}$, respectively, the oriented versions of $\L$ and $\L'$, their rotation numbers satisfy $\rot(\Vec{\L}) - \rot(\Vec{\L'}) = 2$, contradicting the fact that the rotation number is invariant under oriented Legendrian isotopy~\cite{Et05}. 
Therefore $\L$ and $\L'$ are not equivalent.  
To prove the last part of the statement, let $(\F,c)$ be a transvergent front diagram of $\K$ with distinguished cusp $c\in\F$ corresponding to the fixed point $p$.  
When $c$ is right-pointing, the conclusion follows immediately by applying the argument above in a neighborhood of $c$. 
When $c$ is left-pointing, the conclusion is obtained using the same argument to the pictures of Figure~\ref{f:notequ} rotated by $\pi$.  
\end{proof}

\section{The proof of Theorem~\ref{t:stabs}}\label{s:stabequiv}

Up to equivalence, we may assume that $\L_0:=\L$ and $\L_1:=\L'$ have transvergent front diagrams in standard form $\F_0$ and $\F_1$. 
We can turn $\F_i$, $i=0,1$, into a transvergent diagram $D_i$ by smoothing all cusps. 
Let $L_0$ and $L_1$ be the strongly invertible links such that $\pi(L_i) = D_i$, $i=0,1$. 
By assumption, there exists a smooth, equivariant isotopy $L_t$, $t\in [0,1]$ from $L_0$ to $L_1$. 
Our argument will be an adaptation of the argument from~\cite[Theorem~4.4]{FT97}. 
Call a value of $t$ {\em generic} if (i) the projection $\pi(L_t)$ is regular, (ii) its self-intersections are transverse double points, (iii) neither tangent line at each self-intersection point is vertical and (iv) no tangent line at an inflection point is vertical. 
When $t$ is generic, the projection $\pi(L_t)$ is a transvergent diagram of $L_t$ and we denote it by $D_t$.  
By general position, we may assume that there are finitely many non-generic values, say~$t_1,\ldots,t_k\in (0,1)$, that we call the {\em singular values}. 
By the Lobb-Watson equivariant Reidemeister theorem~\cite[Theorem~2.3]{LW21}, for some small $\ep>0$ and each singular value $t_i$, the diagrams $D_{t_i-\ep}$ and $D_{t_i+\ep}$ differ by one of the equivariant Reidemeister moves given in~\cite[Figure~9]{LW21} or by equivariant planar isotopy. 
We warn the reader that Lobb and Watson draw their pictures using a vertical axis of symmetry, while we always represent the $x$-axis as horizontal.  
In the notation of~\cite[Figure~9]{LW21}, (i) is violated if and only if $D_{t_i-\ep}$ and $D_{t_i+\ep}$ differ by a move of type IR1 or R1, (ii) is violated if and only if $D_{t_i-\ep}$ and $D_{t_i+\ep}$ differ by one of the remaining equivariant Reidemeister moves and (iii), (iv) are violated during equivariant planar isotopies. 
We can apply Proposition~\ref{p:legrep} to ``Legendrianize'' $D_0$ and $D_1$ and obtain precisely the transvergent front diagrams in standard form $\F_0$ and $\F_1$.  
Since each $t\in [0,t_1-\ep]$ is regular, the whole family $D_t$, $t\in [0,t_1-\ep]$, can be Legendrianized, therefore the front $\F_0$ extends to a family of trasvergent front diagrams $\F_t$ which lift to strongly invertible Legendrian links $\L_t$ for $t\in [0,t_1-\ep]$. 
This argument applies, more generally, to $D_t$ for
\[
t\in A := [0, t_1 - \ep]\cup [t_1+\ep, t_2-\ep]\cup\cdots\cup [t_k+\ep, 1]\subset [0,1], 
\]
yielding strongly invertible Legendrian links $\L_t$ with trasvergent front diagrams $\F_t$.

Observe that after a number of $S$-stabilizations we can create arbitrarily many symmetric pairs of ``zig-zags'' on the fronts $\F_t$ -- cf.~\cite[Figure~19]{Et05} for the definition of ``zig-zag'' in the non-equivariant setting. 
Moreover, using the moves of Theorem~\ref{t:SILRT} we can equivariantly relocate any symmetric pair of ``zag-zags'' to an arbitrarily chosen place. 
This follows simply by ``symmetrizing'' the proof of \cite[Lemma~4.3]{FT97} as long as the new positions of the symmetric ``zag-zags'' can be reached by moving along arcs of the diagram without crossing the axis. 
To deal with the case when the ``zig-zags'' need to cross the axis, we are going to show that two transvergent fronts differing as in Figure~\ref{f:zigzagmove} represent equivalent strongly invertible links. 
\begin{figure}[ht]
\centering
\begin{tikzpicture}[very thick, scale =.6]
\draw[color = red, very thin] (-2.25,0) -- (2.25,0);
\draw (-2,-2) 
.. controls +(.5,.5) and +(-.5,0) .. (0,-1.5)
.. controls +(-.5,0) and +(.5,0) .. (-1,-.5)
.. controls +(.5,0) and +(-.5,-.5) .. (2,2);
\draw[line width = 3, white] (-2,2) 
.. controls +(.5,-.5) and +(-.5,0) .. (0,1.5)
.. controls +(-.5,0) and +(.5,0) .. (-1,.5)
.. controls +(.5,0) and +(-.5,.5) .. (2,-2);
\draw (-2,2) 
.. controls +(.5,-.5) and +(-.5,0) .. (0,1.5)
.. controls +(-.5,0) and +(.5,0) .. (-1,.5)
.. controls +(.5,0) and +(-.5,.5) .. (2,-2);
\begin{scope}[shift = {+(10,0)}]
\draw[color = red, very thin] (-2.25,0) -- (2.25,0);
\draw (2,2) 
.. controls +(-.5,-.5) and +(.5,0) .. (0,1.5)
.. controls +(.5,0) and +(-.5,0) .. (1,.5)
.. controls +(-.5,0) and +(.5,.5) .. (-2,-2);
\draw[line width = 3, white]  (2,-2) 
.. controls +(-.5,.5) and +(.5,0) .. (0,-1.5)
.. controls +(.5,0) and +(-.5,0) .. (1,-.5)
.. controls +(-.5,0) and +(.5,-.5) .. (-2,2);
\draw (2,-2) 
.. controls +(-.5,.5) and +(.5,0) .. (0,-1.5)
.. controls +(.5,0) and +(-.5,0) .. (1,-.5)
.. controls +(-.5,0) and +(.5,-.5) .. (-2,2);
\end{scope}
\end{tikzpicture} \caption{Symmetric ``zig-zags'' can be moved across the $x$-axis}
\label{f:zigzagmove}
\end{figure}
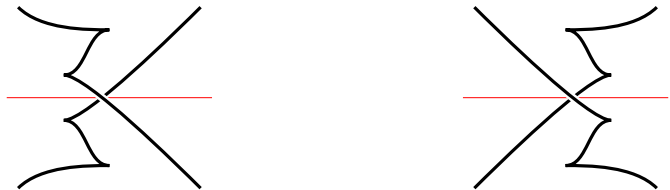
An analogous statement holds for the other type of ``zig-zag''.
The equivalence of the two links is illustrated in Figure~\ref{f:zigzagequiv}, where the first diagram from the left (right, respectively) is obtained from the left (right, respectively) diagram of Figure~\ref{f:zigzagmove} by a symmetric $R2$-move. 
\begin{figure}[ht]
\centering
\begin{tikzpicture}[very thick, scale =.6]
\draw[color = red, very thin] (-2.25,0) -- (2.25,0);
\draw  (-2,-2) 
.. controls +(.5,.5) and +(-.5,0) .. (2,-1.5)
.. controls +(-1,0) and +(1,0) .. (-1,-.5)
.. controls +(1,0) and +(-2,-.5) .. (2,2);
\draw[line width = 3, white] (-2,2) 
.. controls +(.5,-.5) and +(-.5,0) .. (2,1.5)
.. controls +(-1,0) and +(1,0) .. (-1,.5)
.. controls +(1,0) and +(-2,.5) .. (2,-2);
\draw (-2,2) 
.. controls +(.5,-.5) and +(-.5,0) .. (2,1.5)
.. controls +(-1,0) and +(1,0) .. (-1,.5)
.. controls +(1,0) and +(-2,.5) .. (2,-2);

\begin{scope}[shift = {+(6.5,0)}]
\draw[color = red, very thin] (-2.25,0) -- (2.25,0);
\draw  (-2,-2) 
.. controls +(2,.5) and +(-.5,0) .. (2,1.5)
.. controls +(-1,0) and +(1,0) .. (-1,-.5)
.. controls +(1,0) and +(-2,-.5) .. (2,2);
\draw[line width = 3, white] (-2,2) 
.. controls +(2,-.5) and +(-.5,0) .. (2,-1.5)
.. controls +(-1,0) and +(1,0) .. (-1,.5)
.. controls +(1,0) and +(-2,.5) .. (2,-2);
\draw (-2,2) 
.. controls +(2,-.5) and +(-.5,0) .. (2,-1.5)
.. controls +(-1,0) and +(1,0) .. (-1,.5)
.. controls +(1,0) and +(-2,.5) .. (2,-2);
\end{scope}

\begin{scope}[shift = {+(13,0)}]
\draw[color = red, very thin] (-2.25,0) -- (2.25,0);
\draw  (-2,-2) 
.. controls +(2,.5) and +(-.5,0) .. (2,1.5)
.. controls +(-1,0) and +(1,0) .. (-1,.5)
.. controls +(1,0) and +(-2,-.5) .. (2,2);
\draw[line width = 3, white] (-2,2) 
.. controls +(2,-.5) and +(-.5,0) .. (2,-1.5)
.. controls +(-1,0) and +(1,0) .. (-1,-.5)
.. controls +(1,0) and +(-2,.5) .. (2,-2);
\draw (-2,2) 
.. controls +(2,-.5) and +(-.5,0) .. (2,-1.5)
.. controls +(-1,0) and +(1,0) .. (-1,-.5)
.. controls +(1,0) and +(-2,.5) .. (2,-2);
\end{scope}

\draw[latex-latex] (2.5,0) -- (3.25,0) node[above] {$CC$} -- (4,0);
\draw[latex-latex] (9,0) -- (9.75,0) node[above] {$CC$} -- (10.5,0);

\end{tikzpicture}
 \caption{Equivalence of stabilizations}
\label{f:zigzagequiv}
\end{figure}
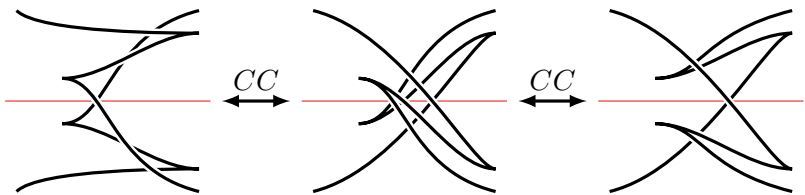
The proof for the other type of zig-zag is similar and left to the reader. 

We now claim that, for each $i = 1,\ldots, k$ the strongly invertible Legendrian knots $\L_{t_i-\ep}$ and~$\L_{t_i+\ep}$ become equivalent after applying sufficiently many $S$-stabilizations. 
If $D_{t_i-\ep}$ and $D_{t_i+\ep}$ differ by one of the moves IR1, IR2 or IR3 the claim follows by ``symmetrizing'' and possibly slightly adapting Pictures R1(a), R2(a) and R3 of~\cite[Figure~22]{FT97}. 
Note that in such adaptations one needs to take into account the choice of configuration in Figure~\ref{f:crossings} made to Legendrianize $D_{t_i+\ep}$. 
If either condition (iii) or condition (iv) is violated at a point of $D_{t_i}$ away from the $x$-axis, we need to ``symmetrize'' the stabilizations provided in the proof of~\cite[Theorem~4.4]{FT97} -- specifically, those of Pictures V1, V2 in~\cite[Figure~21]{FT97} and Pictures V1(a), V2(a) in~\cite[Figure~22]{FT97}.

It remains to deal with the cases when one of the conditions (i)--(iv) is violated at a point of $D_{t_i}$ lying on the $x$-axis. 
We deal first with the violation of (i) or (ii) on the $x$-axis, in which cases $D_{t_i-\ep}$ and $D_{t_i+\ep}$ differ by a Lobb-Watson move of type $R1$, $R2$, $M1$, $M2$ or $M3$. 
Therefore, we may assume that $\F_{t_i-\ep}$ and  $\F_{t_i+\ep}$ differ by a ``Legendrianization'' of one of the five moves. 
For each of the five moves, we need to show that, after sufficiently many equivariant stabilizations, $\F_{t_i-\ep}$ and  $\F_{t_i+\ep}$ differ by a  sequence of moves from the list of Theorem~\ref{t:SILRT}. 

\noindent{\bf move $R1$:}
Up to a $\pi$-rotation of the page, the transvergent front diagram $\F_{t_i-\ep}$ is given, in a disk, by the first picture from the left of Figure~\ref{f:R1}. 
\begin{figure}[ht]
\centering
\begin{minipage}[c]{3.5cm}
\begin{tikzpicture}[scale = 0.5]
\path[draw=red,line cap=butt,line join=miter,line width=0.010cm] (3.8244,-4.6156) -- (8.4735,-4.6184);
\draw[very thick] 
(4.3624,-3.0320) .. controls (4.9500, -3.0172) and (7.2302, -4.6234) .. (7.9370, -4.6116);
\draw[very thick]
(4.3710,-6.1968) .. controls (4.9586, -6.2116) and (7.2387, -4.6054) .. (7.9455, -4.6172);
\end{tikzpicture}
\end{minipage}
\hspace{1cm}
\begin{minipage}[c]{3.5cm}
\begin{tikzpicture}[scale = 1.2]
\draw[very thick]
(8.1035,-4.7173) .. controls (8.2977, -4.4211) and (8.8940, -3.9974) .. (9.2945, -4.0035);
\draw[very thick]
(7.4409,-4.6398) .. controls (7.7349, -4.6442) and (8.0515, -5.0586) .. (8.4845, -5.0697) .. controls (8.9175, -5.0807) and (9.0360, -4.8059) .. (9.2759, -4.8038);
\draw[very thick]
(8.1035,-4.7173) .. controls (8.1433, -4.6020) and (8.7608, -4.4035) .. (9.1612, -4.4096);
\path[draw=white,line cap=butt,line join=miter,line width=0.10cm] (8.2895,-4.4067) -- (8.5237,-4.4098);
\path[draw=red,line cap=butt,line join=miter,line width=0.010cm] (7.2231,-4.4067) -- (9.5914,-4.4095);
\path[draw=white,line cap=butt,line join=miter,line width=0.040cm] (8.3252,-4.3511) -- (8.4686,-4.4614);
\path[draw=white,line cap=butt,line join=miter,line width=0.10cm] (8.4808,-4.4724) -- (8.5768,-4.5416);
\draw[very thick]
(8.0880,-4.0891) .. controls (8.2821, -4.3853) and (8.8785, -4.8090) .. (9.2789, -4.8030);
\path[draw=white,line cap=butt,line join=miter,line width=0.10cm] (8.4516,-4.2830) -- (8.5868,-4.3237);
\draw[very thick] 
(8.0911,-4.0916) .. controls (8.1309, -4.2070) and (8.7605, -4.4148) .. (9.1610, -4.4087);
\draw[very thick] 
(7.4426,-4.1078) .. controls (7.7536, -4.0977) and (8.0497, -3.7507) .. (8.4827, -3.7397) .. controls (8.9157, -3.7286) and (8.9734, -4.0064) .. (9.2965, -4.0042);
\end{tikzpicture}
\end{minipage}
\begin{minipage}[c]{3.5cm}
\begin{tikzpicture}[scale=1.2]
\draw[very thick] 
(5.5448,-2.0875) .. controls (6.3429, -2.0913) and (6.3962, -0.9288) .. (6.8018, -0.9218) .. controls (7.1993, -0.9150) and (7.2153, -1.5752) .. (7.5838, -1.5786);
\path[draw=white,line cap=butt,line join=miter,line width=0.15cm] (6.4113,-1.8412) -- (6.1383,-1.3059);
\path[draw=white,line cap=butt,line join=miter,line width=0.10cm] (6.5362,-1.5812) -- (6.0054,-1.5810);
\path[draw=red,line cap=butt,line join=miter,line width=0.010cm] (5.4152,-1.5741) -- (7.7956,-1.5752);
\draw[very thick]
(5.5465,-1.0420) .. controls (6.3182, -1.0449) and (6.3962, -2.2262) .. (6.8018, -2.2331) .. controls (7.1993, -2.2400) and (7.2153, -1.5797) .. (7.5838, -1.5763);
\end{tikzpicture}
\end{minipage} \caption{$R1$: transvergent fronts $\F_{t_i-\ep}$ and $\F_{t_i+\ep}$ (both types of crossing)}
\label{f:R1}
\end{figure}
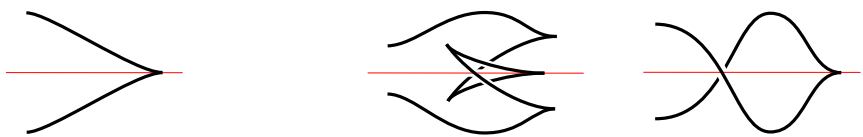
There are two possible cases for the front diagram $\F_{t_i+\ep}$, depending on the type of crossing in the move. 
When the crossing is positive, $\F_{t_i+\ep}$ appears as in the central picture of Figure~\ref{f:R1}, and it is obtained from an equivariant stabilization of $\F_{t_i-\ep}$ by a $CR$-move. 
When the crossing is negative, $\F_{t_i+\ep}$ appears as the first picture from the right of Figure~\ref{f:R1}. In this case, a sequence of moves connecting suitable equivariant stabilizations of $\F_{t_i-\ep}$ and $\F_{t_i+\ep}$ can be obtained using the fronts shown 
in Figure~\ref{f:R1_ii}. 
\begin{figure}[ht]
\centering
\begin{tikzpicture}[very thick, scale = .55]

\draw  (-2,-2.5) .. controls +(1,1) and +(-1,0)..(1.75,1.5) .. controls +(-1,0) and +(.25,-.25).. (-.75,1.5) .. controls +(.25,-.25) and +(-.5,0)..(.5,1) 
;

\draw[line width = 3, white]  (-2,2.5) .. controls +(1,-1) and +(-1,0)..(1.75,-1.5) .. controls +(-1,0) and +(.25,.25).. (-.75,-1.5) .. controls +(.25,.25) and +(-.5,0)..(.5,-1) .. controls +(-.25,0) and +(.5,0)..(-1.5,0)
;

\draw[red, very thin] (-2.6,0) -- (2,0);

\draw[line width = 2, white]  (-2,2.5) .. controls +(1,-1) and +(-1,0)..(1.75,-1.5) .. controls +(-1,0) and +(.25,.25).. (-.75,-1.5) .. controls +(.25,.25) and +(-.5,0)..(.5,-1) ;

\draw (.5,1) .. controls +(-.25,0) and +(.5,0)..(-1.5,0);

\draw[line width = 3, white]  (-2,2.5) .. controls +(1,-1) and +(-1,0)..(1.75,-1.5);

\draw (-2,2.5) .. controls +(1,-1) and +(-1,0)..(1.75,-1.5) .. controls +(-1,0) and +(.25,.25).. (-.75,-1.5) .. controls +(.25,.25) and +(-.5,0)..(.5,-1) .. controls +(-.25,0) and +(.5,0)..(-1.5,0)
;

\begin{scope}[shift = {+(7.1,0)}]
\draw  (-2,-2.5) .. controls +(1,1) and +(-1,0)..(1.75,1.5); 

\draw[white, line width = 3] (-.75,1.5) .. controls +(.25,-.25) and +(-.5,0)..(1.5,1) 
.. controls +(-.25,0) and +(.5,0)..(-1.5,0);

\draw(1.75,1.5)
.. controls +(-1,0) and +(.25,-.25).. (-.75,1.5) .. controls +(.25,-.25) and +(-.5,0)..(1.5,1) 
;

\draw[line width = 3, white]  (-2,2.5) .. controls +(1,-1) and +(-1,0)..(1.75,-1.5) .. controls +(-1,0) and +(.25,.25).. (-.75,-1.5) .. controls +(.25,.25) and +(-.5,0)..(1.5,-1) .. controls +(-.25,0) and +(.5,0)..(-1.5,0)
;

\draw[red, very thin] (-2.6,0) -- (2,0);

\draw[line width = 2, white]  (-2,2.5) .. controls +(1,-1) and +(-1,0)..(1.75,-1.5) .. controls +(-1,0) and +(.25,.25).. (-.75,-1.5) .. controls +(.25,.25) and +(-.5,0)..(1.5,-1) ;

\draw (-.75,-1.5) .. controls +(.25,.25) and +(-.5,0)..(1.5,-1) .. controls +(-.25,0) and +(.5,0)..(-1.5,0);

\draw (1.5,1) .. controls +(-.25,0) and +(.5,0)..(-1.5,0);

\draw[line width = 3, white]  (-2,2.5) .. controls +(1,-1) and +(-1,0)..(1.75,-1.5);

\draw (-2,2.5) .. controls +(1,-1) and +(-1,0)..(1.75,-1.5) .. controls +(-1,0) and +(.25,.25).. (-.75,-1.5) ;
\end{scope}

\begin{scope}[shift = {+(14.2,0)}]
\draw  (-2,-2.5) .. controls +(1,1) and +(-1,0)..(1.75,1.5); 

\draw[white, line width = 3] (-.75,1.5) .. controls +(.25,-.25) and +(-.5,0)..(1.5,1) 
.. controls +(-.25,0) and +(.5,0)..(.5,0);

\draw(1.75,1.5)
.. controls +(-1,0) and +(.25,-.25).. (-.75,1.5) .. controls +(.25,-.25) and +(-.5,0)..(1.5,1) 
;

\draw[line width = 3, white]  (-2,2.5) .. controls +(1,-1) and +(-1,0)..(1.75,-1.5) .. controls +(-1,0) and +(.25,.25).. (-.75,-1.5) .. controls +(.25,.25) and +(-.5,0)..(1.5,-1) .. controls +(-.25,0) and +(.5,0)..(.5,0)
;

\draw[red, very thin] (-2.6,0) -- (2,0);

\draw[line width = 2, white]  (-2,2.5) .. controls +(1,-1) and +(-1,0)..(1.75,-1.5) .. controls +(-1,0) and +(.25,.25).. (-.75,-1.5) .. controls +(.25,.25) and +(-.5,0)..(1.5,-1) ;

\draw (-.75,-1.5) .. controls +(.25,.25) and +(-.5,0)..(1.5,-1) .. controls +(-.25,0) and +(.5,0)..(.5,0);

\draw (1.5,1) .. controls +(-.25,0) and +(.5,0)..(.5,0);

\draw[line width = 3, white]  (-2,2.5) .. controls +(1,-1) and +(-1,0)..(1.75,-1.5);

\draw (-2,2.5) .. controls +(1,-1) and +(-1,0)..(1.75,-1.5) .. controls +(-1,0) and +(.25,.25).. (-.75,-1.5) ;
\end{scope}

\node[above] at (3.25, 0) {SR2};
\node[above] at (10.35, 0) {CR};
\draw[-latex] (2.5, 0) -- (4,0);

\draw[-latex] (9.6, 0) -- (11.1,0);
\end{tikzpicture} \caption{$R1$: equivalence of $\F_{t_i-\ep}$ and $\F_{t_i+\ep}$ up to stabilizations}
\label{f:R1_ii}
\end{figure}
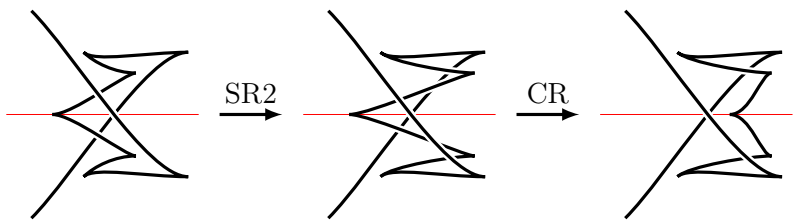
More precisely, the leftmost front in Figure~\ref{f:R1_ii} is obtained from $\F_{t_i-\ep}$ by performing a CR-move and two equivariant stabilizations, while by performing two symmetric R1 on the rightmost front one obtains a stabilization of $\F_{t_i+\ep}$. 

\noindent{\bf move $R2$:}
Up to $\pi$-rotation, there are two types of $R2$-moves, but they can be treated similarly, so without loss of generality we fix an arbitrary choice. The corresponding fronts $\F_{t_i-\ep}$ and $\F_{t_i+\ep}$ are given in Figure~\ref{f:R2a}.
\begin{figure}[ht]
\centering
\begin{minipage}[c]{4cm}
\begin{tikzpicture}[scale = 0.5]
\path[draw=red,line cap=butt,line join=miter,line width=0.010cm] (3.8244,-5.5235) -- (8.4735,-5.5263);
\draw[very thick] 
(3.9099,-4.4044) .. controls (4.6133, -5.4922) and (7.6377, -5.4953) .. (8.4316, -4.4058);
\draw[very thick]
(3.9127,-6.6735) .. controls (4.6160, -5.5857) and (7.6404, -5.5826) .. (8.4343, -6.6721);
\end{tikzpicture}
\end{minipage}
\begin{minipage}[c]{4cm}
\begin{tikzpicture}[scale = 1.2]
\draw[very thick] 
(8.1059,-4.6862) .. controls (8.3000, -4.3900) and (8.8964, -3.9663) .. (9.2968, -3.9723);
\path[draw=white,line cap=butt,line join=miter,line width=0.040cm] (8.2918,-4.3756) -- (8.5261,-4.3787);
\path[draw=red,line cap=butt,line join=miter,line width=0.010cm] (7.2255,-4.3756) -- (9.5938,-4.3784);
\path[draw=white,line cap=butt,line join=miter,line width=0.10cm] (8.3276,-4.3200) -- (8.4710,-4.4303);
\path[draw=white,line cap=butt,line join=miter,line width=0.10cm] (8.5581,-4.1823) -- (8.7595,-4.1994);
\draw[very thick] 
(8.1100,-4.6794) .. controls (8.5326, -4.4681) and (9.4431, -4.5248) .. (9.5971, -4.6778);
\path[draw=white,line cap=butt,line join=miter,line width=0.10cm] (8.5611,-4.4966) -- (8.7274,-4.5978);
\draw[very thick] 
(8.0903,-4.0580) .. controls (8.2845, -4.3542) and (8.8808, -4.7779) .. (9.2813, -4.7719);
\draw[very thick]
(7.2488,-4.7749) .. controls (7.4579, -4.7672) and (8.0679, -3.7169) .. (8.5009, -3.7058) .. controls (8.9339, -3.6948) and (9.0524, -3.9697) .. (9.2923, -3.9717);
\path[draw=white,line cap=butt,line join=miter,line width=0.10cm] (7.6180,-4.3231) -- (7.7266,-4.4418);
\draw[very thick]
(7.2348,-3.9695) .. controls (7.4439, -3.9772) and (8.0539, -5.0275) .. (8.4869, -5.0386) .. controls (8.9199, -5.0496) and (9.0384, -4.7748) .. (9.2783, -4.7727);
\draw[very thick]
(8.0935,-4.0605) .. controls (8.5161, -4.2719) and (9.4267, -4.2151) .. (9.5806, -4.0621);
\end{tikzpicture}
\end{minipage}
 \caption{$R2$: the fronts $\F_{t_i-\ep}$ and $\F_{t_i+\ep}$}
\label{f:R2a}
\end{figure}
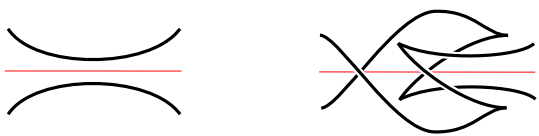
The equivalence of $\F_{t_i-\ep}$ and $\F_{t_i+\ep}$ up to equivariant stabilizations is illustrated in Figure~\ref{f:R2b}. 
Indeed, the pictures show that applying to $\F_{t_i-\ep}$ one equivariant stabilization, a $CC$-move, and a symmetric $LR$-move, the resulting front is obtained from $\F_{t_i+\ep}$ by another symmetric $LR$-move. 
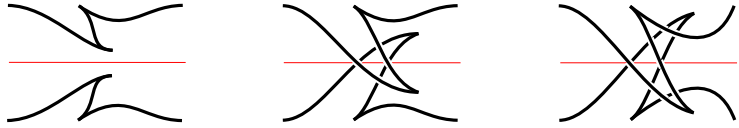
\begin{figure}[ht]
\centering
\begin{minipage}[c]{3.5cm}
\begin{tikzpicture}[scale = 0.5]
\path[draw=red,line cap=butt,line join=miter,line width=0.010cm] (3.8244,-4.7244) -- (8.4735,-4.7272);
\draw[very thick] 
(3.7983,-3.2137) .. controls (4.9090, -3.2322) and (5.7891, -4.3926) .. (6.5559, -4.4045);
\draw[very thick]
(8.3967,-3.2124) .. controls (7.2859, -3.2310) and (6.9538, -4.1000) .. (5.6583, -3.2225);
\draw[very thick]
(5.6787,-3.2305) .. controls (6.2297, -3.6358) and (5.8687, -4.3797) .. (6.5471, -4.4030);
\draw[very thick] 
(3.7781,-6.2673) .. controls (4.8888, -6.2487) and (5.7689, -5.0884) .. (6.5357, -5.0765);
\draw[very thick]
(8.3765,-6.2686) .. controls (7.2657, -6.2500) and (6.9336, -5.3810) .. (5.6381, -6.2585);
\draw[very thick] 
(5.6585,-6.2505) .. controls (6.2095, -5.8451) and (5.8485, -5.1012) .. (6.5269, -5.0779);
\end{tikzpicture}
\end{minipage}
\begin{minipage}[c]{3.5cm}
\begin{tikzpicture}[scale = 0.5]
\path[draw=red,line cap=butt,line join=miter,line width=0.010cm] (3.8244,-4.7244) -- (8.4735,-4.7272);
\draw[very thick]
(8.3967,-3.2124) .. controls (7.2859, -3.2310) and (6.9538, -4.1000) .. (5.6583, -3.2225);
\draw[very thick]
(3.7983,-6.2502) .. controls (4.9090, -6.2316) and (5.7544, -3.9517) .. (7.3655, -3.9537);
\draw[very thick]
(8.3967,-6.2515) .. controls (7.2859, -6.2329) and (6.9538, -5.3639) .. (5.6583, -6.2414);
\draw[very thick]
(5.6787,-6.2334) .. controls (6.2297, -5.8280) and (6.4498, -4.2638) .. (7.3704, -3.9525);
\path[draw=white,line cap=butt,line join=miter,line width=0.10cm] (6.2420,-5.1204) -- (6.4207,-5.2340);
\path[draw=white,line cap=butt,line join=miter,line width=0.10cm] (5.6816,-4.6611) -- (5.8258,-4.7959);
\path[draw=white,line cap=butt,line join=miter,line width=0.10cm] (6.5088,-4.6484) -- (6.6100,-4.8085);
\path[draw=white,line cap=butt,line join=miter,line width=0.10cm] (6.2863,-4.2034) -- (6.3697,-4.3677);
\draw[very thick]
(3.7983,-3.2137) .. controls (4.9090, -3.2322) and (5.7544, -5.5122) .. (7.3655, -5.5102);
\draw[very thick]
(5.6787,-3.2305) .. controls (6.2297, -3.6358) and (6.4498, -5.2001) .. (7.3704, -5.5113);
\end{tikzpicture}
\end{minipage}
\begin{minipage}[c]{3.5cm}
\begin{tikzpicture}[scale = 0.5]
\draw[very thick] 
(8.4129,-6.2612) .. controls (8.0635, -5.3310) and (7.3133, -4.9625) .. (5.6745, -6.2511);
\path[draw=white,line cap=butt,line join=miter,line width=0.1cm] (6.4959,-5.6196) -- (6.6183,-5.7163);
\path[draw=red,line cap=butt,line join=miter,line width=0.010cm] (3.8244,-4.7443) -- (8.4735,-4.7471);
\draw[very thick]
(3.7983,-6.2701) .. controls (4.9124, -6.2556) and (5.9877, -3.7256) .. (7.3523, -3.4313);
\draw[very thick]
(5.6787,-6.2533) .. controls (6.2297, -5.8479) and (6.5744, -3.8917) .. (7.3376, -3.4360);
\path[draw=white,line cap=butt,line join=miter,line width=0.1cm] (6.1447,-5.3247) -- (6.3234,-5.4383);
\path[draw=white,line cap=butt,line join=miter,line width=0.1cm] (5.6003,-4.7065) -- (5.7065,-4.8252);
\path[draw=white,line cap=butt,line join=miter,line width=0.1cm] (6.4336,-4.6874) -- (6.5146,-4.8731);
\path[draw=white,line cap=butt,line join=miter,line width=0.1cm] (6.1849,-4.0952) -- (6.2527,-4.2504);
\path[draw=white,line cap=butt,line join=miter,line width=0.1cm] (6.4957,-3.8036) -- (6.6263,-3.8728);
\path[draw=white,line cap=butt,line join=miter,line width=0.1cm] (6.7977,-3.9431) -- (6.9244,-3.9979);
\path[draw=white,line cap=butt,line join=miter,line width=0.1cm] (6.8063,-5.4678) -- (6.8818,-5.5924);
\draw[very thick] 
(3.7839,-3.2338) .. controls (4.8980, -3.2484) and (5.9734, -5.7783) .. (7.3379, -6.0726);
\draw[very thick]
(5.6643,-3.2507) .. controls (5.9667, -3.4731) and (6.2070, -4.1627) .. (6.4884, -4.8150) .. controls (6.7197, -5.3512) and (6.9789, -5.8623) .. (7.3232, -6.0679);
\draw[very thick] 
(8.3967,-3.2323) .. controls (8.0473, -4.1625) and (7.2971, -4.5309) .. (5.6583, -3.2424);
\end{tikzpicture}
\end{minipage} \caption{$R2$: equivalence of $\F_{t_i-\ep}$ and $\F_{t_i+\ep}$ up to stabilizations}
\label{f:R2b}
\end{figure}

\noindent{\bf move $M1$:} 
Up to rotation, there are two cases, depending on the type of crossing on the $x$-axis. 
In the first case, the equivalence of $\F_{t_i-\ep}$ and $\F_{t_i+\ep}$ up to stabilizations is shown in Figure~\ref{f:M1a}.
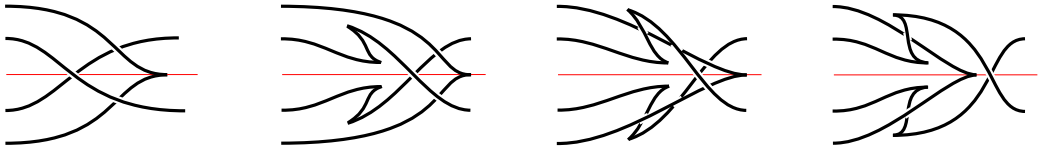
\begin{figure}[ht]
\centering
\begin{minipage}[c]{3.5cm}
\begin{tikzpicture}[scale = 0.6]
\draw[very thick]
(3.8051,-6.2431) .. controls (6.3126, -6.2450) and (6.1160, -4.7247) .. (7.3550, -4.7311);
\path[draw=red,line cap=butt,line join=miter,line width=0.010cm] (3.8244,-4.7244) -- (8.0185,-4.7248);
\path[draw=white,line cap=butt,line join=miter,line width=0.10cm] (6.1785,-5.2489) -- (6.3706,-5.3228);
\draw[very thick]
(3.8050,-5.5186) .. controls (5.0209, -5.5269) and (5.3787, -3.9005) .. (7.6057, -3.9147);
\path[draw=white,line cap=butt,line join=miter,line width=0.10cm] (5.1960,-4.6503) -- (5.3794,-4.7927);
\path[draw=white,line cap=butt,line join=miter,line width=0.10cm] (6.1678,-4.0839) -- (6.3294,-4.2412);
\draw[very thick]
(3.7983,-3.2137) .. controls (6.3057, -3.2117) and (6.1092, -4.7321) .. (7.3482, -4.7257);
\draw[very thick] 
(3.8050,-3.9271) .. controls (5.0209, -3.9189) and (5.2461, -5.5339) .. (7.7452, -5.5258);
\end{tikzpicture}
\end{minipage}
\begin{minipage}[c]{3.5cm}
\begin{tikzpicture}[scale = 0.6]
\path[draw=red,line cap=butt,line join=miter,line width=0.010cm] (3.8244,-4.7538) -- (8.2853,-4.7523);
\draw[very thick] 
(5.2533,-3.6865) .. controls (5.8113, -4.0301) and (5.5641, -4.3459) .. (5.9974, -4.4875);
\draw[very thick] 
(3.7972,-3.9636) .. controls (4.7077, -3.9472) and (5.0819, -4.4923) .. (5.9808, -4.4846);
\draw[very thick]
(3.8039,-6.2698) .. controls (7.6010, -6.2632) and (7.0661, -4.7479) .. (7.9659, -4.7552);
\draw[very thick]
(5.2670,-5.8174) .. controls (5.8250, -5.4738) and (5.5778, -5.1580) .. (6.0111, -5.0164);
\draw[very thick] 
(3.8110,-5.5403) .. controls (4.7214, -5.5567) and (5.0956, -5.0116) .. (5.9945, -5.0193);
\draw[very thick]
(5.2595,-5.8301) .. controls (6.4205, -5.4704) and (7.0492, -3.9680) .. (7.9647, -3.9624);
\path[draw=white,line cap=butt,line join=miter,line width=0.10cm] (6.6224,-4.6784) -- (6.7960,-4.8443);
\path[draw=white,line cap=butt,line join=miter,line width=0.10cm] (7.1666,-4.1983) -- (7.3252,-4.3699);
\path[draw=white,line cap=butt,line join=miter,line width=0.10cm] (7.1515,-5.1829) -- (7.3408,-5.3248);
\draw[very thick] 
(5.2458,-3.6738) .. controls (6.4068, -4.0335) and (7.0354, -5.5359) .. (7.9509, -5.5415);
\draw[very thick]
(3.7983,-3.2430) .. controls (7.5953, -3.2497) and (7.0604, -4.7650) .. (7.9603, -4.7577);\end{tikzpicture}
\end{minipage}
\begin{minipage}[c]{3.5cm}
\begin{tikzpicture}[scale = 0.6]
\draw[very thick]
(5.3740,-6.1758) .. controls (6.5350, -5.8162) and (7.0526, -3.9589) .. (7.9681, -3.9533);
\path[draw=red,line cap=butt,line join=miter,line width=0.010cm] (3.8244,-4.7498) -- (8.2853,-4.7483);
\path[draw=white,line cap=butt,line join=miter,line width=0.10cm] (7.2218,-4.5482) -- (7.0079,-4.4519);
\draw[very thick]
(3.7983,-3.2390) .. controls (5.3306, -3.2201) and (7.0622, -4.7618) .. (7.9620, -4.7545);
\path[draw=white,line cap=butt,line join=miter,line width=0.10cm] (5.6374,-3.6746) -- (5.7346,-3.8669);
\draw[very thick]
(5.3433,-3.3032) .. controls (5.7871, -3.6974) and (5.8113, -4.3459) .. (6.2446, -4.4875);
\draw[very thick] 
(3.7972,-3.9596) .. controls (4.7077, -3.9432) and (5.3412, -4.4961) .. (6.2401, -4.4884);
\path[draw=white,line cap=butt,line join=miter,line width=0.10cm] (6.3507,-4.0588) -- (6.5033,-4.2486);
\draw[very thick]
(5.3604,-6.1876) .. controls (5.8042, -5.7934) and (5.8284, -5.1449) .. (6.2617, -5.0033);
\path[draw=white,line cap=butt,line join=miter,line width=0.10cm] (5.8125,-5.6702) -- (5.5921,-5.7731);
\draw[very thick]
(3.8144,-5.5312) .. controls (4.7248, -5.5476) and (5.3584, -4.9947) .. (6.2573, -5.0024);
\path[draw=white,line cap=butt,line join=miter,line width=0.10cm] (6.8351,-4.6756) -- (6.9757,-4.8490);
\path[draw=white,line cap=butt,line join=miter,line width=0.10cm] (6.5538,-5.2881) -- (6.3320,-5.4027);
\draw[very thick]
(3.7982,-6.2695) .. controls (5.3306, -6.2884) and (7.0621, -4.7467) .. (7.9620, -4.7540);
\path[draw=white,line cap=butt,line join=miter,line width=0.10cm] (7.1849,-5.1019) -- (7.0281,-4.9225);
\draw[very thick]
(5.3569,-3.3150) .. controls (6.5178, -3.6746) and (7.0354, -5.5319) .. (7.9509, -5.5375);
\end{tikzpicture}
\end{minipage}
\begin{minipage}[c]{3.5cm}
\begin{tikzpicture}[scale = 0.6]
\path[draw=red,line cap=butt,line join=miter,line width=0.010cm] (3.8244,-4.7603) -- (8.2853,-4.7588);
\draw[very thick]
(3.7983,-3.2495) .. controls (4.9695, -3.2414) and (6.3028, -4.7558) .. (6.9461, -4.7588);
\path[draw=white,line cap=butt,line join=miter,line width=0.10cm] (5.4272,-3.8197) -- (5.4532,-4.0068);
\draw[very thick]
(5.1248,-3.4373) .. controls (5.6632, -3.4432) and (5.1933, -4.4466) .. (5.8801, -4.4926);
\draw[very thick]
(3.7972,-3.9701) .. controls (4.7077, -3.9537) and (5.0009, -4.5032) .. (5.8998, -4.4954);
\draw[very thick]
(5.1176,-6.0969) .. controls (7.4837, -6.0955) and (7.0989, -3.9675) .. (8.0144, -3.9619);
\path[draw=white,line cap=butt,line join=miter,line width=0.10cm] (7.2009,-4.6655) -- (7.3137,-4.8689);
\draw[very thick]
(5.1248,-6.0923) .. controls (5.6632, -6.0864) and (5.1924, -5.0759) .. (5.8793, -5.0299);
\draw[very thick]
(3.7972,-5.5594) .. controls (4.7077, -5.5758) and (4.9914, -5.0235) .. (5.8903, -5.0313);
\path[draw=white,line cap=butt,line join=miter,line width=0.10cm] (5.5409,-5.5429) -- (5.3626,-5.6583);
\draw[very thick] 
(3.8042,-6.2698) .. controls (4.9754, -6.2779) and (6.3094, -4.7642) .. (6.9528, -4.7613);
\draw[very thick]
(5.1176,-3.4327) .. controls (7.4837, -3.4340) and (7.0989, -5.5621) .. (8.0144, -5.5677);
\end{tikzpicture}
\end{minipage} \caption{$M1$ (first case): equivalence of $\F_{t_i-\ep}$ and $\F_{t_i+\ep}$ up to stabilizations}
\label{f:M1a}
\end{figure}
In the second case, the equivalence is shown in Figure~\ref{f:M1b}. 
\begin{figure}[ht]
\centering
\begin{minipage}[c]{3.5cm}
\begin{tikzpicture}[scale = 0.6]
\path[draw=red,line cap=butt,line join=miter,line width=0.010cm] (3.8244,-4.7696) -- (8.2853,-4.7681);
\draw[very thick] (3.7983,-3.2588) .. controls (4.9695, -3.2508) and (6.5330, -3.5709) .. (7.1763, -3.5738);
\path[draw=white,line cap=butt,line join=miter,line width=0.10cm] (5.4272,-3.8290) -- (5.4532,-4.0161);
\draw[very thick] (7.1885,-3.5769) .. controls (6.6501, -3.5828) and (6.5397, -4.1836) .. (5.8528, -4.2297);
\path[draw=white,line cap=butt,line join=miter,line width=0.10cm] (6.6275,-3.9555) -- (6.4166,-3.9385);
\draw[very thick] (5.8680,-4.2308) .. controls (6.4064, -4.2367) and (6.4834, -4.7158) .. (7.1703, -4.7619);
\draw[very thick] (3.8096,-6.2695) .. controls (4.9808, -6.2776) and (6.5442, -5.9575) .. (7.1876, -5.9546);
\draw[very thick] (5.8793,-5.2976) .. controls (6.4177, -5.2917) and (6.4947, -4.8125) .. (7.1816, -4.7665);
\draw[very thick] (4.5160,-5.8272) .. controls (5.0543, -5.8213) and (5.2824, -4.5471) .. (5.9692, -4.5010);
\path[draw=white,line cap=butt,line join=miter,line width=0.10cm] (5.4061,-4.6775) -- (5.5600,-4.8448);
\path[draw=white,line cap=butt,line join=miter,line width=0.10cm] (5.3385,-5.1118) -- (5.1051,-5.1752);
\draw[very thick] (3.8074,-5.4282) .. controls (4.7179, -5.4118) and (5.0816, -5.0478) .. (5.9805, -5.0400);
\draw[very thick] (3.8018,-4.1123) .. controls (4.7122, -4.1287) and (5.0759, -4.4927) .. (5.9748, -4.5005);
\path[draw=white,line cap=butt,line join=miter,line width=0.10cm] (5.1491,-4.2895) -- (5.2701,-4.4775);
\draw[very thick] (4.5064,-5.8222) .. controls (5.4293, -5.8624) and (7.0057, -5.1883) .. (8.1731, -5.9492);
\path[draw=white,line cap=butt,line join=miter,line width=0.10cm] (6.4323,-5.5150) -- (6.6159,-5.6596);
\draw[very thick] (4.5121,-3.7183) .. controls (5.4350, -3.6781) and (7.0114, -4.3522) .. (8.1788, -3.5913);
\draw[very thick] (7.1998,-5.9515) .. controls (6.6614, -5.9456) and (6.5509, -5.3447) .. (5.8641, -5.2987);
\draw[very thick] (4.5217,-3.7133) .. controls (5.0600, -3.7192) and (5.2880, -4.9934) .. (5.9749, -5.0395);\end{tikzpicture}
\end{minipage}
\begin{minipage}[c]{3.5cm}
\begin{tikzpicture}[scale = 0.6]
\path[draw=red,line cap=butt,line join=miter,line width=0.010cm] (3.8244,-3.5178) -- (8.2853,-3.5164);
\draw[very thick] (3.7983,-2.0071) .. controls (4.9695, -1.9990) and (6.5330, -2.3191) .. (7.1763, -2.3220);
\path[draw=white,line cap=butt,line join=miter,line width=0.10cm] (5.4272,-2.5773) -- (5.4532,-2.7644);
\draw[very thick] (7.1885,-2.3251) .. controls (6.6501, -2.3310) and (6.5397, -2.9319) .. (5.8528, -2.9779);
\path[draw=white,line cap=butt,line join=miter,line width=0.10cm] (6.6428,-2.7220) -- (6.4069,-2.6518);
\draw[very thick] (5.8680,-2.9791) .. controls (6.4064, -2.9850) and (6.4834, -3.4641) .. (7.1703, -3.5101);
\draw[very thick] (5.8793,-4.0458) .. controls (6.4177, -4.0399) and (6.4947, -3.5608) .. (7.1816, -3.5147);
\path[draw=white,line cap=butt,line join=miter,line width=0.10cm] (5.3291,-2.0434) -- (5.5719,-2.2113);
\path[draw=white,line cap=butt,line join=miter,line width=0.10cm] (4.8616,-1.9409) -- (4.9804,-2.1857);
\draw[very thick] (3.8018,-2.8606) .. controls (4.7122, -2.8769) and (5.0759, -3.2410) .. (5.9748, -3.2487);
\path[draw=white,line cap=butt,line join=miter,line width=0.10cm] (5.3365,-3.0635) -- (5.4267,-3.2746);
\path[draw=white,line cap=butt,line join=miter,line width=0.10cm] (6.4323,-4.2633) -- (6.6159,-4.4078);
\draw[very thick] (4.4569,-1.3829) .. controls (6.3444, -2.8866) and (7.0114, -3.1005) .. (8.1788, -2.3395);
\draw[very thick] (4.4564,-5.6515) .. controls (6.3440, -4.1478) and (7.0109, -3.9339) .. (8.1783, -4.6949);
\path[draw=white,line cap=butt,line join=miter,line width=0.10cm] (5.3178,-4.9097) -- (5.6589,-4.8649);
\path[draw=white,line cap=butt,line join=miter,line width=0.10cm] (6.6508,-4.4458) -- (6.4256,-4.2523);
\draw[very thick] (4.4689,-5.6442) .. controls (5.1189, -5.1335) and (5.2876, -3.2927) .. (5.9744, -3.2467);
\path[draw=white,line cap=butt,line join=miter,line width=0.10cm] (5.5014,-3.3998) -- (5.6540,-3.6142);
\path[draw=white,line cap=butt,line join=miter,line width=0.10cm] (5.2602,-3.8775) -- (5.5401,-3.8246);
\path[draw=white,line cap=butt,line join=miter,line width=0.10cm] (4.8101,-4.9646) -- (5.0669,-4.9403);
\draw[very thick] (7.1998,-4.6997) .. controls (6.6614, -4.6938) and (6.5509, -4.0930) .. (5.8641, -4.0469);
\draw[very thick] (3.8096,-5.0178) .. controls (4.9808, -5.0259) and (6.5442, -4.7058) .. (7.1876, -4.7028);
\draw[very thick] (3.8074,-4.1764) .. controls (4.7179, -4.1601) and (5.0816, -3.7960) .. (5.9805, -3.7883);
\draw[very thick] (4.4694,-1.3902) .. controls (5.1194, -1.9009) and (5.2880, -3.7417) .. (5.9749, -3.7877);\end{tikzpicture}
\end{minipage}
\begin{minipage}[c]{3.5cm}
\begin{tikzpicture}[scale = 0.6]
\path[draw=red,line cap=butt,line join=miter,line width=0.010cm] (3.8244,-3.7454) -- (8.2853,-3.7439);
\draw[very thick] (3.7983,-2.2346) .. controls (4.9695, -2.2265) and (4.8117, -2.5711) .. (5.4550, -2.5740);
\draw[very thick] (5.4490,-2.5706) .. controls (4.9107, -2.5765) and (5.2143, -2.8952) .. (4.5275, -2.9412);
\draw[very thick] (5.9861,-2.7011) .. controls (6.6535, -2.7103) and (7.4510, -2.2714) .. (8.1871, -2.2785);
\path[draw=white,line cap=butt,line join=miter,line width=0.10cm] (6.6272,-2.5176) -- (6.9438,-2.6182);
\draw[very thick] (3.8174,-5.7368) .. controls (4.7278, -5.7205) and (6.7796, -4.6872) .. (7.6785, -4.6795);
\draw[very thick] (3.8000,-5.2529) .. controls (4.9712, -5.2609) and (4.8133, -4.9164) .. (5.4567, -4.9135);
\draw[very thick] (5.4507,-4.9169) .. controls (4.9123, -4.9110) and (5.2160, -4.5923) .. (4.5291, -4.5463);
\draw[very thick] (6.0072,-4.7329) .. controls (6.8011, -4.6986) and (7.0721, -2.7532) .. (7.7122, -2.7592);
\path[draw=white,line cap=butt,line join=miter,line width=0.10cm] (6.4522,-4.3600) -- (6.7050,-4.2893);
\draw[very thick] (4.5219,-4.5408) .. controls (7.5254, -4.5305) and (7.2212, -3.7241) .. (8.0833, -3.7381);
\path[draw=white,line cap=butt,line join=miter,line width=0.10cm] (7.2172,-3.4116) -- (6.9920,-3.2981);
\draw[very thick] (4.5171,-2.9398) .. controls (7.5206, -2.9501) and (7.2163, -3.7564) .. (8.0784, -3.7424);
\path[draw=white,line cap=butt,line join=miter,line width=0.10cm] (7.1823,-4.2491) -- (7.0495,-4.0220);
\path[draw=white,line cap=butt,line join=miter,line width=0.10cm] (6.8429,-3.6257) -- (6.9702,-3.8686);
\path[draw=white,line cap=butt,line join=miter,line width=0.10cm] (6.4965,-3.0396) -- (6.6603,-3.2745);
\path[draw=white,line cap=butt,line join=miter,line width=0.10cm] (6.6377,-4.8243) -- (6.9186,-4.9009);
\draw[very thick] (3.8514,-1.7016) .. controls (4.7618, -1.7180) and (6.8136, -2.7512) .. (7.7125, -2.7590);
\draw[very thick] (5.9781,-2.6988) .. controls (6.7720, -2.7331) and (7.0430, -4.6785) .. (7.6831, -4.6724);
\draw[very thick] (6.0100,-4.7379) .. controls (6.6774, -4.7286) and (7.4610, -5.1606) .. (8.1971, -5.1534);\end{tikzpicture}
\end{minipage}
\begin{minipage}[c]{3.5cm}
\begin{tikzpicture}[scale = 0.6]
\path[draw=red,line cap=butt,line join=miter,line width=0.010cm] (3.8244,-3.7454) -- (8.2853,-3.7439);
\draw[very thick] (3.7983,-2.2346) .. controls (4.9695, -2.2265) and (4.8117, -2.5711) .. (5.4550, -2.5740);
\draw[very thick] (5.4490,-2.5706) .. controls (4.9107, -2.5765) and (5.2143, -2.8952) .. (4.5275, -2.9412);
\draw[very thick] (5.9861,-2.7011) .. controls (6.6535, -2.7103) and (7.4510, -2.2714) .. (8.1871, -2.2785);
\path[draw=white,line cap=butt,line join=miter,line width=0.10cm] (6.6272,-2.5176) -- (6.9438,-2.6182);
\draw[very thick] (3.8174,-5.7368) .. controls (4.7278, -5.7205) and (6.7796, -4.6872) .. (7.6785, -4.6795);
\draw[very thick] (3.8000,-5.2529) .. controls (4.9712, -5.2609) and (4.8133, -4.9164) .. (5.4567, -4.9135);
\draw[very thick] (5.4507,-4.9169) .. controls (4.9123, -4.9110) and (5.2160, -4.5923) .. (4.5291, -4.5463);
\draw[very thick] (6.0072,-4.7329) .. controls (6.8011, -4.6986) and (7.0721, -2.7532) .. (7.7122, -2.7592);
\draw[very thick] (4.5171,-2.9398) .. controls (5.3743, -3.0483) and (5.7652, -3.7618) .. (6.6273, -3.7478);
\path[draw=white,line cap=butt,line join=miter,line width=0.10cm] (6.8429,-3.6257) -- (6.9702,-3.8686);
\path[draw=white,line cap=butt,line join=miter,line width=0.10cm] (6.6377,-4.8243) -- (6.9186,-4.9009);
\draw[very thick] (3.8514,-1.7016) .. controls (4.7618, -1.7180) and (6.8136, -2.7512) .. (7.7125, -2.7590);
\draw[very thick] (5.9781,-2.6988) .. controls (6.7720, -2.7331) and (7.0430, -4.6785) .. (7.6831, -4.6724);
\draw[very thick] (6.0100,-4.7379) .. controls (6.6774, -4.7286) and (7.4610, -5.1606) .. (8.1971, -5.1534);
\draw[very thick] (4.5274,-4.5511) .. controls (5.3846, -4.4425) and (5.7645, -3.7347) .. (6.6266, -3.7487);\end{tikzpicture}
\end{minipage} \caption{$M1$ (second case): equivalence of $\F_{t_i-\ep}$ and $\F_{t_i+\ep}$ up to stabilizations}
\label{f:M1b}
\end{figure}
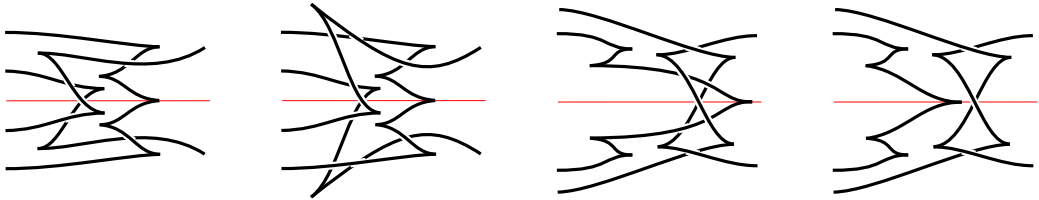

\noindent {\bf move $M2$:} 
Again, there are two cases up to rotation, depending on the crossing on the $x$-asis. But in one case $\F_{t_i-\ep}$ and $\F_{t_i+\ep}$ are identical, so there is nothing to prove. 
In the other case, it is immediately evident that $\F_{t_i-\ep}$ is obtained directly from $\F_{t_i+\ep}$ by equivariant stabilization. We omit the obvious pictures. 

\noindent{\bf move $M3$:}
There are several cases depending on the types of crossings. They can all be treated in similar ways, so we just show how to deal with the most complicated case -- see Figures~\ref{f:M3a} and~\ref{f:M3b}. 
\begin{figure}[ht]
\centering
\begin{tikzpicture}[very thick, scale =0.8]
\draw[very thick] (4.4121,-5.2046) .. controls (4.9173, -5.2069) and (5.2626, -4.8379) .. (6.1615, -4.8302);
\draw[very thick] (3.8600,-5.9963) .. controls (4.2807, -5.9750) and (4.3409, -5.6092) .. (5.0200, -5.6015);
\path[draw=white,line cap=butt,line join=miter,line width=0.10cm] (5.4542,-4.0459) -- (5.5420,-4.2364);
\path[draw=white,line cap=butt,line join=miter,line width=0.10cm] (6.1680,-4.0353) -- (6.2803,-4.2473);
\path[draw=red,line cap=butt,line join=miter,line width=0.010cm] (3.3926,-4.1593) -- (9.3488,-4.1546);
\draw[very thick] (3.8600,-5.9963) .. controls (4.2807, -5.9750) and (4.3409, -5.6092) .. (5.0200, -5.6015);
\draw[very thick] (4.3939,-3.1559) .. controls (4.8991, -3.1536) and (5.2394, -3.5140) .. (6.1383, -3.5217);
\draw[very thick] (5.2476,-4.5891) .. controls (5.7623, -4.5855) and (5.7176, -3.5423) .. (6.1555, -3.5245);
\path[draw=white,line cap=butt,line join=miter,line width=0.10cm] (5.7851,-3.7311) -- (6.0188,-3.6824);
\draw[very thick] (6.9715,-3.1886) .. controls (6.7600, -3.1811) and (6.8899, -3.6119) .. (6.6402, -3.6175);
\path[draw=white,line cap=butt,line join=miter,line width=0.10cm] (5.6363,-4.0523) -- (5.7439,-4.2914);
\draw[very thick] (4.3914,-3.1566) .. controls (4.5933, -3.1359) and (4.6923, -2.8403) .. (4.9942, -2.7683);
\path[draw=white,line cap=butt,line join=miter,line width=0.10cm] (4.6009,-2.8973) -- (4.8648,-2.9529);
\draw[very thick] (3.8417,-2.3642) .. controls (4.2625, -2.3855) and (4.3227, -2.7513) .. (5.0017, -2.7590);
\draw[very thick] (6.4212,-4.5415) .. controls (7.4263, -4.5307) and (7.6290, -4.8631) .. (8.7777, -5.3281);
\path[draw=white,line cap=butt,line join=miter,line width=0.10cm] (6.9941,-4.5149) -- (7.0813,-4.6793);
\draw[very thick] (6.4137,-4.5419) .. controls (6.9285, -4.5382) and (6.8837, -3.4951) .. (7.3217, -3.4772);
\path[draw=white,line cap=butt,line join=miter,line width=0.10cm] (6.9279,-3.7576) -- (7.1141,-3.7184);
\draw[very thick] (6.4111,-3.7876) .. controls (7.4163, -3.7984) and (7.6436, -3.4706) .. (8.7923, -3.0056);
\path[draw=white,line cap=butt,line join=miter,line width=0.10cm] (6.8062,-4.0815) -- (6.8815,-4.2368);
\draw[very thick] (3.9108,-2.8363) .. controls (4.9557, -2.8222) and (6.0797, -3.1786) .. (6.9786, -3.1864);
\draw[very thick] (6.4284,-3.7918) .. controls (6.9431, -3.7955) and (6.8984, -4.8387) .. (7.3364, -4.8565);
\path[draw=white,line cap=butt,line join=miter,line width=0.10cm] (6.7059,-3.4877) -- (6.8977,-3.3914);
\draw[very thick] (5.2485,-3.7671) .. controls (8.6032, -3.4334) and (5.8729, -2.5474) .. (8.7435, -2.5604);
\draw[very thick] (7.3229,-3.4792) .. controls (7.1145, -3.4704) and (6.8830, -3.6213) .. (6.6427, -3.6188);
\draw[very thick] (5.2489,-4.5915) .. controls (8.6036, -4.9252) and (5.8732, -5.8112) .. (8.7439, -5.7982);
\path[draw=white,line cap=butt,line join=miter,line width=0.10cm] (5.8841,-4.6192) -- (5.9845,-4.7491);
\draw[very thick] (3.9314,-5.4978) .. controls (4.9763, -5.5118) and (6.1003, -5.1554) .. (6.9992, -5.1477);
\path[draw=white,line cap=butt,line join=miter,line width=0.10cm] (4.6682,-5.3767) -- (4.8189,-5.5068);
\draw[very thick] (7.3435,-4.8548) .. controls (7.1351, -4.8636) and (6.9036, -4.7128) .. (6.6633, -4.7153);
\path[draw=white,line cap=butt,line join=miter,line width=0.10cm] (6.8208,-4.8589) -- (6.8546,-5.0058);
\draw[very thick] (5.2578,-3.7661) .. controls (5.7726, -3.7698) and (5.7278, -4.8129) .. (6.1658, -4.8308);
\draw[very thick] (4.4096,-5.2040) .. controls (4.6115, -5.2246) and (4.7105, -5.5202) .. (5.0125, -5.5922);
\draw[very thick] (6.9921,-5.1454) .. controls (6.7806, -5.1529) and (6.9105, -4.7222) .. (6.6609, -4.7165);
\begin{scope}[shift={+(8,0)}]
\draw[very thick] (6.9844,-3.1898) .. controls (6.7729, -3.1823) and (6.7099, -3.5776) .. (6.4680, -3.6500);
\draw[very thick] (5.2503,-4.5858) .. controls (5.7651, -4.5821) and (6.1188, -2.7171) .. (6.5567, -2.6992);
\path[draw=white,line cap=butt,line join=miter,line width=0.10cm] (5.6964,-4.0372) -- (5.8418,-4.0052);
\draw[very thick] (4.6685,-4.4363) .. controls (5.2024, -3.9707) and (5.7707, -4.1975) .. (7.1496, -3.5490);
\draw[very thick] (5.7704,-5.4350) .. controls (6.2757, -5.4374) and (6.0314, -5.6597) .. (6.5610, -5.6580);
\path[draw=red,line cap=butt,line join=miter,line width=0.010cm] (3.3926,-4.1593) -- (9.3488,-4.1546);
\draw[very thick] (3.8600,-5.9963) .. controls (4.2807, -5.9750) and (5.5812, -5.6588) .. (6.2028, -5.7713);
\path[draw=white,line cap=butt,line join=miter,line width=0.10cm] (6.0335,-3.0848) -- (6.2974,-3.1404);
\path[draw=white,line cap=butt,line join=miter,line width=0.10cm] (5.7493,-3.9476) -- (5.8480,-3.9394);
\path[draw=white,line cap=butt,line join=miter,line width=0.10cm] (5.9580,-3.9255) -- (6.6260,-3.8350);
\draw[very thick] (4.6621,-3.8929) .. controls (5.2742, -4.0853) and (7.4175, -3.9113) .. (8.7923, -3.0056);
\path[draw=white,line cap=butt,line join=miter,line width=0.10cm] (5.5627,-3.9780) -- (5.6380,-4.1333);
\path[draw=white,line cap=butt,line join=miter,line width=0.10cm] (5.5054,-4.3742) -- (5.6010,-4.3821);
\draw[very thick] (5.2489,-4.5915) .. controls (8.6036, -4.9252) and (5.8732, -5.8112) .. (8.7439, -5.7982);
\path[draw=white,line cap=butt,line join=miter,line width=0.10cm] (6.6373,-4.7942) -- (6.7536,-4.9600);
\path[draw=white,line cap=butt,line join=miter,line width=0.10cm] (5.8940,-4.5973) -- (5.9671,-4.7720);
\draw[very thick] (3.9314,-5.4978) .. controls (4.9763, -5.5118) and (6.1003, -5.1554) .. (6.9992, -5.1477);
\draw[very thick] (7.1298,-4.7757) .. controls (6.8741, -4.7327) and (6.7522, -4.8021) .. (6.4821, -4.6876);
\path[draw=white,line cap=butt,line join=miter,line width=0.10cm] (6.1322,-5.1361) -- (6.2209,-5.3051);
\draw[very thick] (5.7862,-5.4368) .. controls (5.9881, -5.4575) and (5.9054, -5.6931) .. (6.2073, -5.7651);
\draw[very thick] (5.7629,-2.9169) .. controls (6.2682, -2.9145) and (6.0239, -2.6922) .. (6.5535, -2.6939);
\draw[very thick] (3.8525,-2.3556) .. controls (4.2733, -2.3769) and (5.5737, -2.6931) .. (6.1953, -2.5806);
\path[draw=white,line cap=butt,line join=miter,line width=0.10cm] (6.0700,-3.0959) -- (6.2960,-3.1319);
\draw[very thick] (3.9108,-2.8363) .. controls (4.9557, -2.8222) and (6.0797, -3.1786) .. (6.9786, -3.1864);
\path[draw=white,line cap=butt,line join=miter,line width=0.10cm] (5.7916,-3.7032) -- (6.0394,-3.6598);
\path[draw=white,line cap=butt,line join=miter,line width=0.10cm] (5.6258,-4.0687) -- (5.7347,-4.2627);
\draw[very thick] (5.7787,-2.9151) .. controls (5.9806, -2.8945) and (5.8979, -2.6588) .. (6.1998, -2.5868);
\path[draw=white,line cap=butt,line join=miter,line width=0.10cm] (6.5725,-3.5275) -- (6.7657,-3.4626);
\draw[very thick] (5.2485,-3.7671) .. controls (8.6032, -3.4334) and (5.8729, -2.5474) .. (8.7435, -2.5604);
\path[draw=white,line cap=butt,line join=miter,line width=0.10cm] (5.5185,-3.9226) -- (5.6324,-4.0834);
\path[draw=white,line cap=butt,line join=miter,line width=0.10cm] (5.0914,-4.1389) -- (5.2447,-4.1868);
\draw[very thick] (7.1460,-3.5531) .. controls (6.8945, -3.6047) and (6.7561, -3.5332) .. (6.4859, -3.6477);
\draw[very thick] (4.6543,-4.4423) .. controls (5.2664, -4.2499) and (7.4098, -4.4239) .. (8.7846, -5.3296);
\path[draw=white,line cap=butt,line join=miter,line width=0.10cm] (6.0100,-4.3694) -- (6.5950,-4.5485);
\path[draw=white,line cap=butt,line join=miter,line width=0.10cm] (5.5729,-4.2714) -- (5.6693,-4.2915);
\draw[very thick] (4.6762,-3.8989) .. controls (5.2101, -4.3645) and (5.7784, -4.1378) .. (7.1573, -4.7863);
\path[draw=white,line cap=butt,line join=miter,line width=0.10cm] (5.7294,-4.2482) -- (5.8420,-4.4849);
\draw[very thick] (5.2578,-3.7661) .. controls (5.7726, -3.7698) and (6.1263, -5.6348) .. (6.5642, -5.6527);
\draw[very thick] (6.9921,-5.1454) .. controls (6.7806, -5.1529) and (6.7177, -4.7577) .. (6.4757, -4.6853);
\end{scope}
\end{tikzpicture}
 \caption{$M3$: equivalence of $\F_{t_i-\ep}$ and $\F_{t_i+\ep}$ up to stabilizations (first part)}
\label{f:M3a}
\end{figure}
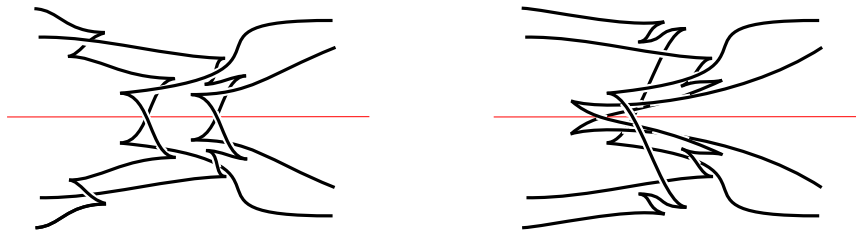
\begin{figure}[ht]
\centering
\begin{tikzpicture}[scale = 0.8]
\draw[very thick] (7.2253,-4.6279) .. controls (7.7400, -4.6242) and (8.2268, -1.5602) .. (8.6647, -1.5423);
\draw[very thick] (6.9844,-2.4704) .. controls (6.7729, -2.4629) and (6.7099, -2.8581) .. (6.4680, -2.9306);
\draw[very thick] (4.6685,-3.7169) .. controls (5.2024, -3.2513) and (5.7707, -3.4781) .. (7.1496, -2.8296);
\draw[very thick] (6.9921,-4.4260) .. controls (6.7806, -4.4335) and (6.7177, -4.0383) .. (6.4757, -3.9659);
\draw[very thick] (7.9929,-5.4457) .. controls (8.4982, -5.4481) and (8.2539, -5.6704) .. (8.7835, -5.6687);
\path[draw=red,line cap=butt,line join=miter,line width=0.010cm] (3.9468,-3.4507) -- (9.3488,-3.4352);
\draw[very thick] (3.9427,-5.5215) .. controls (4.3634, -5.5002) and (7.8000, -5.6676) .. (8.4216, -5.7801);
\path[draw=white,line cap=butt,line join=miter,line width=0.10cm] (6.0335,-2.3654) -- (6.2974,-2.4210);
\path[draw=white,line cap=butt,line join=miter,line width=0.10cm] (5.9580,-3.2061) -- (6.6260,-3.1156);
\path[draw=white,line cap=butt,line join=miter,line width=0.10cm] (7.9301,-2.9935) -- (8.0741,-2.9792);
\draw[very thick] (4.6621,-3.1735) .. controls (5.2742, -3.3659) and (8.0350, -2.9544) .. (9.3525, -2.8464);
\path[draw=white,line cap=butt,line join=miter,line width=0.10cm] (7.6678,-2.9601) -- (7.7108,-3.0674);
\path[draw=white,line cap=butt,line join=miter,line width=0.10cm] (5.5054,-3.6548) -- (5.6010,-3.6627);
\draw[very thick] (3.9314,-4.7784) .. controls (4.9763, -4.7924) and (6.1003, -4.4360) .. (6.9992, -4.4282);
\draw[very thick] (7.1298,-4.0563) .. controls (6.8741, -4.0133) and (6.7522, -4.0827) .. (6.4821, -3.9682);
\draw[very thick] (8.0004,-5.4532) .. controls (8.2023, -5.4739) and (8.1196, -5.7095) .. (8.4215, -5.7815);
\path[draw=white,line cap=butt,line join=miter,line width=0.10cm] (6.0700,-2.3765) -- (6.2960,-2.4125);
\draw[very thick] (3.9108,-2.1169) .. controls (4.9557, -2.1028) and (6.0797, -2.4592) .. (6.9786, -2.4670);
\path[draw=white,line cap=butt,line join=miter,line width=0.10cm] (7.6293,-3.9487) -- (7.7262,-3.9653);
\path[draw=white,line cap=butt,line join=miter,line width=0.10cm] (5.0914,-3.4195) -- (5.2447,-3.4674);
\draw[very thick] (7.1460,-2.8337) .. controls (6.8945, -2.8853) and (6.7561, -2.8138) .. (6.4859, -2.9283);
\draw[very thick] (4.6543,-3.7229) .. controls (5.2664, -3.5305) and (7.8583, -3.9659) .. (9.3144, -4.2459);
\path[draw=white,line cap=butt,line join=miter,line width=0.10cm] (6.0100,-3.6500) -- (6.5950,-3.8291);
\draw[very thick] (4.6762,-3.1795) .. controls (5.2101, -3.6451) and (5.7784, -3.4184) .. (7.1573, -4.0669);
\path[draw=white,line cap=butt,line join=miter,line width=0.10cm] (7.8278,-3.3587) -- (7.8840,-3.5158);
\path[draw=white,line cap=butt,line join=miter,line width=0.10cm] (8.0313,-3.9663) -- (8.0638,-4.0611);
\draw[very thick] (7.8871,-1.7652) .. controls (8.3923, -1.7628) and (8.1480, -1.5405) .. (8.6776, -1.5422);
\draw[very thick] (3.9376,-1.5142) .. controls (4.3584, -1.5355) and (7.6941, -1.5433) .. (8.3157, -1.4308);
\draw[very thick] (7.8945,-1.7577) .. controls (8.0965, -1.7370) and (8.0137, -1.5014) .. (8.3156, -1.4293);
\draw[very thick] (7.2095,-4.6204) .. controls (8.0658, -4.6531) and (7.9551, -5.1741) .. (9.3400, -5.1139);
\path[draw=white,line cap=butt,line join=miter,line width=0.10cm] (8.3486,-4.9520) -- (8.3917,-5.0713);
\path[draw=white,line cap=butt,line join=miter,line width=0.10cm] (8.2003,-2.2121) -- (8.3417,-2.1414);
\draw[very thick] (7.3398,-2.5815) .. controls (7.9578, -2.5822) and (8.0792, -1.9542) .. (9.3477, -1.9490);
\draw[very thick] (7.3311,-2.5830) .. controls (7.8458, -2.5866) and (8.3326, -5.6507) .. (8.7706, -5.6685);
\begin{scope}[shift={+(8,-0.2)}]  
\draw[very thick] (7.5914,-3.6685) .. controls (8.0668, -3.3825) and (8.3835, -1.0092) .. (8.8214, -0.9913);
\draw[very thick] (7.0162,-2.2214) .. controls (6.8047, -2.2140) and (6.7096, -2.3413) .. (6.4676, -2.4137);
\draw[very thick] (4.7003,-3.4680) .. controls (5.3471, -3.1064) and (6.7148, -2.4072) .. (7.1164, -2.3667);
\path[draw=red,line cap=butt,line join=miter,line width=0.010cm] (3.9215,-3.2018) -- (9.3806,-3.1863);
\draw[very thick] (3.9581,-5.3194) .. controls (4.3788, -5.2981) and (8.1755, -5.3115) .. (8.7971, -5.4240);
\path[draw=white,line cap=butt,line join=miter,line width=0.10cm] (6.0653,-2.1164) -- (6.3292,-2.1721);
\path[draw=white,line cap=butt,line join=miter,line width=0.10cm] (5.8281,-2.7914) -- (6.3113,-2.7113);
\path[draw=white,line cap=butt,line join=miter,line width=0.10cm] (8.0518,-2.5010) -- (8.2048,-2.4962);
\draw[very thick] (4.7145,-2.9321) .. controls (5.3752, -2.9408) and (8.3163, -2.2116) .. (9.4049, -2.6050);
\path[draw=white,line cap=butt,line join=miter,line width=0.10cm] (8.6759,-2.4320) -- (8.7162,-2.5339);
\path[draw=white,line cap=butt,line join=miter,line width=0.10cm] (5.5373,-3.4059) -- (5.6329,-3.4137);
\draw[very thick] (3.9632,-4.5294) .. controls (5.0081, -4.5435) and (6.1321, -4.1870) .. (7.0310, -4.1793);
\draw[very thick] (8.4900,-2.1507) .. controls (8.6920, -2.1713) and (8.6022, -2.6097) .. (8.9041, -2.6818);
\path[draw=white,line cap=butt,line join=miter,line width=0.10cm] (6.1018,-2.1275) -- (6.3278,-2.1636);
\draw[very thick] (3.9426,-1.8679) .. controls (4.9875, -1.8538) and (6.1115, -2.2103) .. (7.0104, -2.2180);
\path[draw=white,line cap=butt,line join=miter,line width=0.10cm] (5.1330,-3.1483) -- (5.2899,-3.2263);
\draw[very thick] (7.1296,-2.3624) .. controls (6.8782, -2.4140) and (6.7181, -2.3113) .. (6.4479, -2.4257);
\draw[very thick] (3.9492,-1.0821) .. controls (4.3699, -1.1034) and (8.2186, -1.0994) .. (8.8402, -0.9869);
\path[draw=white,line cap=butt,line join=miter,line width=0.10cm] (8.3804,-4.7030) -- (8.4235,-4.8223);
\draw[very thick] (8.4995,-2.1544) .. controls (8.7200, -2.2122) and (9.0886, -2.2433) .. (9.3860, -2.1544);
\draw[very thick] (7.5986,-3.6620) .. controls (7.8552, -3.4709) and (8.7664, -3.6622) .. (8.9435, -3.6372);
\path[draw=white,line cap=butt,line join=miter,line width=0.10cm] (8.0075,-3.5264) -- (8.0407,-3.6283);
\draw[very thick] (8.5238,-4.1673) .. controls (8.7257, -4.1466) and (8.6292, -3.7078) .. (8.9312, -3.6357);
\path[draw=white,line cap=butt,line join=miter,line width=0.10cm] (8.6646,-3.8451) -- (8.7858,-3.8501);
\draw[very thick] (8.5333,-4.1636) .. controls (8.7537, -4.1058) and (9.1215, -4.0633) .. (9.4190, -4.1522);
\path[draw=white,line cap=butt,line join=miter,line width=0.10cm] (7.8596,-3.1097) -- (7.9159,-3.2668);
\path[draw=white,line cap=butt,line join=miter,line width=0.10cm] (7.9831,-2.7427) -- (8.1283,-2.7413);
\draw[very thick] (7.5649,-2.6559) .. controls (7.8215, -2.8471) and (8.7327, -2.6557) .. (8.9098, -2.6808);
\draw[very thick] (4.7001,-3.4640) .. controls (5.3610, -3.4533) and (8.3140, -4.0300) .. (9.3906, -3.7911);
\path[draw=white,line cap=butt,line join=miter,line width=0.10cm] (5.8165,-3.4797) -- (6.3302,-3.7271);
\path[draw=white,line cap=butt,line join=miter,line width=0.10cm] (8.0852,-3.7675) -- (8.1161,-3.8751);
\draw[very thick] (7.0340,-4.1730) .. controls (6.8225, -4.1805) and (6.7273, -4.0531) .. (6.4854, -3.9807);
\draw[very thick] (7.1474,-4.0320) .. controls (6.8959, -3.9804) and (6.7359, -4.0832) .. (6.4657, -3.9687);
\draw[very thick] (7.5577,-2.6495) .. controls (8.0331, -2.9355) and (8.3644, -5.4017) .. (8.8024, -5.4196);
\draw[very thick] (4.7193,-2.9278) .. controls (5.4457, -3.3125) and (6.6500, -3.9171) .. (7.1419, -4.0230);
\end{scope}    
\end{tikzpicture} \caption{$M3$: equivalence of $\F_{t_i-\ep}$ and $\F_{t_i+\ep}$ up to stabilizations (second part)}
\label{f:M3b}
\end{figure}
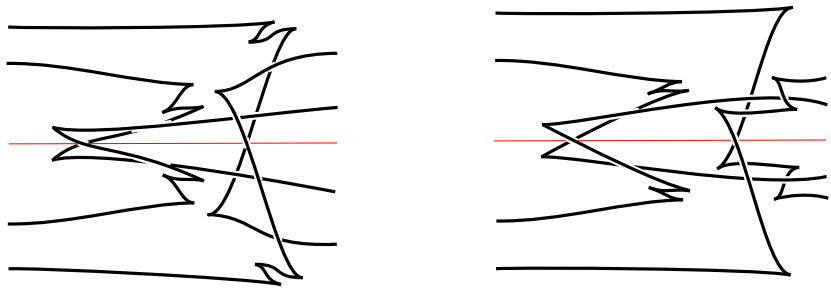

We can now assume that (i) and (ii) hold. 
Then, neither (iii) nor (iv) can be violated at a point of $D_{t_i}$ on the $x$-axis, because equivariance would imply that both tangent lines are vertical, which contradicts (ii). 
This concludes the proof of Theorem~\ref{t:stabs}. \qed

\section{Realization of maximal equivariant Thurston-Bennequin numbers}\label{s:maxeqtb} 

In this section we exhibit two infinite families of strongly invertible Legendrian links which maximize the equivariant Thurston-Bennequin number. 

\begin{proof}[Proof of Proposition~\ref{p:torus-twist}]
The maximal Thurston-Bennequin number of torus knots was computed in~\cite[Theorem~4.1]{EH01}: 
\[
\overline{\tb}(T(2,2n+1)) = \begin{cases}
2n - 1 & n \geq 0,\\
4n - 2 & n < 0.\\
\end{cases}
\]
Strongly invertible Legendrian representatives realizing the maximal Thurston-Bennequin number for $T(2,2n+1)$ are shown in Figure~\ref{f:SILtorusK}.
\begin{figure}[ht]
\centering
\begin{tikzpicture}[very thick]
\draw[very thin, red] (-3,0) -- (3,0);

\draw (-.5,-1.5) .. controls +(-1,0) and +(.5,0) .. (-2.5,0) .. controls +(.5,0) and +(-1,0) .. (-.5,1.5);

\draw (.5,-1.5) .. controls +(1,0) and +(-.5,0) .. (2.5,0) .. controls +(-.5,0) and +(1,0) .. (.5,1.5);

\draw (-.5,-1.5) .. controls +(.25,0) and +(-.25,0) .. (.25,-1.15);
\draw[white, line width = 5](.5,-1.5) .. controls +(-.25,0) and +(.25,0) .. (-.25,-1.15);
\draw (.5,-1.5) .. controls +(-.25,0) and +(.25,0) .. (-.25,-1.15);

\draw (.5,1.5) .. controls +(-.25,0) and +(.25,0) .. (-.25,1.15);

 \draw[white, line width = 5] (-.5,1.5) .. controls +(.25,0) and +(-.25,0) .. (.25,1.15);
\draw (-.5,1.5) .. controls +(.25,0) and +(-.25,0) .. (.25,1.15);

\draw (.25,.25) .. controls +(-.25,0) and +(.25,0) .. (-.25,-.25);

 \draw[white, line width = 5] (-.25,.5) .. controls +(.25,0) and +(-.25,0) .. (.25,-.5);
\draw (-.25,.25) .. controls +(.25,0) and +(-.25,0) .. (.25,-.25);

\draw[line width = 3,fill,white] (.35,1.3) .. controls +(.15,0) and +(-.15,0) .. (.45,0) .. controls +(-.15,0) and +(.15,0) .. (.35,-1.3) .. controls +(.125,0) and +(-.125,0) .. (.45,0) .. controls +(-.125,0) and +(.125,0) .. (.35,1.3);
\draw[line width = 1,fill, very thin] (.35,1.3) .. controls +(.15,0) and +(-.15,0) .. (.45,0) .. controls +(-.15,0) and +(.15,0) .. (.35,-1.3) .. controls +(.125,0) and +(-.125,0) .. (.45,0) .. controls +(-.125,0) and +(.125,0) .. (.35,1.3);
\draw[fill, white] (.5,.2) rectangle (1.7,-.5);
\node[right] at (.34,0) {$2\vert n\vert+ 1$};
\node[below right] at (.34,-.05) {\small{crossings}};

\node at (0,.8) {$\vdots$};
\node at (0,-.65) {$\vdots$};

\node at (0,-2) {$n < 0$};

\begin{scope}[shift = {+(6.5,0)}]
\draw[very thin, red] (-2.75,0) -- (2.25,0);

\draw (-.5,-1) .. controls +(-.5,0) and +(.5,0) .. (-2,0) .. controls +(.5,0) and +(-.5,0) .. (-.5,1);
\draw (-.5,-1.5) .. controls +(-1,0) and +(.5,0) .. (-2.5,0) .. controls +(.5,0) and +(-1,0) .. (-.5,1.5);

\draw[white, line width = 5] (2.22,.55) .. controls +(-.25,0) and +(1,0) .. (1,-1.5);
\draw (2.25,.5) .. controls +(-.25,0) and +(1,0) .. (1,-1.5);
\draw (2.25,.5) .. controls +(-.25,0) and +(1,0) .. (1,1);
 \draw[white, line width = 5] (2.2,-.4) .. controls +(-.25,0) and +(1,0) .. (1,1.5);
\draw(2.25,-.5) .. controls +(-.25,0) and +(1,0) .. (1,1.5);

\draw[white, line width = 5] (2.2,-.55) .. controls +(-.25,0) and +(1,0) .. (1,-1);
\draw (2.25,-.5) .. controls +(-.25,0) and +(1,0) .. (1,-1);
\draw[thin] (-.5, 1.65) rectangle (1,.85);
\draw[thin] (-.5, -1.65) rectangle (1,-.85);

\node at (.25,1.25) {$n$};
\node at (.25,-1.25) {$n$};
\node at (0.25,-2) {$n \geq 0 $};
\end{scope}
\end{tikzpicture} \caption{Strongly invertible Legendrian representatives of $T(2,2n+1)$ for each $n$. Each box contains $n$ right-handed half twists.}
\label{f:SILtorusK}
\end{figure}
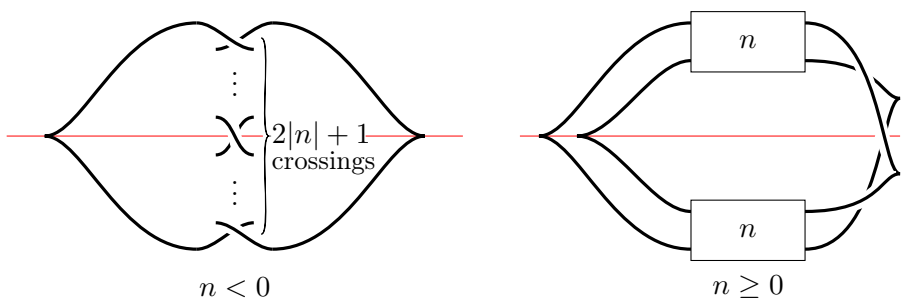

Denote by $K_m$ the twist knot of Figure~\ref{f:twists}.
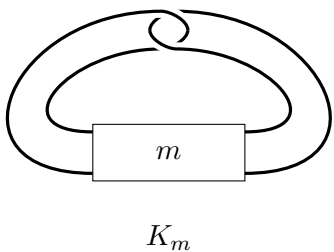
\begin{figure}[ht]
\centering
\begin{tikzpicture}[very thick]

\draw (-1,-1.1) .. controls +(-1.25,0) and +(-1.25,0) .. (-.1,0) .. controls +(.125,0) and +(0,-.125) .. (.25,.25) ;

\draw[white, line width = 5] (1,-1.1) .. controls +(1.25,0) and +(1.25,0) .. (.1,0) .. controls +(-.125,0) and +(0,-.125) .. (-.25,.25) ;

\draw (1,-1.1) .. controls +(1.25,0) and +(1.25,0) .. (.1,0) .. controls +(-.125,0) and +(0,-.125) .. (-.25,.25) ;

\draw (1,-1.65) .. controls +(2,0) and +(2,0) .. (.1,.5) .. controls +(-.125,0) and +(0,.125) .. (-.25,.25) ;

\draw[white, line width = 5] (-1,-1.65) .. controls +(-2,0) and +(-2,0) .. (-.1,.5) .. controls +(.125,0) and +(0,.125) .. (.25,.25) ;

\draw (-1,-1.65) .. controls +(-2,0) and +(-2,0) .. (-.1,.5) .. controls +(.125,0) and +(0,.125) .. (.25,.25) ;

\draw[thin] (-1,-1) rectangle (1,-1.75);
\node at (0,-1.375) {$m$};

\node at (0,-2.5) {$K_m$};
\end{tikzpicture} \caption{The twist knot $K_m$. The box contains $m$ right-handed half twists if
$m \geq 0$, and $\vert m\vert$ left-handed half twists if $m < 0$.}
\label{f:twists}
\end{figure}
It was proved in \cite[Proposition~2.6]{ENV13} that
\[
\overline{\tb}(K_m) = \begin{cases}
-m-1 & m\geq 0 \text{ and even},\\
-m-5 & m\geq 0 \text{ and odd},\\
-1 & m =-1,\\
1 & m < 0 \text{ and even},\\
-3 & m < -1 \text{ and odd}.\\
\end{cases}
\]
Note that $K_{-1}$ and $K_{0}$ are both the unknot, and the desired strongly invertible Legendrian representative is shown in Figure~\ref{f:unknot}. 
In the remaining cases, the strongly invertible Legendrian representatives are shown in Figure~\ref{f:SILtwistK}.
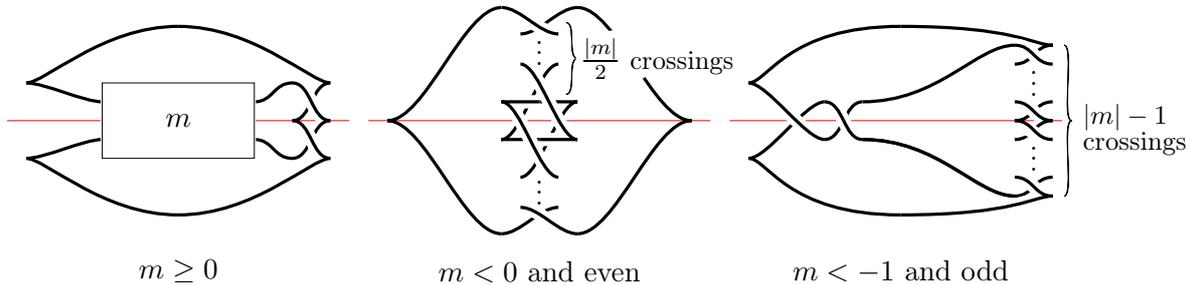
\begin{figure}[ht]
    \centering
    \begin{tikzpicture}[very thick]
\draw[very thin, red] (-2.25,0) -- (2.25,0);

\draw (-.5,-1.5) .. controls +(-.5,0) and +(.5,0) .. (-2,0) .. controls +(.5,0) and +(-.5,0) .. (-.5,1.5);

\draw (.5,-1.5) .. controls +(.5,0) and +(-.5,0) .. (2,0) .. controls +(-.5,0) and +(.5,0) .. (.5,1.5);

\draw (-.5,-1.5) .. controls +(.25,0) and +(-.25,0) .. (.25,-1.15);
\draw[white, line width = 5](.5,-1.5) .. controls +(-.25,0) and +(.25,0) .. (-.25,-1.15);
\draw (.5,-1.5) .. controls +(-.25,0) and +(.25,0) .. (-.25,-1.15);

\draw (.5,1.5) .. controls +(-.25,0) and +(.25,0) .. (-.25,1.15);

\draw[white, line width = 5] (-.5,1.5) .. controls +(.25,0) and +(-.25,0) .. (.25,1.15);
\draw (-.5,1.5) .. controls +(.25,0) and +(-.25,0) .. (.25,1.15);

\draw (0,.25) -- (.5,.25) .. controls +(-.25,0) and +(.25,0) .. (-.25,-.75);
\draw (0,-.25) -- (-.5,-.25) .. controls +(.25,0) and +(-.25,0) .. (.25,.75);

\draw[white, line width = 5] (0,-.25) --(.5,-.25) .. controls +(-.25,0) and +(.25,0) .. (-.25,.75);
\draw (0,-.25) -- (.5,-.25) .. controls +(-.25,0) and +(.25,0) .. (-.25,.75);

\draw[white, line width = 5] (0,.25) -- (-.5,.25) .. controls +(.25,0) and +(-.25,0) .. (.25,-.75);
\draw (0,.25) -- (-.5,.25) .. controls +(.25,0) and +(-.25,0) .. (.25,-.75);

\draw[line width = 1,fill, very thin] (.35,1.3) .. controls +(.15,0) and +(-.15,0) .. (.5,.825) .. controls +(-.15,0) and +(.15,0) .. (.35,.35) .. controls +(.125,0) and +(-.125,0) .. (.5,.825) .. controls +(-.125,0) and +(.125,0) .. (.35,1.3);
\draw[fill, white] (1,.825) ellipse (.5 and .35);
\node[black, right] at  (0.4,.825)  {$\frac{\vert m \vert}{2}$ \small{crossings}};

\node at (0,1) {$\vdots$};
\node at (0,-.85) {$\vdots$};

\node at (0,-2) {$m < 0$ and even};

\begin{scope}[shift = {+(-4.75,0)}]
\draw[very thin, red] (-2.25,0) -- (2.25,0);

\draw (2,.5) .. controls +(-.25,0) and +(.25,0) ..  (1.5,0);

\draw[white, line width = 5] (2,0) .. controls +(-.25,0) and +(.25,0) .. (1.5,.5);
\draw (2,0) .. controls +(-.25,0) and +(.25,0) .. (1.5,.5) .. controls +(-.25,0) and +(.25,0) ..  (1,.25);

\draw  (1,-.25) .. controls +(.25,0) and +(-.25,0) .. (1.5,-.5) .. controls +(.25,0) and +(-.25,0) ..  (2,0);

\draw[white, line width = 5] (2,-.5) .. controls +(-.25,0) and +(.25,0) ..  (1.55,-.1);
\draw (2,-.5) .. controls +(-.25,0) and +(.25,0) ..  (1.5,0);

\draw (2,-.5) .. controls +(-.25,0) and +(1,0) ..  (0,-1.25);

\draw (2,.5) .. controls +(-.25,0) and +(1,0) ..  (0,1.25);

\draw (-2,-.5) .. controls +(.25,0) and +(-1,0) ..  (0,-1.25);

\draw (-2,.5) .. controls +(.25,0) and +(-1,0) ..  (0,1.25);

\draw (-2,-.5) .. controls +(.25,0) and +(-.25,0) ..  (-1,-.25);

\draw (-2,.5) .. controls +(.25,0) and +(-.25,0) ..  (-1,.25);

\draw[white, fill] (1,.5) rectangle (-1,-.5);

\draw[very thin] (1,.5) rectangle (-1,-.5);

\node at (0,0) {$m$};
\node at (0,-2) {$m \geq 0$};
\end{scope}

\begin{scope}[shift = {+(4.75,0)}]
\draw[very thin, red] (-2.25,0) -- (2.25,0);

\draw (2,.25) .. controls +(-.25,0) and +(.25,0) ..  (1.5,0);

\draw[white, line width = 5] (2,0) .. controls +(-.25,0) and +(.25,0) .. (1.5,.25);
\draw (2,0) .. controls +(-.25,0) and +(.25,0) .. (1.5,.25);

\draw   (1.5,-.25) .. controls +(.25,0) and +(-.25,0) ..  (2,0);

\draw[white, line width = 5] (2,-.25) .. controls +(-.25,0) and +(.25,0) ..  (1.55,-.1);
\draw (2,-.25) .. controls +(-.25,0) and +(.25,0) ..  (1.5,0);

\draw (-2,-.5) .. controls +(.25,0) and +(-.25,0) ..  (-1,.25);

\draw[white, line width = 5] (-2,.5) .. controls +(.25,0) and +(-.25,0) ..  (-1,-.25);
\draw (-2,.5) .. controls +(.25,0) and +(-.25,0) ..  (-1,-.25);

\draw (-.5,.25) .. controls +(-.25,0) and +(.25,0) ..  (-1,-.25);

\draw[white, line width = 5](-.5,-.25) .. controls +(-.25,0) and +(.25,0) ..  (-1,.25);

\draw (-.5,-.25) .. controls +(-.25,0) and +(.25,0) ..  (-1,.25);

\draw (-.5,-.25) .. controls +(1,0) and +(-.5,0) ..  (1.5,-1) .. controls +(.25,0) and +(-.25,0) ..  (2,-.75);

\draw (-.5,-.25) .. controls +(1,0) and +(-.5,0) ..  (1.5,-1);

\draw[white, line width = 5] (1.5,-.75) .. controls +(.25,0) and +(-.25,0) .. (2,-1) .. controls +(-.25,0) and +(1.5,0) ..  (0,-1.25);

\draw (1.5,-.75) .. controls +(.25,0) and +(-.25,0) .. (2,-1) .. controls +(-.25,0) and +(1.5,0) ..  (0,-1.25);

\draw (1.5,.75) .. controls +(.25,0) and +(-.25,0) .. (2,1) .. controls +(-.25,0) and +(1.5,0) ..  (0,1.25);

\draw (-2,-.5) .. controls +(.25,0) and +(-1.5,0) ..  (0,-1.25);

\draw (-2,.5) .. controls +(.25,0) and +(-1.5,0) ..  (0,1.25);

\draw[white, line width = 5] (-.5,.25) .. controls +(1,0) and +(-.5,0) ..  (1.5,1) .. controls +(.25,0) and +(-.25,0) ..  (2,.75);

\draw (-.5,.25) .. controls +(1,0) and +(-.5,0) ..  (1.5,1) .. controls +(.25,0) and +(-.25,0) ..  (2,.75);
\draw[line width = 5,fill,white] (2.15,1) .. controls +(.15,0) and +(-.15,0) .. (2.25,0) .. controls +(-.15,0) and +(.15,0) .. (2.15,-1) .. controls +(.125,0) and +(-.125,0) .. (2.25,0) .. controls +(-.125,0) and +(.125,0) .. (2.15,1);
\draw[line width = 1,very thin, fill] (2.15,1) .. controls +(.15,0) and +(-.15,0) .. (2.25,0) .. controls +(-.15,0) and +(.15,0) .. (2.15,-1) .. controls +(.125,0) and +(-.125,0) .. (2.25,0) .. controls +(-.125,0) and +(.125,0) .. (2.15,1);
\node[right] at (2.25,0.05) {\small{$\vert m \vert - 1$}};
\node[below right] at (2.25,0) {\small{crossings}};
\node at (0,-2) {$m < -1 $ and odd};

\node at (1.75,.6) {$\vdots$};
\node at (1.75,-.4) {$\vdots$};
\end{scope}
\end{tikzpicture}     \caption{Strongly invertible Legendrian representatives of maximal Thurston Bennequin number for $K_m$, with $m \neq -1$.}
    \label{f:SILtwistK}
\end{figure}
This proves the statement for twist knots. 
\end{proof}

\section{Experimental results and conjectures}\label{s:expconj}

In this section we discuss some experimental results and state three conjectures. 

In view of Proposition~\ref{p:torus-twist}, we looked for strongly invertible Legendrian knots realizing the maximal Thurston-Bennequin number in the simplest strongly invertible knot types which are neither twist knots nor of type $T(2,2n+1)$. 
Among the knots appearing in~\cite{LKA13}, which have grid number less than $10$, we could find such Legendrian representatives up to and including crossing number eight. 
On the other hand, we were unable to find strongly invertible Legendrian representatives for each knot type with grid number less than ten and crossing number nine. 
For example, according to~\cite{LKA13} we have $\overline{\tb}(9_{42}) = -3$, but we could not find a strongly invertible Legendrian knot of type $9_{42}$ with Thurston-Bennequin number bigger than $-5$. 

On the basis of our experimental evidence we formulate the following 
\begin{con}\label{conj:tbmax}
$\overline{\tb}_e(9_{42}) < \overline{\tb}(9_{42})$.
\end{con}
Recall that, given an oriented Legendrian knot $\overrightarrow{\K}\subset\R^3$, the {\em Legendrian mirror} $\mu(\overrightarrow{\K})$ is obtained from $\overrightarrow{\K}$ by a $\pi$-rotation around the $x$-axis~\cite{Et05}. 
It is easy to check that, for any choice of orientation, a strongly invertible Legendrian knot coincides with its reversed Legendrian mirror. 
The oriented Legendrian representative of $9_{42}$ with maximal Thurston-Bennequin number shown in~\cite{LKA13} is not isotopic to its reversed Legendrian mirror, so we initially wondered whether this phenomenon could be the only one obstructing the equality $\overline{\tb}_e = \overline{\tb}$. 
Then, we became aware of the knot type $m(10_{125})$, denoted by $m10n2$ in~\cite{LKA23}, for which every  maximal Legendrian representative listed in~\cite{LKA23} is equivalent to its reversed Legendrian mirror. 
We were unable to find a strongly invertible representative of $m(10_{125})$ with maximal Thurston-Bennequin invariant.  
Hence, we state the following: 
\begin{con}\label{conj:tbmaxII}
Let $\K$ be a Legendrian knot representing $m(10_{125})$ with $\tb(\K)=-6$. 
Then, $\K$ is Legendrian isotopic to $-\mu(\K)$ but not to a strongly invertible Legendrian knot. 
\end{con}
Recall that the Jones' conjecture proved in~\cite{DP13, LFM14} implies the equality
\begin{equation}\label{eq:JonesConj}
\overline{\tb}(K) + \overline{\tb}(m(K)) = -{\rm grid\ number} (K)\ .
\end{equation} 
One can introduce a natural definition of ``equivariant grid number'' for a strongly invertible knot, as the minimal grid number of a symmetric grid (with respect to the north-west-to-south-east diagonal) realizing the given knot type. 
A positive answer to Conjecture~\ref{conj:tbmaxII} would yield a countexample to the equivariant analog of Equation~\eqref{eq:JonesConj}, where $\overline{\tb}$ is replaced with $\overline{\tb}_e$ and the grid number with the equivariant grid number. 
In fact, this conclusion would follow from the fact, which we have verified, that both the grid number and the Thurston-Bennequin number of the knot type $10_{125}$ coincide with their equivariant analogs.  

Aside from the equality $\overline{\tb}_e = \overline{\tb}$, one could investigate the existence of strongly invertible oriented Legendrian knots with a given pair of values $(\tb,\rot)$. 
For example, according to~\cite{LKA23} there is an oriented Legendrian knot $\overrightarrow\K$ of type $8_5$,  Legendrian isotopic to $-\mu(\overrightarrow\K)$, with $\tb(\overrightarrow\K) = -11$ and $\rot(\overrightarrow\K) = 2$.
We could find a strongly invertible oriented Legendrian representative of $8_5$ with $(\tb ,\rot ) = (-11,0)$, but were unable to find a strongly invertible Legendrian knot isotopic to $\K$ with either orientation.  
As a result, we state the following: 
\begin{con}\label{conj:norepr}
Let $\overrightarrow\K$ be a strongly invertible oriented Legendrian knot of type $8_5$ with $\tb(\overrightarrow\K) = -11$. Then, $\rot(\overrightarrow\K) = 0$.
\end{con}
We close the paper with the description of a fortuitous discovery. 
While examining nine-crossing knots we ran into the following example, which shows that the information in~\cite{LKA13} about $m(9_{44})$ needs to be slightly revised. 
Consider the oriented Legendrian knot $\overrightarrow\K_1$ with topological type $m(9_{44})$ and maximal Thurston-Bennequin number $-3$ described in~\cite{LKA13}, where one finds the claim that $\overrightarrow\K_1$ is Legendrian isotopic to its Legendrian mirror~\footnote{In~\cite{LKA13} the Legendrian mirror is defined using $\pi$-rotation around the {\em $y$-axis}, but this definition yields an equivalent Legendrian to the one above. 
This follows from~\cite[Theorem~1.2]{Keg18} and the fact that the composition of the two rotations is a contactomorphism which preserves the orientation of the contact planes.}.
We were unable to verify the claim using {\tt KnotMatcher}~\cite{KnotMatcher}.
In fact, we are going to show that the claim leads to a contradiction with recent work of Dynnikov and Shastin~\cite{DS23}. 
Let $\K_1$ be the unoriented Legendrian knot underlying $\overrightarrow\K_1$. 
A front diagram for $\K_1$ is illustrated on the left of Figure~\ref{f:m9_44-pair}. 
\begin{figure}[ht]
\begin{minipage}[c]{5.5cm}
\begin{tikzpicture}[scale = 0.5]
\draw[very thick]
(4.0568,-1.5905) .. controls (4.5995, -1.5898) and (4.7885, 0.0557) .. (5.5674, 0.0620);

\draw[very thick]
(7.0781,-1.5905) .. controls (7.6208, -1.5898) and (8.2098, 0.6729) .. (8.9888, 0.6792);

\draw[very thick]
(10.8995,-1.5905) .. controls (11.4422, -1.5898) and (12.0313, 0.6729) .. (12.8103, 0.6792);

\draw[very thick] 
(14.7210,-1.5905) .. controls (14.1783, -1.5898) and (13.5892, 0.6729) .. (12.8103, 0.6792);

\draw[very thick]
(14.7210,-1.5905) .. controls (14.1783, -1.5911) and (13.1289, -3.8184) .. (11.9153, -3.7961);

\draw[very thick]
(8.8543,-1.4467) .. controls (9.3970, -1.4460) and (9.9861, 0.8167) .. (10.7650, 0.8230);

\draw[very thick]
(5.5007,-1.5956) .. controls (6.0433, -1.5949) and (6.2324, 0.0506) .. (7.0113, 0.0569);

\draw[very thick]
(12.6758,-1.4467) .. controls (12.1331, -1.4474) and (11.0837, -3.6746) .. (9.8701, -3.6523);

\draw[very thick]
(6.4622,-2.8838) .. controls (5.3726, -2.8736) and (4.8357, -1.5968) .. (4.0568, -1.5905);

\draw[very thick]
(10.9144,-0.5713) .. controls (10.0359, -0.5814) and (9.5036, -2.8888) .. (6.4622, -2.8838);

\draw[very thick]
(10.0141,-4.4724) .. controls (11.2633, -4.4773) and (11.8995, -2.3745) .. (12.6785, -2.3682);

\draw[color=white, line width=0.15cm] 
(6.2246,-0.4269) -- (6.3577,-0.6612);

\draw[color=white, line width=0.15cm]
(7.7477,-0.5523) -- (7.8845,-0.7796);

\draw[color=white, line width=0.15cm]
(9.7709,-0.0052) -- (9.8966,-0.2190);

\draw[color=white, line width=0.15cm] 
(8.6266,-2.2528) -- (8.7267,-2.4500);

\draw[color=white, line width=0.15cm]
(7.1179-0.218,-2.7363+0.1963) -- (7.3359+0.218,-2.9326-0.1963);

\draw[color=white, line width=0.15cm] 
(10.2473,-0.8461) -- (10.3411,-1.0121);

\draw[color=white, line width=0.15cm] 
(9.3436,-1.7555) -- (9.4691,-1.8804);

\draw[color=white, line width=0.15cm] 
(11.4580-0.0769,-0.8745+0.0853) -- (11.5349+0.0769,-0.9598-0.0853);

\draw[color=white, line width=0.15cm] 
(11.7739-0.1045,-0.2521+0.1856) -- (11.8784+0.1045,-0.4377-0.1856);

\draw[color=white, line width=0.15cm] 
(12.0726,-1.7775) -- (12.1794,-1.9437);

\draw[color=white, line width=0.15cm]
(10.7431,-3.1837) -- (10.9148,-3.3309);

\draw[color=white, line width=0.15cm] 
(11.2711,-3.5838) -- (11.4393,-3.6762);

\draw[very thick] 
(12.6758,-1.4467) .. controls (12.1331, -1.4460) and (11.5440, 0.8167) .. (10.7650, 0.8230);

\draw[very thick] 
(12.6785,-2.3682) .. controls (12.1358, -2.3675) and (11.6933, -0.5776) .. (10.9144, -0.5713);

\draw[very thick]
(8.8543,-1.4467) .. controls (9.3970, -1.4474) and (10.8049, -3.8038) .. (11.9153, -3.7961);

\draw[very thick] 
(10.8995,-1.5905) .. controls (10.3568, -1.5898) and (9.7678, 0.6729) .. (8.9888, 0.6792);

\draw[very thick]
(7.0113,0.0569) .. controls (8.0185, 0.0604) and (8.7597, -3.6600) .. (9.8701, -3.6523);

\draw[very thick]
(7.0781,-1.5905) .. controls (6.5354, -1.5898) and (6.3464, 0.0557) .. (5.5674, 0.0620);

\draw[very thick] 
(5.5007,-1.5956) .. controls (6.3486, -1.5805) and (8.3463, -4.4890) .. (10.0141, -4.4724);
\end{tikzpicture}
\end{minipage}
\hspace{1cm}
\begin{minipage}[c]{6cm}
\begin{tikzpicture}[very thick, scale = .35]
\draw[thin, red] (-2.75,0) -- (13.75,0);
\draw (0,.5) .. controls +(.5,0) and +(-.5,0) .. (2.75,3.5);
\draw (2.75,-3.5) .. controls +(.5,0) and +(-1,0) .. (6,-.5);
\draw (0,-.5) .. controls +(1.5,0) and +(-2,0) .. (6,4);
\draw (10,-.75) .. controls +(-1,0) and +(1,0) .. (6,-4);
\draw (10,3) .. controls +(-1.5,0) and +(1,0) .. (6,.5);
\draw (2,-2) .. controls +(2,0) and +(-1,0) .. (6,2);
\draw[white, line width = 4] (2,2) .. controls +(2,0) and +(-1,0) .. (6,-2);
\draw (2,2) .. controls +(2,0) and +(-1,0) .. (6,-2);
\draw (6,-2) .. controls +(1,0) and +(-1,0) .. (9,.5);
\draw[white, line width = 4] (6,2) .. controls +(1,0) and +(-1,0) .. (9,-.5);
\draw (6,2) .. controls +(1,0) and +(-1,0) .. (9,-.5);
\draw (9,-.5) .. controls +(.5,0) and +(-.25,0) .. (10,.75);
\draw[white, line width = 4] (9,.5) .. controls +(.5,0) and +(-.25,0) .. (9.85,-.5);
\draw (9,.5) .. controls +(.5,0) and +(-.25,0) .. (10,-.75);
\draw[white, line width = 4] (10,-3) .. controls +(-1.5,0) and +(1,0) .. (6,-.5);
\draw (10,-3) .. controls +(-1.5,0) and +(1,0) .. (6,-.5);
\draw[white, line width = 4] (9.85,.855) .. controls +(-1.1,0) and +(1,0) .. (6,4);
\draw (10,.75) .. controls +(-1,0) and +(1,0) .. (6,4);
\draw[white, line width = 4] (0.275,.39) .. controls +(1.6,0) and +(-2,0) .. (6,-4);
\draw (0,.5) .. controls +(1.5,0) and +(-2,0) .. (6,-4);
\draw[white, line width = 4]  (-2,0.1) .. controls +(.5,0) and +(-2,0) .. (2,2);
\draw (-2,0) .. controls +(.5,0) and +(-2,0) .. (2,2);
\draw (-2,0) .. controls +(.5,0) and +(-2,0) .. (2,-2);
\draw[white, line width = 4] (0.275,-.69) .. controls +(.55,0) and +(-.5,0) .. (2.75,-3.5);
\draw (0,-.5) .. controls +(.5,0) and +(-.5,0) .. (2.75,-3.5);
\draw[white, line width = 4] (2.75,3.5) .. controls +(.5,0) and +(-1,0) .. (6,.5);
\draw (2.75,3.5) .. controls +(.5,0) and +(-1,0) .. (6,.5);
\draw (13,0) .. controls +(-.5,0) and +(1.5,0) .. (10,3);
\draw (13,0) .. controls +(-.5,0) and +(1.5,0) .. (10,-3);
\end{tikzpicture}
\end{minipage}
 \caption{The two Legendrian knots $\K_1$ (left) and $\K_2$ (right)}
\label{f:m9_44-pair}
\end{figure}
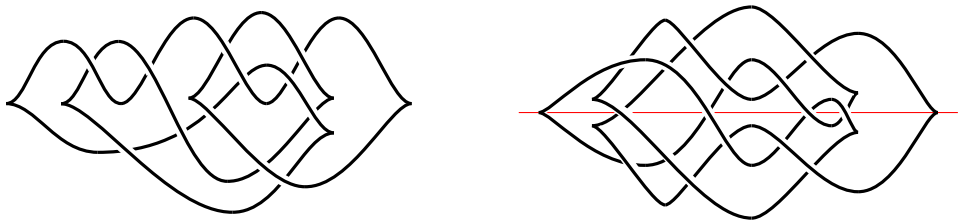
It can be verified using {\tt KnotMatcher} that $\K_1$ is Legendrian isotopic to the strongly invertible Legendrian knot $\K_2$ with transvergent front diagram shown on the right of Figure~\ref{f:m9_44-pair}. 
Thus, we can choose an orientated version $\overrightarrow\K_2$ of $\K_2$ Legendrian isotopic to $\overrightarrow\K_1$. 
Therefore we have $\overrightarrow\K_1 = \overrightarrow\K_2 = -\mu(\overrightarrow\K_2) = -\mu(\overrightarrow\K_1)$.  
But by~\cite[Proposition~7.4]{DS23} the Legendrian knot $\overrightarrow\K_1$ is not isotopic to its reverse $-\overrightarrow\K_1$.
This implies that~$\overrightarrow\K_1$ is not Legendrian isotopic to $\mu(\overrightarrow\K_1)$, contrary to the claim made in \cite{LKA13}.
Summarizing, the previous considerations show that the entry for $m(9_{44})$ in~\cite{LKA13} should be amended as follows: 
\begin{table*}[!h]
\begin{tabular}{c | cr c c c c c c}
\hline \hline
			Knot & Grid & \multirow{2}{*}{\raisebox{1em}{$(tb,r)$}} & \multirow{2}{*}{\raisebox{1em}{$L=-L$?}} & \multirow{2}{*}{\raisebox{1em}{$L=\mu(L)$?}} & \multirow{2}{*}{\raisebox{1em}{$L=-\mu(L)$?}} & $\begin{matrix} \text{Ruling} \\ \text{polynomial} \end{matrix}$ & LCH  \\
   \hline
			$m(9_{44})$ &
\raisebox{-2em}{\begin{tikzpicture}[scale = .25]
\draw[white, opacity =0 ] (0,9) -- (8,9);
\draw[red] (6,4) -- (0,4);
\draw[red] (3,7) -- (1,7);
\draw[red] (0,8) -- (2,8);
\draw[red] (7,2) -- (3,2);
\draw[red] (8,0) -- (4,0);
\draw[red] (2,6) -- (5,6);
\draw[red] (1,1) -- (6,1);
\draw[red] (4,5) -- (7,5);
\draw[red] (5,3) -- (8,3);
\draw[line width = 3, white] (0,4) -- (0,8);
\draw[-latex, red] (0,4) -- (0,8);
\draw[line width = 3, white] (1,7) -- (1,1);
\draw[-latex, red] (1,7) -- (1,1);
\draw[line width = 3, white] (2,8) -- (2,6);
\draw[-latex, red] (2,8) -- (2,6);
\draw[line width = 3, white] (3,2) -- (3,7);
\draw[-latex, red] (3,2) -- (3,7);
\draw[line width = 3, white] (4,0) -- (4,5);
\draw[-latex, red] (4,0) -- (4,5);
\draw[line width = 3, white] (5,6) -- (5,3);
\draw[-latex, red] (5,6) -- (5,3);
\draw[line width = 3, white] (6,1) -- (6,4);
\draw[-latex, red] (6,1) -- (6,4);
\draw[line width = 3, white] (7,5) -- (7,2);
\draw[-latex, red] (7,5) -- (7,2);
\draw[line width = 3, white] (8,3) -- (8,0);
\draw[-latex, red] (8,3) -- (8,0);
\end{tikzpicture}
}
& $(-3,0)$ & \ding{55} & \ding{55} & $\checkmark$  & $1$ & $t^{-1}+2t$ \\
\end{tabular}
\end{table*}

\printbibliography
\Addresses
\end{document}